\documentclass[11pt,a4paper]{amsart}
\usepackage[utf8]{inputenc}
\usepackage[T1]{fontenc}
\usepackage[english]{babel}

\usepackage{sidecap}

\usepackage{verbatim}
\usepackage{lmodern}
\usepackage{amsmath}
\usepackage{amssymb} 
\usepackage{amsthm} 
\usepackage{enumitem}
\usepackage[margin=1.2in]{geometry}
\usepackage{graphicx}
\usepackage{hyperref}
\usepackage{mathtools}
\usepackage{indentfirst}
\usepackage{tabularx}
\usepackage{bbm}
\usepackage[capitalise]{cleveref}

\usepackage{tikz} 
\usepackage{tikz-cd} 
\usetikzlibrary{arrows,calc,decorations.markings}
\usepackage{xparse}

\makeatletter
\@namedef{subjclassname@2020}{%
  \textup{2020} Mathematics Subject Classification}
\makeatother

\newcommand{\Cat}{\mathrm{Cat}}
\newcommand{\DblCat}{\mathrm{DblCat}}
\newcommand{\DblInfH}{\mathrm{DblCat}^h_\infty}
\newcommand{\DblInfV}{\mathrm{DblCat}^v_\infty}
\newcommand{\TwoCSS}{\mathrm{2CSS}}
\newcommand{\TwoCat}{2\mathrm{Cat}}
\newcommand{\Set}{\mathrm{Set}}
\newcommand{\sSet}{\mathrm{sSet}}
\newcommand{\Map}{\mathrm{Map}}

\newcommand{\op}{{\mathrm{op}}}
\newcommand{\ps}{{\mathrm{ps}}}
\newcommand{\id}{{\mathrm{id}}}

\newcommand{\threeD}{{\Delta\times \Delta\times \Delta}}
\newcommand{\twoDop}{{\Delta^{\op}\times \Delta^{\op}}}
\newcommand{\threeDop}{{(\Delta^{\op})^{\times 3}}}

\newcommand{\ND}{\mathbb N}
\newcommand{\CD}{\mathbb C}
\newcommand{\XD}{\mathbb X}
\newcommand{\CC}{\mathcal C}
\newcommand{\cC}{\mathcal C_2}
\newcommand{\cX}{\mathcal X_2}
\newcommand{\cN}{\mathcal N_2}
\newcommand{\bbH}{\mathbb H}
\newcommand{\bI}{\mathbb I}
\newcommand{\bbHsim}{\mathbb H^{\simeq}}
\newcommand{\Lsim}{L^{\simeq}}
\newcommand{\bbV}{\mathbb V}
\newcommand{\bfH}{\mathbf H}
\newcommand{\bfV}{\mathbf V}
\newcommand{\cV}{\mathcal V}
\newcommand{\bA}{\mathbb A}
\newcommand{\bB}{\mathbb B}
\newcommand{\cA}{\mathcal A}
\newcommand{\cB}{\mathcal B}
\newcommand{\ON}[1]{\widetilde{O_2({#1})}}
\newcommand{\OM}[1]{O_2^\sim ({#1})}
\newcommand{\VK}[1]{\mathbb V O_2^\sim{({#1})}}
\newcommand{\Eadj}{E_\mathrm{adj}}
\newcommand{\Reph}[1]{F^h[{#1}]}
\newcommand{\Sph}[1]{G^h[{#1}]}
\newcommand{\Repv}[1]{F^v[{#1}]}
\newcommand{\Spv}[1]{G^v[{#1}]}
\newcommand{\Reps}[1]{\Delta[{#1}]}
\newcommand{\gh}[1]{g_{#1}^h}
\newcommand{\gv}[1]{g_{#1}^v}
\newcommand{\eh}{e^h}
\newcommand{\ev}{e^v}
\newcommand{\iotah}[1]{\iota^h_{#1}}
\newcommand{\iotav}[1]{\iota^v_{#1}}
\newcommand{\iotas}[1]{\iota^s_{#1}}
\newcommand{\Nh}{N^h}
\newcommand{\Nv}{N^v}

\newcommand{\vcong}{\rotatebox{270}{$\cong$}}
\newcommand{\lsimeq}{\rotatebox{270}{$\simeq$}}
\newcommand{\rsimeq}{\rotatebox{90}{$\simeq$}}

\newenvironment{tz}{\begin{center}\begin{tikzpicture}}{\end{tikzpicture}\end{center}}
\newenvironment{tzsmall}{\begin{center}\begin{tikzpicture}[node distance=1.2cm]}{\end{tikzpicture}\end{center}}
\tikzstyle{d}=[double distance=.3ex]

\tikzset{%
node distance=1.5cm, la/.style={scale=0.8}, rr/.style={xshift=1.5cm},
space/.style={xshift=.5cm}, over/.style={auto=false,fill=white,inner sep=1.5pt, minimum size=0, outer sep=0},
    symbol/.style={%
        draw=none,
        every to/.append style={%
            edge node={node [sloped, allow upside down, auto=false]{$#1$}}},
            
    }, pro/.style={postaction={decorate,decoration={
        markings,
        mark=at position .5 with {\node at (0,0) {$\bullet$};}
      }},
      inner sep=.9ex,
      },
  n/.style={double equal sign distance, -implies}, t/.style={double distance=2.5pt, -implies, postaction={draw,-}},
}
\newcommand{\arrowdot}{
\ensuremath{\begin{tikzpicture}
\node (A) at (0,-.4) {};
\node (B) at (.4,-.4) {};
\draw[->, line width=.1ex] (0,-.6) -- (.4,-.6);
\node[shape=circle, fill=black, scale=0.35] (A) at  (.17,-.6) {};
\end{tikzpicture}
}}

\newcommand{\sq}[5]{{#1}\colon({#4} \; ^{{#2}}_{\substack{{#3}}} \; {#5})}
\newcommand{\pushout}[4]{({#1}\,\square_{{#4}}\, {#2})\,\square_{{#4}}\,{#3}}

\def\cellslide{0.5}
\def\celllength{.2cm}

\NewDocumentCommand{\cell}{ O{} O{n} O{\cellslide} O{\celllength} m m m }{
  \coordinate (mid) at ($({#5})!{#3}!({#6})$);
  \coordinate (start) at ($(mid)!{#4}!({#5})$);
  \coordinate (end) at ($(mid)!{#4}!({#6})$);
  \draw[#2] (start) to node
  [inner sep=4pt,outer sep=0,minimum size=0,#1]{{#7}} (end);
}

\newlist{rome}{enumerate}{7}
\setlist[rome]{label=(\roman*)}

\newtheorem{theorem}{Theorem}[section]
\newtheorem{theoremA}{Theorem}

\newtheorem{cor}[theorem]{Corollary}
\newtheorem{prop}[theorem]{Proposition}
\newtheorem{lemme}[theorem]{Lemma}

\theoremstyle{definition}
\newtheorem{defn}[theorem]{Definition}
\newtheorem{ex}[theorem]{Example}
\newtheorem{notation}[theorem]{Notation}
\newtheorem{descr}[theorem]{Description}

\theoremstyle{remark}
\newtheorem{rem}[theorem]{Remark}

\crefname{theorem}{Theorem}{Theorems}
\crefname{theoremA}{Theorem}{Theorems}
\crefname{cor}{Corollary}{Corollaries}
\crefname{prop}{Proposition}{Propositions}
\crefname{lemme}{Lemma}{Lemmas}

\crefname{defn}{Definition}{Definitions}
\crefname{ex}{Example}{Examples}
\crefname{notation}{Notation}{Notations}
\crefname{descr}{Description}{Descriptions}

\crefname{rem}{Remark}{Remarks}

\title{A double \texorpdfstring{$(\infty,1)$}{(infinity,1)}-categorical nerve for double categories}

\author[L.\ Moser]{Lyne Moser}
\address{Fakultät für Mathematik, Universität Regensburg, 93040 Regensburg, Germany}
\email{lyne.moser@ur.de}

\subjclass[2020]{18N10, 18N65, 55U35}
\keywords{Nerve, 2-categories, double categories, 2-fold complete Segal spaces, double $\infty$-categories}

\begin{document}

\maketitle

\begin{abstract}
We construct a nerve from double categories into double $(\infty,1)$-categories and show that it gives a right Quillen and homotopically fully faithful functor between the model structure for weakly horizontally invariant double categories and the model structure on bisimplicial spaces for double $(\infty,1)$-categories seen as double Segal objects in spaces complete in the horizontal direction. We then restrict the nerve along a homotopical horizontal embedding of $2$-categories into double categories, and show that it gives a right Quillen and homotopically fully faithful functor between Lack's model structure for $2$-categories and the model structure for $2$-fold complete Segal spaces. We further show that Lack's model structure is right-induced along this nerve from the model structure for $2$-fold complete Segal spaces.
\end{abstract}

\section{Introduction}
Higher category theory aims to study more structured objects than categories. While categories consist of objects and morphisms between objects, higher categories also have higher morphisms. In this perspective, a $2$-category is obtained by also adding $2$-mor\-phisms between the morphisms. A $2$-category can actually be seen as a category enriched in categories -- its morphisms and $2$-morphisms between any pair of objects form a category. Another type of $2$-dimensional categories is given by internal categories to categories, called \emph{double categories}. Such a structure has two types of morphisms between objects -- \emph{horizontal} and \emph{vertical} morphisms -- and its $2$-morphisms are \emph{squares}. In particular, a $2$-category~$\cA$ can be seen as a horizontal double category $\bbH\cA$ in which every vertical morphism is trivial; or equivalently, as an internal category to categories whose category of objects is discrete.

Many aspects of $2$-category theory benefit from a passage to double categories. For example, a good notion of limit for $2$-categories is that of a \emph{$2$-limit}, where the universal property is expressed by an isomorphism between hom-categories, rather than hom-sets. As clingman and the author show in \cite{clingmanMoser}, a $2$-limit cannot be characterized as a $2$-terminal object in the $2$-category of cones, but a passage to double categories allows such a characterization by results of Grandis and Par\'e \cite{Grandis,GrandisPare}. Indeed, they show that the $2$-limit of a $2$-functor $F$ is double terminal in the double category of cones over the corresponding double functor~$\bbH F$. This result also holds in the more homotopical case of \emph{bi-limits}, where the universal property is expressed by an \emph{equivalence} of hom-categories, as clingman and the author show in \cite{cM2}.

These notions of categories, $2$-categories, and double categories are often too strict to accommodate many examples that appear in nature. In the perspective of generalizing categories, an $(\infty,1)$-category is interpreted as a categorical structure that admits morphisms in all dimensions with all $k$-morphisms invertible for $k>1$, where compositions are only associative and unital up to higher invertible morphisms. Such a higher structure should be thought of as a homotopical version of a category. Similarly to the strict case, we can then interpret an $(\infty,2)$-category, as a ``category enriched in $(\infty,1)$-categories'', and a double $(\infty,1)$-categories, as an ``internal category to $(\infty,1)$-categories''. A natural expectation is that $(\infty,2)$-categories also admit a ``horizontal embedding'' into double $(\infty,1)$-categories and that $2$-categories and double categories embed into their more homotopical versions, in such a way that the following diagram commutes (maybe only up to ``homotopy''). 
\begin{tz}
\node[](1) {\{$2$-categories\}};
\node[below of=1,yshift=-.5cm](2) {\{double categories\}};
\node[right of=1,xshift=4cm](3) {\{$(\infty,2)$-categories\}};
\node[below of=3,yshift=-.5cm](4) {\{double $(\infty,1)$-categories\}};
\draw[right hook->] (1) to (3);
\draw[right hook->] (1) to node[left,la] {$\bbH$} (2);
\draw[right hook->] (2) to (4);
\draw[right hook->] (3) to (4);
\end{tz}

The existence of such a commutative diagram would show that aspects of the theory of $(\infty,2)$-categories would also benefit from a passage to double $(\infty,1)$-categories. With this idea in mind, the author, Rasekh, and Rovelli develop in \cite{MRR3} a notion of $(\infty,2)$-limits by defining a limit of an $(\infty,2)$-functor as a terminal object in the double $(\infty,1)$-category of cones over the induced ``horizontal'' double $(\infty,1)$-functor. 

To make these $\infty$-notions precise, the machinery used is often that of \emph{model categories}, introduced by Quillen in \cite{Qui}, and these $\infty$-notions are then defined as the fibrant objects of a given model structure. This is the approach we will be taking here. As a model for $(\infty,1)$-category, we consider complete Segal spaces, due to Rezk \cite{Rezk} and defined as the Segal objects in spaces such that the space of objects is equivalent to the space of equivalences, i.e., invertible morphisms up to higher morphisms. This last condition is called the \emph{completeness condition} and ensures that no extra data has been added by considering a space of objects instead of a set of objects. There are many other models of $(\infty,1)$-categories, but the choice we make here is motivated by the fact that models of $(\infty,2)$-categories and double $(\infty,1)$-categories have been developed as ``internal categories'' to complete Segal spaces, where the complete Segal space of objects is required to be discrete in the case of $(\infty,2)$-categories. More precisely, these are given by $2$-fold complete Segal spaces defined by Barwick in \cite{Bar} as the complete Segal objects in complete Segal spaces, and by double $(\infty,1)$-categories defined by Haugseng in \cite{HaugsengThesis} as the Segal objects in complete Segal spaces. Haugseng's definition of double $(\infty,1)$-categories requires the completeness condition in the vertical direction, i.e., that the space of objects is equivalent to the space of vertical equivalences. Since we want our double $(\infty,1)$-categories to be compatible with the horizontal embedding of $2$-categories into double categories, we require instead horizontal completeness, i.e., that the space of objects is equivalent to the space of horizontal equivalences. However, these two models of double $(\infty,1)$-categories are equivalent via a transpose functor. Furthermore, there are model structures $\TwoCSS$ and $\DblInfH$ on bisimplicial spaces whose fibrant objects are the $2$-fold complete Segal spaces and the horizontally complete double $(\infty,1)$-categories, respectively. We can obtain $\TwoCSS$ as localization of $\DblInfH$, and this implies that the identity functor $\id\colon \TwoCSS\to \DblInfH$ is a right Quillen functor, which we interpret as the horizontal embedding of $(\infty,2)$-categories into double $(\infty,1)$-categories. 

To define an embedding -- called \emph{nerve} -- of $2$-categories and double categories into their $\infty$-analogues, we also need model structures in this stricter setting. In \cite{Lack2Cat,LackBicat}, Lack endows the category $\TwoCat$ of $2$-categories and $2$-functors with a model structure in which the weak equivalences are the biequivalences, the trivial fibrations are the $2$-functors which are surjective on objects, full on morphisms, and fully faithful on $2$-morphisms, and all $2$-categories are fibrant. 

In the double categorical case, several model structures for double categories are constructed by Fiore, Paoli, and Pronk in \cite{FPP}, but the horizontal embedding of $2$-categories into double categories does not induce a Quillen pair between Lack’s model structure and any of these model structures. Therefore, in \cite{MSVfirst}, the author, Sarazola, and Verdugo construct a model structure on the category $\DblCat$ of double categories and double functors, obtained as a right-induced model structure from two copies of Lack's model structure on $\TwoCat$, which is such that the horizontal embedding $\bbH\colon \TwoCat\to \DblCat$ is as well-behaved as possible: it is both left and right Quillen, and homotopically fully faithful, and it preserves and reflects the whole homotopical structure. However, in this model structure, all double categories are fibrant, and the trivial fibrations are only surjective on vertical morphisms, rather than full. These are both obstructions for the nerve being right Quillen, as the cofibrations in $\DblInfH$ are the monomorphisms and the nerve of a double category is fibrant precisely when the double category is \emph{weakly horizontally invariant} (see \cref{def:weakhorinv}), as shown in \cref{thm:fibrantnerve}.

To remedy this issue, the author, Sarazola, and Verdugo construct in \cite{MSVsecond} another model structure on $\DblCat$ whose trivial fibrations are the double functors which are surjective on objects, full on horizontal \emph{and} vertical morphisms, and fully faithful on squares, and the fibrant objects are the weakly horizontally invariant double categories. The existence of this model structure was independently noticed by Campbell \cite{CampMos}. Since the horizontal double category $\bbH\cA$ associated to a $2$-category~$\cA$ is not weakly horizontally invariant in general, the horizontal embedding $\bbH\colon \TwoCat\to \DblCat$ is not right Quillen anymore. Instead, we need to consider a fibrant replacement of $\bbH$ given by a more homotopical version $\bbHsim \colon \TwoCat\to \DblCat$ of $\bbH$, which sends a $2$-category $\cA$ to the double category $\bbHsim\cA$ whose underlying horizontal $2$-category is still~$\cA$, but whose vertical morphisms are given by the adjoint equivalences of $\cA$. This gives a right Quillen and homotopically fully faithful functor $\bbHsim\colon \TwoCat\to \DblCat$, where $\DblCat$ is endowed with the model structure for weakly horizontally invariant double categories. 

In this paper, we construct a nerve functor $\ND\colon \DblCat\to \sSet^{\twoDop}$, and we show in \cref{thm:CD-ND-Quillen,thm:counit-CD-ND} that $\ND$ is a right Quillen and homotopically fully faithful functor from $\DblCat$ to $\DblInfH$.

\begin{theoremA} \label{thmA:doublenerve} 
The nerve functor $\ND\colon \DblCat\to \DblInfH$ is right Quillen, and homotopically fully faithful from the model structure on $\DblCat$ for weakly horizontally invariant double categories to the model structure on $\sSet^{\twoDop}$ for horizontally complete double $(\infty,1)$-categories. 

Moreover, the nerve $\ND\bA$ of a double category $\bA$ is fibrant if and only if $\bA$ is weakly horizontally invariant.
\end{theoremA}

We then restrict the nerve functor $\ND$ along the homotopical horizontal embedding $\bbHsim\colon \TwoCat\to \DblCat$ and show in \cref{thm:Quillen-NDH,thm:counit-NDH} that this gives a right Quillen and homotopically fully faithful functor from $\TwoCat$ to $\TwoCSS$. Furthermore, the homotopy theory of $2$-categories is completely determined from that of $2$-fold complete Segal spaces through its image under $\ND\bbHsim$, as $\TwoCat$ is right-induced from $\TwoCSS$ along $\ND\bbHsim$ as shown in \cref{thm:2catRI2css}.

\begin{theoremA} \label{thmB:2nerve}
The nerve functor $\ND\bbHsim\colon \TwoCat\to \TwoCSS$ is right Quillen, and homotopically fully faithful from Lack's model structure on $\TwoCat$ to the model structure on $\sSet^{\twoDop}$ for $2$-fold complete Segal spaces, i.e., $(\infty,2)$-categories. Furthermore, Lack's model structure on $\TwoCat$ is right-induced from $\TwoCSS$ along $\ND\bbHsim$.
\end{theoremA}

While several nerves that fully embed the homotopy theory of $2$-categories into the one of $(\infty,2)$-categories have already been constructed: into saturated $2$-precomplicial sets by Ozornova and Rovelli in \cite{OR2}, into $2$-quasi-categories by Campbell in \cite{Camp}, and into $\infty$-bicategories by Gagna, Harpaz, and Lanari in \cite{GHL}, the nerve presented in the above theorem is, to our knowledge, the first nerve to be constructed with good homotopical properties into the model of $2$-fold complete Segal spaces. In a subsequent paper \cite{MOR}, the author, Ozornova, and Rovelli demonstrate that these different nerve constructions coincide up to a change of models, establishing their equivalence at the $\infty$-categorical level.

\cref{thmA:doublenerve,thmB:2nerve} then yield a commutative diagram of right Quillen, and homotopically fully faithful functors as desired.
\begin{tz}
\node[](1) {$\TwoCat$};
\node[below of=1](2) {$\DblCat$};
\node[right of=1,rr](3) {$\TwoCSS$};
\node[below of=3](4) {$\DblInfH$};
\draw[->] (1) to node[above,la]{$\ND\bbHsim$} (3);
\draw[->] (1) to node[left,la] {$\bbHsim$} (2);
\draw[->] (2) to node[below,la] {$\ND$} (4);
\draw[->] (3) to node[right,la]{$\id$} (4);
\end{tz}

However, we were hoping to find a nerve that is compatible with the horizontal embedding functor $\bbH$, but the nerve $\ND\bbH\cA$ of a horizontal double category $\bbH\cA$ associated to a $2$-category is not generally a double $(\infty,1)$-category nor a $2$-fold complete Segal space (see \cref{rem:NHnotfibrant}). We show in \cref{thm:homotopy} that $\ND\bbHsim\cA$ gives a fibrant replacement of $\ND\bbH\cA$ in $\TwoCSS$ (or in $\DblInfH$).

\begin{theoremA}
There is a level-wise homotopy equivalence ${\ND\bbH\cA\to \ND\bbHsim\cA}$, which exhibits $\ND\bbHsim\cA$ as a fibrant replacement of $\ND\bbH\cA$ in $\TwoCSS$ (or in $\DblInfH$), for every $2$-category~$\cA$.
\end{theoremA}

In particular, it follows from this result that we have a diagram of right Quillen and homotopically fully faithful functors
\begin{tz}
\node[](1) {$\TwoCat$};
\node[below of=1](2) {$\DblCat$};
\node[right of=2,rr](5) {$\DblCat_\mathrm{whi}$};
\node[right of=5,rr](4) {$\DblInfH$};
\node[above of=4](3) {$\TwoCSS$};
\draw[->] (1) to node[above,la]{$\ND\bbHsim$} (3);
\draw[->] (1) to node[left,la] {$\bbH$} (2);
\draw[->] (5) to node[below,la] {$\id$} (2);
\draw[->] (5) to node[below,la] {$\ND$} (4);
\draw[->] (3) to node[right,la]{$\id$} (4);
\coordinate(a) at ($(2)+(2cm,0)$);
\coordinate(b) at ($(3)-(2cm,0)$);
\cell[right,la,xshift=3pt]{a}{b}{$\simeq$};
\end{tz}
filled with a natural transformation which is level-wise a weak equivalence. This gives the expected compatibility of the nerve $\ND$ with the horizontal embedding $\bbH$. 

\subsection{Outline}

In \cref{sec:prelim}, we first recall the basic terminology for $2$-categories and double categories, and describe several functors between the categories $\TwoCat$ and $\DblCat$. We then introduce notions of \emph{horizontal equivalences} and \emph{weakly horizontally invertible squares} in a double category, which allows us to define \emph{weakly horizontally invariant} double categories. In \cref{sec:MS-2Cat-DblCat}, we recall the main features of Lack's model structure on $\TwoCat$ and of the model structure of \cite{MSVsecond} on $\DblCat$. Then, in \cref{sec:MS-infinity}, we describe the model structures $\DblInfH$ and $\TwoCSS$ for horizontally complete double $(\infty,1)$-categories and $2$-fold complete Segal spaces. Finally, in \cref{sec:nerveDblCat}, we construct a nerve functor $\ND\colon \DblCat\to \DblInfH$ and show that it is right Quillen and homotopically fully faithful. By restricting $\ND$ along the homotopical horizontal embedding $\bbHsim\colon \TwoCat\to \DblCat$, we show in \cref{sec:nerve2Cat} that the nerve functor $\ND\bbHsim\colon \TwoCat\to \TwoCSS$ is also right Quillen and homotopically fully faithful. Furthermore, we prove that Lack's model structure on $\TwoCat$ is right-induced from $\TwoCSS$ along the nerve $\ND\bbHsim$. We then construct a level-wise homotopy equivalence $\ND\bbH\cA\to \ND\bbHsim\cA$ for every $2$-category $\cA$, which exhibits $\ND\bbHsim\cA$ as a fibrant replacement of~$\ND\bbH \cA$.

The aim of \cref{app:weakhorinv} is to prove some technical results about weakly horizontally invertible squares, which were recently introduced independently by the author, Sarazola, and Verdugo in \cite{MSVfirst}, and by Grandis and Par\'e in \cite{GraPar19}. In particular, we show that a horizontal pseudo-natural transformation is an equivalence if and only if each of its square components are weakly horizontally invertible squares. The aim of \cref{app:nerveinlow} is to describe the lower simplices of the nerves $\ND\bA$, $\ND\bbHsim\cA$, and $\ND\bbH\cA$ in order to give intuition for the nerve construction of a double category or a $2$-category. In particular, this allows us to better understand the difference between the nerves $\ND\bbHsim\cA$, and $\ND\bbH\cA$ and provides intuition on why the latter is not fibrant.

\subsection*{Acknowledgments}

I am grateful to my advisor, Jérôme Scherer, for the many fruitful conversations on the content of this paper and for a close reading of several drafts. I would also like to thank Martina Rovelli, Viktoriya Ozornova, Nima Rasekh, and Maru Sarazola for their helpful answers to many of my questions. In particular, Martina Rovelli suggested using \cref{thm:Quillen-loc} and Maru Sarazola suggested a nice trick for the proof of \cref{lem:whiiffvi}. I am also grateful to the anonymous referee whose insightful comments improved greatly the content of this paper.  

During the realization of this work, the author was supported by the Swiss National Science Foundation under the project P1ELP2\_188039, and by the Max Planck Institute for Mathematics. Revisions were taken while the author was a member of the Collaborative Research Centre ``SFB 1085: Higher
Invariants'' funded by the Deutsche Forschungsgemeinschaft (DFG).

\section{Preliminaries on \texorpdfstring{$2$}{2}-dimensional categories} \label{sec:prelim}

In this paper, we consider two kinds of strict $2$-dimensional categories: $2$-categories and double categories. Every $2$-category $\cA$ can be seen as a horizontal double category~$\bbH\cA$ with only trivial vertical morphisms, and this yields a functor $\bbH\colon \TwoCat\to \DblCat$ which admits both adjoints. In particular, its right adjoint extracts from a double category $\bA$ its underlying horizontal $2$-category $\bfH\bA$. The horizontal embedding $\bbH$ is however not homotopically well-behaved, and therefore we also need to consider its more homotopical version $\bbHsim\colon \TwoCat\to \DblCat$, which sends a $2$-category $\cA$ to the double category $\bbHsim\cA$ whose underlying horizontal $2$-category is still $\cA$ itself, but its vertical morphisms are given by the adjoint equivalences of $\cA$. We first recall these notions in \cref{subsec:doublecat}.

Then, in \cref{subsec:Gray}, we recall the closed symmetric monoidal structure on $\TwoCat$ given by the Gray tensor product, introduced by Gray in \cite{Gray}, which can be interpreted as a pseudo-version of the cartesian product. Similarly, the category $\DblCat$ also admits a Gray tensor product, introduced by B\"ohm in \cite{Bohm}, which restricts along $\bbH$ to a tensoring functor of $\DblCat$ over $\TwoCat$. This provides a $\TwoCat$-enrichment on $\DblCat$, whose internal homs are described more explicitly in \cref{subsec:ps-equ}. 

Finally, in \cref{subseq:equivDblCat}, we define notions of weak horizontal invertibility in a double category $\bA$ for horizontal morphisms and squares. We then introduce \emph{weakly horizontally invariant} double categories, which are the fibrant objects of the model structure on $\DblCat$ we consider. In particular, they are precisely the double categories whose nerve is fibrant.

\subsection{\texorpdfstring{$2$}{2}-categories, double categories, and their relations} \label{subsec:doublecat}

Recall that a $2$-category~$\cA$ consists of objects, morphisms $f\colon A\to B$ between objects, and $2$-morphisms $\alpha\colon f\Rightarrow g$ between parallel morphisms, together with a horizontal composition law for morphisms and $2$-morphisms along common objects, and a vertical composition law for $2$-morphisms along common morphisms, which are associative, unital, and satisfy the interchange law. A $2$-functor $F\colon \cA\to \cB$ consists of assignments on objects, on morphisms, and on $2$-morphisms which preserve the $2$-categorical structures strictly. 

\begin{notation}
We denote by $\TwoCat$ the category of $2$-categories and $2$-functors. 
\end{notation}

Since $2$-categories have not only morphisms, but also $2$-morphisms, a good notion of invertibility for a morphism in a $2$-category is given by requiring that it has an inverse up to invertible $2$-morphism, rather than strictly. 

\begin{defn}
An \textbf{equivalence} $f\colon A\xrightarrow{\simeq} B$ in a $2$-category $\cA$ is a tuple $(f,g,\eta,\epsilon)$ consisting of morphisms $f\colon A\to B$ and $g\colon B\to A$ and invertible $2$-morphisms $\eta\colon \id_A\xRightarrow{\cong} gf$ and $\epsilon\colon fg\xRightarrow{\cong}\id_B$ in $\cA$. An equivalence $(f,g,\eta,\epsilon)$ is an \textbf{adjoint equivalence} if the invertible $2$-morphisms $\eta$ and $\epsilon$ further satisfy the triangle identities. 

We often denote the whole data $(f,g,\eta,\epsilon)$ by just $f$.
\end{defn}

\begin{rem}\label{lem:promoteadjoint}
Every equivalence in a $2$-category can be promoted to an adjoint equivalence; see, for example, \cite[Lemma 2.1.12]{RiehlVerity}.
\end{rem}

We are now ready to introduce the other type of $2$-dimensional categories of interest in this paper: the \emph{double categories}. While $2$-categories are categories enriched over the category $\Cat$ of categories and functors, double categories are internal categories to $\Cat$. 

\begin{defn}
A double category $\bA$ consists of the following data: 
\begin{rome}
\item objects $A,B,\ldots$,
\item horizontal morphisms $f\colon A\to B$ with a horizontal identity $\id_A\colon A\to A$ for each object $A$,
\item vertical morphisms $u\colon A\arrowdot A'$ with a vertical identity $e_A\colon A\arrowdot A$ for each object~$A$, 
\item squares $\sq{\alpha}{f}{f'}{u}{v}$ of the form
\begin{tz}
    \node[](1) {$A$};
    \node[right of=1](2) {$B$};
    \node[below of=1](3) {$A'$};
    \node[right of=3](4) {$B'$};
    \draw[->] (1) to node[above,la] {$f$} (2);
    \draw[->] (3) to node[below,la] {$f'$} (4);
    \draw[->,pro] (1) to node[left,la] {$u$} (3);
    \draw[->,pro] (2) to node[right,la] {$v$} (4);
    
    \node[la] at ($(1)!0.5!(4)$) {$\alpha$};
\end{tz}
with a vertical identity $\sq{e_f}{f}{f}{e_A}{e_B}$ for each horizontal morphism $f\colon A\to B$ and a horizontal identity $\sq{\id_u}{\id_A}{\id_{A'}}{u}{u}$ for each vertical morphism $u\colon A\arrowdot A'$, 
\item an associative and unital horizontal composition law for horizontal morphisms, and squares along their vertical boundaries, 
\item an associative and unital vertical composition law for vertical morphisms, and squares along their horizontal boundaries,
\end{rome}
such that horizontal and vertical compositions of squares satisfy the interchange law. 
\end{defn}

\begin{defn}
A double functor $F\colon \bA\to \bB$ consists of assignments on objects, on horizontal morphisms, on vertical morphisms, and on squares, which are compatible with domains and codomains and preserve all compositions and identities strictly.
\end{defn}

\begin{notation}
We denote by $\DblCat$ the category of double categories and double functors. 
\end{notation}

In particular, a $2$-category can be seen as an internal category to $\Cat$ where the category of objects is discrete. This gives an embedding of $\TwoCat$ into $\DblCat$ which associates to a $2$-category its corresponding horizontal double category. 

\begin{defn}
We define the \textbf{horizontal embedding functor} $\bbH\colon \TwoCat\to \DblCat$. It sends a $2$-category $\cA$ to the double category $\bbH\cA$ with the same objects as $\cA$, the morphisms of $\cA$ as its horizontal morphisms, only trivial vertical morphisms, and the $2$-morphisms of $\cA$ as its squares. Compositions in $\bbH\cA$ are induced by the ones in $\cA$. A $2$-functor $F\colon \cA\to \cB$ is sent to the double functor $\bbH F\colon \bbH\cA\to \bbH\cB$ which acts as $F$ does on the corresponding data. 
\end{defn}

The functor $\bbH$ admits both adjoints. Its right adjoint is given by the following functor (see \cite[Proposition 2.5]{FPP}).

\begin{defn}
The functor $\bbH\colon \TwoCat\to \DblCat$ admits a right adjoint given by the functor $\bfH\colon \DblCat\to \TwoCat$. It sends a double category~$\bA$ to its \textbf{underlying horizontal $2$-category} $\bfH\bA$ with the same objects as $\bA$, whose morphisms are the horizontal morphisms of $\bA$, and whose $2$-morphisms $\alpha\colon f\Rightarrow f'$ are the squares in $\bA$ of the form
\begin{tz}
    \node[](A) {$A$};
    \node[right of=A](B) {$B$};
    \node[below of=A](A') {$A$};
    \node[right of=A'](B') {$B$};
\node at ($(B'.east)-(0,4pt)$) {.};
    \draw[->] (A) to node(a)[above,la] {$f$} (B);
    \draw[->] (A') to node(c)[below,la] {$f'$} (B');
    \draw[d,pro] (A) to (A');
    \draw[d,pro] (B) to (B');
    
    \node[la] at ($(A)!0.5!(B')$) {$\alpha$};
\end{tz}
\end{defn}

\begin{rem}
The functor $\bbH\colon \TwoCat\to \DblCat$ also admits a left adjoint, denoted by $L\colon \DblCat\to \TwoCat$, which sends a double category $\bA$ to a $2$-category $L\bA$ whose objects are equivalence classes of objects in $\bA$ under the following relation: two objects are identified if and only if they are related by a zig-zag of vertical morphisms. The morphisms of $L\bA$ are then generated by the horizontal morphisms of $\bA$, and the $2$-morphisms of $L\bA$ are generated by the squares of $\bA$.
\end{rem}

Since $2$-categories can also be embedded vertically into double categories, there are analogous functors for the vertical direction. However, in this paper, a $2$-category is always seen as a horizontal double category, unless specified otherwise.  

\begin{rem}
Similarly, there is a functor $\bbV\colon \TwoCat\to \DblCat$ sending a $2$-category $\cA$ to the double category $\bbV\cA$ with the same objects as $\cA$, only trivial horizontal morphisms, the morphisms of $\cA$ as its vertical morphisms, and the $2$-morphisms of $\cA$ as its squares. This functor also admits both adjoints, and its right adjoint $\bfV\colon \DblCat\to \TwoCat$ sends a double category to its \textbf{underlying vertical $2$-category}. 
\end{rem}

As we will see in \cref{sec:MS-2Cat-DblCat}, the horizontal embedding $\bbH$ is not homotopically well-behaved, and we therefore introduce another functor $\bbHsim\colon \TwoCat\to \DblCat$, which provides the correct homotopical replacement of $\bbH$.

\begin{defn}
We define the functor $\bbHsim\colon \TwoCat\to \DblCat$. It sends a $2$-category~$\cA$ to the double category $\bbHsim\cA$ with the same objects as $\cA$, whose horizontal morphisms are the morphisms of $\cA$, whose vertical morphisms are the adjoint equivalences $(u,u',\eta_u,\epsilon_u)$ of $\cA$, and whose squares 
\begin{tz}
    \node[](A) {$A$};
    \node[right of=A](B) {$B$};
    \node[below of=A](A') {$A'$};
    \node[right of=A'](B') {$B'$};
    \draw[->] (A) to node[above,la] {$f$} (B);
    \draw[->] (A') to node[below,la] {$f'$} (B');
    \draw[->] (A) to node[left,la]{$\underline{u}=(u,u',\eta_u,\epsilon_u)$} node[right,la]{$\lsimeq$} (A');
    \draw[->] (B) to node[right,la]{$\underline{v}=(v,v',\eta_v,\epsilon_v)$} node[left,la]{$\rsimeq$} (B');
    
    \cell[la,above,xshift=-.2cm]{B}{A'}{$\alpha$};
\end{tz}
are given by the $2$-morphisms $\alpha\colon vf\Rightarrow f'u$ in $\cA$. Compositions in $\bbHsim\cA$ are induced by the ones in $\cA$. A $2$-functor $F\colon \cA\to \cB$ is sent to the double functor $\bbHsim F\colon \bbHsim \cA\to \bbHsim\cB$ which acts as $F$ does on the corresponding data.
\end{defn}

The functor $\bbHsim$ is not a left adjoint, since it does not preserve colimits; see \cite[Remark 2.17]{MSVsecond}. However, it admits a left adjoint, which we describe below.

\begin{rem}
By \cite[Proposition 2.15]{MSVsecond}, the functor $\bbHsim$ admits a left adjoint, denoted by $\Lsim\colon \DblCat\to \TwoCat$. It sends a double category $\bA$ to the $2$-category $\Lsim\bA$ with the same objects as $\bA$, and whose morphisms are generated by a morphism for each horizontal morphism in $\bA$ and by an adjoint equivalence for each vertical morphism in $\bA$. Its $2$-morphisms are further generated by the squares of $\bA$. See \cite[Remark 2.16]{MSVsecond} for a precise description.
\end{rem}

\subsection{Gray tensor products and \texorpdfstring{$\TwoCat$}{2Cat}-enrichment} \label{subsec:Gray}

The category $\TwoCat$ admits a closed symmetric monoidal structure introduced by Gray in \cite{Gray}.

\begin{defn} \label{def:otimes2}
Let $\mathcal I$ and $\cA$ be $2$-categories. We denote by $[\mathcal I,\cA]_{2,\ps}$ the \textbf{pseudo-hom $2$-category} of $2$-functors $\mathcal I\to \cA$, pseudo-natural transformations, and modifications. For a definition of these notions, see \cite[Definition 4.2.1 and 4.4.1]{JohYau}. 

Then the \textbf{Gray tensor product} $\otimes_2\colon \TwoCat\times \TwoCat\to \TwoCat$ endows the category $\TwoCat$ with a closed symmetric monoidal structure with respect to these pseudo-homs. More explicitly, for all $2$-categories $\mathcal I$, $\cA$, and $\cB$, we have a bijection 
\[ \TwoCat(\mathcal I\otimes_2 \cB,\cA)\cong \TwoCat(\cB, [\mathcal I,\cA]_{2,\ps}) \]
natural in $\mathcal I$, $\cA$, and $\cB$.
\end{defn}

\begin{notation} \label{not:pushprod2}
Let $i\colon \mathcal I\to \cA$ and $i'\colon \mathcal I'\to \cA'$ be $2$-functors. We denote by $i\,\square_{\otimes_2}\, i'$ their pushout-product 
\[ i\,\square_{\otimes_2}\, i'\colon \cA\otimes_2 \mathcal I'\bigsqcup_{\mathcal I\otimes_2 \mathcal I'}\mathcal I\otimes_2 \cA'\to \cA\otimes_2 \cA'. \]
\end{notation}

Similarly, the category $\DblCat$ also admits a closed symmetric monoidal structure given by B\"ohm's Gray tensor product \cite{Bohm}, whose internal homs have horizontal (resp.~vertical) \emph{pseudo}-natural transformations as its horizontal (resp.~vertical) morphisms. These transformations consist of the same data as the horizontal (resp.~vertical) transformations of double functors with additional vertically (resp.~horizontally) invertible squares giving the pseudo-naturality conditions with respect to horizontal (resp.~vertical) morphisms.

\begin{defn} \label{def:Graytensordbl}
Let $\bI$ and $\bA$ be double categories. We define the \textbf{pseudo-hom double category} $[\bI,\bA]_\ps$ to be the double category of double functors $\bI\to \bA$, horizontal pseudo-natural transformations, vertical pseudo-natural transformations, and modifications. See \cite[\S 2.2]{Bohm} or \cite[\S 3.8]{Grandis} for more details. 

By \cite[\S 3]{Bohm}, the \textbf{Gray tensor product} $\otimes_G\colon \DblCat\times \DblCat\to \DblCat$ endows the category $\DblCat$ with a closed symmetric monoidal structure with respect to these pseudo-homs. More explicitly, for all double categories $\bI$, $\bA$, and $\bB$, we have a bijection 
\[ \DblCat(\bI\otimes_G \bB,\bA)\cong \DblCat(\bB, [\bI,\bA]_{\ps}) \]
natural in $\bI$, $\bA$, and $\bB$.
\end{defn}

In this paper, we are interested in the underlying horizontal $2$-categories of these pseudo-hom double categories. This gives a tensored and cotensored $\TwoCat$-enrichment on $\DblCat$ with tensoring functor obtained by restricting the Gray tensor product for double categories defined above along the horizontal embedding $\bbH$ in one of the variables. 

\begin{defn} \label{def:tensor}
Let $\bI$ and $\bA$ be double categories. We define the \textbf{pseudo-hom $2$-category} $\bfH[\bI,\bA]_\ps$ to be the $2$-category of double functors $\bI\to \bA$, horizontal pseudo-natural transformations, and modifications; see \cref{def:pseudohor,def:modif}.

Then the Gray tensor product $\otimes_G\colon \DblCat\times \DblCat\to \DblCat$ restricts to a \textbf{tensoring functor} 
\[ \otimes \coloneqq \DblCat\times \TwoCat \xrightarrow{\id\times \bbH}\DblCat\times \DblCat\xrightarrow{\otimes_G} \DblCat \]
with respect to these pseudo-homs. More explicitly, for every pair of double categories $\bI$ and $\bA$, and every $2$-category $\cB$, we have a bijection 
\[ \DblCat(\bI\otimes \cB,\bA)\cong \TwoCat(\cB, \bfH[\bI,\bA]_\ps) \]
natural in $\bI$, $\bA$, and $\cB$. See \cite[Proposition 7.5]{MSVfirst}.
\end{defn}

\begin{notation} \label{not:pushproddbl}
Given a double functor $I\colon \bI\to \bA$ in $\DblCat$ and a $2$-functor $i\colon \mathcal I\to \cA$ in $\TwoCat$, we denote by $I\,\square_{\otimes}\, i$ their pushout-product 
\[ I\,\square_{\otimes}\, i\colon \bA\otimes \mathcal I\bigsqcup_{\bI\otimes \mathcal I} \bI\otimes \cA\to \bA\otimes \cA. \]
\end{notation}

\subsection{Weak horizontal invertibility in a double category} \label{subseq:equivDblCat}

As for $2$-categories, a good notion of invertibility for a horizontal morphism in a double category is not given by that of an isomorphism, but rather by a weaker notion. Indeed, a double category has an underlying horizontal $2$-category which contains all horizontal morphisms, and which we can use to define the notion of \emph{horizontal equivalences}. Let us fix a double category $\bA$. 

\begin{defn} 
A horizontal morphism $f\colon A\to B$ in $\bA$ is a \textbf{horizontal equivalence} if $f$ is an equivalence in the underlying horizontal $2$-category $\bfH\bA$, i.e., if we have a tuple $(f,g,\eta,\epsilon)$ of horizontal morphisms $f\colon A\to B$ and $g\colon B\to A$ in $\bA$ and vertically invertible squares $\eta$ and $\epsilon$ in $\bA$ as depicted below. 
\begin{tz}
\node[](A) {$A$};
\node[right of=A,rr](C) {$A$};
\node[below of=A](A') {$A$};
\node[right of=A'](B') {$B$};
\node[right of=B'](C') {$A$};
\draw[->] (A') to node[below,la] {$f$} (B');
\draw[->] (B') to node[below,la] {$g$} (C');
\draw[d] (A) to (C);
\draw[d,pro] (A) to node(a)[]{} (A');
\draw[d,pro] (C) to node(b)[]{} (C');

\node[la] at ($(a)!0.45!(b)$) {$\eta$};
\node[la] at ($(a)!0.55!(b)$) {$\vcong$};

\node[right of=C,space](A) {$B$};
\node[right of=A](B) {$A$};
\node[right of=B](C) {$B$};
\node[below of=A](A') {$B$};
\node[right of=A',rr](C') {$B$};
\draw[->] (A) to node[above,la] {$g$} (B);
\draw[->] (B) to node[above,la] {$f$} (C);
\draw[d] (A') to (C');
\draw[d,pro] (A) to node(a)[]{} (A');
\draw[d,pro] (C) to node(b)[]{} (C');

\node[la] at ($(a)!0.45!(b)$) {$\epsilon$};
\node[la] at ($(a)!0.55!(b)$) {$\vcong$};
\end{tz}
 
The horizontal morphism $f\colon A\to B$ is a \textbf{horizontal adjoint equivalence} if $f$ is an adjoint equivalence in the underlying horizontal $2$-category $\bfH\bA$, i.e., if the vertically invertible squares $\eta$ and $\epsilon$ further satisfy the triangle identities which require the following pastings to be the vertical identity squares at $f$ and $g$, respectively. 
\begin{tz}
\node[](A) {$A$};
\node[right of=A,rr](C) {$A$};
\node[below of=A](A') {$A$};
\node[right of=A'](B') {$B$};
\node[right of=B'](C') {$A$};
\node[right of=C](D) {$B$};
\node[right of=C'](D') {$B$};
\draw[->] (A') to node[over,la] {$f$} (B');
\draw[->] (B') to node[over,la] {$g$} (C');
\draw[->] (C) to node[above,la] {$f$} (D);
\draw[->] (C') to node[over,la] {$f$} (D');
\draw[d] (A) to (C);
\draw[d,pro] (A) to node(a)[]{} (A');
\draw[d,pro] (C) to node(b)[]{} (C');
\draw[d,pro] (D) to (D');

\node[la] at ($(a)!0.45!(b)$) {$\eta$};
\node[la] at ($(a)!0.55!(b)$) {$\vcong$};
\node[la] at ($(C)!0.5!(D')$) {$e_f$};

\node[below of=A'](A'') {$A$};
\node[below of=B'](B'') {$B$};
\node[below of=D'](D'') {$B$};
\draw[->] (A'') to node[below,la] {$f$} (B'');
\draw[d] (B'') to (D'');
\draw[d,pro] (A') to (A'');
\draw[d,pro] (B') to node(b)[]{} (B'');
\draw[d,pro] (D') to node(c)[]{} (D'');

\node[la] at ($(A')!0.5!(B'')$) {$e_f$};
\node[la] at ($(b)!0.45!(c)$) {$\epsilon$};
\node[la] at ($(b)!0.55!(c)$) {$\vcong$};

\node[right of=D,xshift=2cm](A) {$A$};
\node[right of=A,rr](C) {$A$};
\node[below of=A](A') {$A$};
\node[right of=A'](B') {$B$};
\node[right of=B'](C') {$A$};
\node[left of=A](Z) {$B$};
\node[below of=Z](Z') {$B$};
\draw[->] (A') to node[over,la] {$f$} (B');
\draw[->] (B') to node[over,la] {$g$} (C');
\draw[->] (Z) to node[above, la] {$g$} (A);
\draw[->] (Z') to node[over, la] {$g$} (A');
\draw[d] (A) to (C);
\draw[d,pro] (Z) to (Z');
\draw[d,pro] (A) to node(a)[]{} (A');
\draw[d,pro] (C) to node(b)[]{} (C');

\node[la] at ($(Z)!0.5!(A')$) {$e_g$};
\node[la] at ($(a)!0.45!(b)$) {$\eta$};
\node[la] at ($(a)!0.55!(b)$) {$\vcong$};

\node[below of=Z'](Z'') {$B$};
\node[below of=B'](B'') {$B$};
\node[below of=C'](C'') {$A$};
\draw[d] (Z'') to (B'');
\draw[->] (B'') to node[below, la] {$g$} (C'');
\draw[d,pro] (Z') to node(z)[]{} (Z'');
\draw[d,pro] (B') to node(a)[]{} (B'');
\draw[d,pro] (C') to (C'');

\node[la] at ($(z)!0.45!(a)$) {$\epsilon$};
\node[la] at ($(z)!0.55!(a)$) {$\vcong$};
\node[la] at ($(B')!0.5!(C'')$) {$e_g$};
\end{tz}
\end{defn}

\begin{rem}
By applying \cref{lem:promoteadjoint} to the equivalences of the underlying horizontal $2$-category $\bfH\bA$, we can see that every horizontal equivalence in a double category $\bA$ can be promoted to a horizontal adjoint equivalence.
\end{rem}

Before introducing the notion of weak horizontal invertibility, we first settle the following notations. 

\begin{notation}
We denote by $[n]$ the category given by the poset $\{0<1<\ldots <n\}$, for $n\geq 0$. In other words, it is the free category on $n$ composable morphisms. In particular, the category $[0]$ is the terminal category, and the category $[1]$ is the free category $\{0\to 1\}$ on a morphism.
\end{notation}

We can extract from a double category $\bA$ a $2$-category $\cV\bA\coloneqq \bfH[\bbV[1],\bA]_\ps$ whose objects are the vertical morphisms of $\bA$, and whose morphisms are the squares of~$\bA$. By considering the equivalences in this $2$-category, we get a notion of weak horizontal invertibility for squares. Technical, useful results about weakly horizontally invertible squares are proven in \cref{app:weakhorinv}.

\begin{defn} 
A square $\sq{\alpha}{f}{f'}{u}{v}$ in $\bA$ is \textbf{weakly horizontally invertible} if $\alpha$ is an equivalence in the $2$-category $\bfH[\bbV[1],\bA]_\ps$, i.e., if we have the data of a square $\sq{\beta}{g}{g'}{u}{v}$ in $\bA$ together with vertically invertible squares $\eta$, $\eta'$, $\epsilon$, and $\epsilon'$ satisfying the following pasting equalities. 
\begin{tz}
\node[](A) {$A$};
\node[right of=A,rr](C) {$A$};
\node[below of=A](A') {$A$};
\node[right of=A'](B') {$B$};
\node[right of=B'](C') {$A$};
\draw[->] (A') to node[over,la] {$f$} (B');
\draw[->] (B') to node[over,la] {$g$} (C');
\draw[d] (A) to (C);
\draw[d,pro] (A) to node(a)[]{} (A');
\draw[d,pro] (C) to node(b)[]{} (C');

\node[la] at ($(a)!0.45!(b)$) {$\eta$};
\node[la] at ($(a)!0.55!(b)$) {$\vcong$};

\node[below of=A'](A'') {$A'$};
\node[right of=A''](B'') {$B'$};
\node[right of=B''](C'') {$A'$};
\draw[->] (A'') to node[below,la] {$f'$} (B'');
\draw[->] (B'') to node[below,la] {$g'$} (C'');
\draw[->,pro] (A') to node[left,la] {$u$} (A'');
\draw[->,pro] (B') to node[right,la] {$v$} (B'');
\draw[->,pro] (C') to node[right,la] {$u$} (C'');

\node[la] at ($(A')!0.5!(B'')$) {$\alpha$};
\node[la] at ($(B')!0.55!(C'')$) {$\beta$};

\node[right of=C,xshift=.5cm](X) {$A$};
\node[right of=X,rr](Z) {$A$};
\node[below of=X](A) {$A'$};
\node[right of=A,rr](C) {$A'$};
\node[la] at ($(C')!0.5!(A)$) {$=$};
\draw[d] (X) to (Z);
\draw[d] (A) to (C);
\draw[->,pro] (X) to node[left,la] {$u$} (A);
\draw[->,pro] (Z) to node[right,la] {$u$} (C);

\node[la] at ($(X)!0.5!(C)$) {$\id_u$};

\node[below of=A](A') {$A'$};
\node[right of=A'](B') {$B'$};
\node[right of=B'](C') {$A'$};
\draw[->] (A') to node[below,la] {$f'$} (B');
\draw[->] (B') to node[below,la] {$g'$} (C');
\draw[d,pro] (A) to node(a)[]{} (A');
\draw[d,pro] (C) to node(b)[]{} (C');

\node[la] at ($(a)!0.45!(b)$) {$\eta'$};
\node[la] at ($(a)!0.55!(b)$) {$\vcong$};
\end{tz}

\begin{tz}
\node[](A) {$B$};
\node[right of=A](B) {$A$};
\node[right of=B](C) {$B$};
\node[below of=A](A') {$B$};
\node[right of=A',rr](C') {$B$};
\draw[->] (A) to node[above,la] {$g$} (B);
\draw[->] (B) to node[above,la] {$f$} (C);
\draw[d] (A') to (C');
\draw[d,pro] (A) to node(a)[]{} (A');
\draw[d,pro] (C) to node(b)[]{} (C');

\node[la] at ($(a)!0.45!(b)$) {$\epsilon$};
\node[la] at ($(a)!0.55!(b)$) {$\vcong$};

\node[below of=A'](A'') {$B'$};
\node[below of=C'](C'') {$B'$};
\draw[d] (A'') to (C'');
\draw[->,pro] (A') to node[left,la] {$v$} (A'');
\draw[->,pro] (C') to node[right,la] {$v$} (C'');

\node[la] at ($(A')!0.5!(C'')$) {$\id_v$};

\node[right of=C,xshift=.5cm](X) {$B$};
\node[right of=X](Y) {$A$};
\node[right of=Y](Z) {$B$};
\node[below of=X](A) {$B'$};
\node[la] at ($(C')!0.5!(A)$) {$=$};
\node[right of=A](B) {$A'$};
\node[right of=B](C) {$B'$};
\draw[->] (X) to node[above,la] {$g$} (Y);
\draw[->] (Y) to node[above,la] {$f$} (Z);
\draw[->] (A) to node[over,la] {$g'$} (B);
\draw[->] (B) to node[over,la] {$f'$} (C);
\draw[->,pro] (X) to node[left,la] {$v$} (A);
\draw[->,pro] (Y) to node[right,la] {$u$} (B);
\draw[->,pro] (Z) to node[right,la] {$v$} (C);

\node[la] at ($(X)!0.5!(B)$) {$\beta$};
\node[la] at ($(Y)!0.55!(C)$) {$\alpha$};

\node[below of=A](A') {$B'$};
\node[right of=A',rr](C') {$B'$};
\draw[d] (A') to (C');
\draw[d,pro] (A) to node(a)[]{} (A');
\draw[d,pro] (C) to node(b)[]{} (C');

\node[la] at ($(a)!0.45!(b)$) {$\epsilon'$};
\node[la] at ($(a)!0.55!(b)$) {$\vcong$};
\end{tz}
Note that the data $(f,g,\eta,\epsilon)$ and $(f',g',\eta',\epsilon')$ are horizontal equivalences in $\bA$, and we call $\beta$ a \textbf{weak inverse} of $\alpha$ with respect to the horizontal equivalence data $(f,g,\eta,\epsilon)$ and $(f',g',\eta',\epsilon')$.
\end{defn}

\begin{rem}
By applying \cref{lem:promoteadjoint} to the $2$-category $\bfH[\bbV[1],\bA]_\ps$, we can see that every weakly horizontally invertible square in a double category $\bA$ can be promoted to a weakly horizontally invertible square whose horizontal equivalence data are horizontal adjoint equivalences. Indeed, a square is an adjoint equivalence in the $2$-category $\bfH[\bbV[1],\bA]_\ps$ if and only if its horizontal equivalence data are horizontal adjoint equivalences.
\end{rem}

\begin{rem}
If the horizontal equivalence data of a weakly horizontally invertible square are horizontal adjoint equivalences, we call them \emph{horizontal adjoint equivalence data}.
\end{rem}

With this terminology settled, we are now ready to introduce the fibrant double categories of the considered model structure on $\DblCat$.

\begin{defn} \label{def:weakhorinv}
A double category $\bA$ is \textbf{weakly horizontally invariant} if for every pair of horizontal equivalences $f\colon A\xrightarrow{\simeq} B$ and $f'\colon A'\xrightarrow{\simeq} B'$ and every vertical morphism $v\colon B\arrowdot B'$ in $\bA$, there is a vertical morphism $u\colon A\arrowdot A'$ together with a weakly horizontally invertible square $\alpha$ in $\bA$ as depicted below. 
\begin{tz}
    \node[](1) {$A$};
    \node[right of=1](2) {$B$};
    \node[below of=1](3) {$A'$};
    \node[right of=3](4) {$B'$};
    \draw[->] (1) to node[above,la] {$f$} node[below,la] {$\simeq$} (2);
    \draw[->] (3) to node[below,la] {$f'$} node[above,la] {$\simeq$} (4);
    \draw[->,pro] (1) to node[left,la] {$u$} (3);
    \draw[->,pro] (2) to node[right,la] {$v$} (4);
    
    \node[la] at ($(1)!0.5!(4)-(5pt,0)$) {$\alpha$};
    \node[la] at ($(1)!0.5!(4)+(5pt,0)$) {$\simeq$};
\end{tz}
\end{defn}

\section{Model structures on \texorpdfstring{$\TwoCat$}{2Cat} and \texorpdfstring{$\DblCat$}{DblCat}} \label{sec:MS-2Cat-DblCat}

Lack constructs in \cite{Lack2Cat,LackBicat} a model structure on the category $\TwoCat$ in which the trivial fibrations are the $2$-functors which are surjective on objects, full on morphisms, and fully faithful on $2$-morphisms, and every $2$-category is fibrant. In particular, the weak equivalences in this model structure are the biequivalences. Since $2$-categories can be horizontally embedded into double categories, we then expect that the category $\DblCat$ can be endowed with a model structure which is compatible with that of $2$-categories through this horizontal embedding. 

The first positive answer is given in \cite{MSVfirst}, in which we construct a model structure on~$\DblCat$ right-induced along the functor $(\bfH,\cV)\colon \DblCat\to \TwoCat\times \TwoCat$. With respect to this model structure, the horizontal embedding $\bbH\colon \TwoCat\to \DblCat$ is as well-behaved as possible: it is both left and right Quillen, and homotopically fully faithful. However, the trivial fibrations are only surjective on vertical morphisms, rather than full, and all double categories are fibrant, which are both obstructions for the nerve from double categories to double $(\infty,1)$-categories to be right Quillen. 

Therefore, in \cite{MSVsecond}, we construct another model structure on $\DblCat$ in which the trivial fibrations are the double functors which are surjective on objects, full on horizontal \emph{and} vertical morphisms, and fully faithful on squares, and the fibrant objects are given by the weakly horizontally invariant double categories, which are precisely those double categories whose nerve is fibrant (see \cref{thm:fibrantnerve}). While the horizontal embedding $\bbH$ is still left Quillen and homotopically fully faithful, it is not right Quillen anymore. Instead, its more homotopical version $\bbHsim$ fulfills this role, and actually provides a level-wise fibrant replacement of $\bbH$. 

We recall in \cref{subsec:Lack-MS} the main features of Lack's model structure on $\TwoCat$, and in \cref{subsec:MSV-MS} those of the model structure on $\DblCat$ of \cite{MSVsecond}. In particular, we characterize the cofibrations of these model structures since these descriptions will be used to prove that the left adjoint of the double $(\infty,1)$-categorical nerve is left Quillen. 

\subsection{Lack's model structure for  \texorpdfstring{$2$}{2}-categories}\label{subsec:Lack-MS}

Let us first recall the definition of a biequivalence between $2$-categories, and give generating sets of cofibrations and trivial cofibrations for Lack's model structure on $\TwoCat$. 

\begin{defn}
A $2$-functor is a \textbf{biequivalence} if it is surjective on objects up to equivalence, full on morphisms up to invertible $2$-morphism, and fully faithful on $2$-morphisms.
\end{defn}

\begin{rem}
    Through the canonical inclusion $\Cat\hookrightarrow 2\Cat$, we regard any category as a $2$-category without further specification. Moreover, we denote by $\Sigma\colon \Cat\to 2\Cat$ the suspension functor sending a category $\CC$ to the $2$-category $\Sigma \CC$ with two objects $0$ and $1$, and hom-categories 
    \[ \Sigma\CC(0,0)=\Sigma\CC(1,1)=[0],\quad \Sigma\CC(1,0)=\emptyset, \quad \text{and}\quad \Sigma\CC(0,1)=\CC. \]
\end{rem}

\begin{notation} \label{not:cofin2cat}
We denote by $\mathcal I_2$ the set containing the following $2$-functors:
\begin{rome}
\item the unique map $i_1\colon \emptyset \to [0]$,
\item the inclusion $i_2\colon [0]\sqcup [0]\to [1]$ of the two end points into the free-living morphism,
\item the inclusion $i_3\colon \partial\Sigma[1]\to \Sigma[1]$ of the two parallel morphisms into the free-living $2$-morphism, where $\partial\Sigma[1]\coloneqq \Sigma([0]\sqcup [0])$,
\item the $2$-functor $i_4\colon \Sigma [1]_2\to \Sigma[1]$ sending the two non-trivial parallel $2$-morphisms of $\Sigma [1]_2$ to the non-trivial $2$-morphism of $\Sigma[1]$, where $[1]_2\coloneqq [1]\sqcup_{[0]\sqcup [0]} [1]$ is the free category on two parallel morphisms.
\end{rome}
We denote by $\mathcal J_2$ the set containing the following $2$-functors:
\begin{rome}
\item the inclusion $j_1\colon [0]\to \Eadj$, where the $2$-category $\Eadj$ is the ``free-living adjoint equivalence'', 
\item the inclusion $j_2\colon [1]\to \Sigma I$, where the category $I=\{ x\cong y\}$ is the ``free-living isomorphism''. 
\end{rome}
\end{notation}

\begin{theorem} \label{thm:Lack-MS}
There is a cofibrantly generated model structure on $\TwoCat$, in which the weak equivalences are the biequivalences, generating sets of cofibrations and trivial cofibrations are given by the sets $\mathcal I_2$ and $\mathcal J_2$, respectively, and every $2$-category is fibrant. 

Moreover, this model structure is monoidal with respect to the Gray tensor product~$\otimes_2$.
\end{theorem}

\begin{proof}
The existence of the model structure is given in \cite[Theorem 4]{LackBicat} (which is a slightly modified version of \cite[Theorem 3.3]{Lack2Cat}). The sets of generating (trivial) cofibrations are described at the beginning of \cite[\S 3]{Lack2Cat}, and the monoidality is the content of \cite[Theorem~7.5]{Lack2Cat}.
\end{proof}

\begin{rem} \label{rem:pushprod2Cat}
In particular, the model structure on $\TwoCat$ being monoidal with respect to $\otimes_2$ means that the pushout-product $i\,\square_{\otimes_2}\,i'$ (see \cref{not:pushprod2}) of two cofibrations $i$ and $i'$ in $\TwoCat$ is a cofibration in $\TwoCat$, which is trivial if $i$ or $i'$ is a biequivalence. 
\end{rem}

The following results provide characterizations of cofibrations and of cofibrant objects in $\TwoCat$. We denote by $U\colon \TwoCat\to \Cat$ the functor sending a $2$-category to its underlying category. 

\begin{prop} \label{prop:cofin2Cat}
A $2$-functor $F\colon \cA\to \cB$ is a cofibration in $\TwoCat$ if and only if
\begin{rome}
\item it is injective on objects and faithful on morphisms, and
\item the underlying category $U\cB$ is a retract of a category obtained from the image of $U\cA$ under $UF$ by freely adjoining objects and then morphisms between objects.
\end{rome}
\end{prop}

\begin{proof}
This follows from \cite[Lemma 4.1 and Corollary 4.12]{Lack2Cat}.
\end{proof}

\begin{cor}\label{prop:cofibrantTwoCat}
A $2$-category $\cA$ is cofibrant in $\TwoCat$ if and only if its underlying category $U\cA$ is free.
\end{cor}

\begin{proof}
This is given by \cite[Theorem 4.8]{Lack2Cat}.
\end{proof}

\subsection{Model structure for weakly horizontally invariant double categories} \label{subsec:MSV-MS}

While the weak equivalences in the model structure on $\DblCat$ for weakly horizontally invariant double categories constructed in \cite{MSVsecond} do not admit an explicit description, they contain the \emph{double biequivalences}. These correspond to the weak equivalences of the model structure on $\DblCat$ of \cite{MSVfirst}, and were first introduced in \cite[Definition 3.6]{MSVfirst}. 

\begin{defn}
A double functor is a \textbf{double biequivalence} if it is surjective on objects up to horizontal equivalence, full on horizontal morphisms up to vertically invertible square with trivial vertical boundaries, surjective on vertical morphisms up to weakly horizontally invertible square, and fully faithful on squares.
\end{defn}

We now introduce a set of generating cofibrations for the model structure on $\DblCat$ of \cite{MSVsecond}.

\begin{notation} \label{not:IJ}
We denote by $\mathcal I$ the set containing the following double functors:
\begin{rome}
\item the unique map $I_1\colon \emptyset \to [0]$,
\item the inclusion $I_2\colon [0]\sqcup [0]\to \bbH[1]$,
\item the inclusion $I_3\colon [0]\sqcup [0]\to \bbV[1]$,
\item the inclusion $I_4\colon \partial\mathbb S\to \mathbb S$, where $\mathbb S\coloneqq \bbH[1]\times \bbV[1]$ is the free double category on a square, and $\partial\mathbb S$ is its sub-double category containing the boundary of the square, i.e. it is free on two horizontal morphisms and two vertical morphisms sharing some boundaries,
\item the $2$-functor $I_5\colon \mathbb S_2\to \mathbb S$ sending the two non-trivial squares of $\mathbb S_2$ to the non-trivial square of $\mathbb S$, where $\mathbb S_2$ is the free double category on two parallel squares.
\end{rome}
\end{notation}

\begin{theorem} \label{thm:WHI-MS}
There is a model structure on $\DblCat$ in which the cofibrations are generated by the set $\mathcal I$ and the fibrant objects are the weakly horizontally invariant double categories. The class of weak equivalences contains the double biequivalences.

Moreover, the model structure on $\DblCat$ is monoidal with respect to the Gray tensor product $\otimes_G$, and it is enriched over Lack's model structure on $\TwoCat$ with respect to the $\TwoCat$-enrichment $\bfH[-,-]_\ps$.
\end{theorem}

\begin{proof}
The existence of the model structure is given in \cite[Theorem 3.26]{MSVsecond}. The monoidality and enrichment are the content of \cite[Theorem 7.8 and Remark 7.9]{MSVsecond}.
\end{proof}

\begin{rem} \label{rem:pushprodDblCat}
In particular, the model structure on $\DblCat$ being enriched over $\TwoCat$ with respect to $\bfH[-,-]_\ps$ means that the pushout-product $I\,\square_{\otimes}\,i$ (see \cref{not:pushproddbl}) of a cofibration $I$ in $\DblCat$ and a cofibration $i$ in $\TwoCat$ is a cofibration in $\DblCat$, which is trivial if $I$ is a double biequivalence or $i$ is a biequivalence. 
\end{rem}

\begin{rem}
The weak equivalences of the model structure on $\DblCat$ of \cref{thm:WHI-MS} can be described as those double functors which induce a double biequivalence between fibrant replacements. 
\end{rem}

The following results state characterizations of cofibrations and cofibrant objects in $\DblCat$.

\begin{prop} \label{prop:cofinDblCat'}
A double functor $F\colon \bA\to \bB$ is a cofibration in $\DblCat$ if and only if
\begin{rome}
\item it is injective on objects and faithful on horizontal and vertical morphisms,
\item the horizontal (resp.~vertical) underlying category $U\bfH\bB$ (resp.~$U\bfV\bB$) is a retract of a category obtained from the image of $U\bfH\bA$ (resp.~$U\bfV\bA$) under $U\bfH F$ (resp.~$U\bfV F$) by freely adjoining objects and then morphisms between objects.
\end{rome}
\end{prop}

\begin{proof}
This follows from \cite[Lemma 4.1]{Lack2Cat} and \cite[Theorem 3.11]{MSVsecond}.
\end{proof}

\begin{cor}\label{prop:cofibrantDblCat'}
A double category $\bA$ is cofibrant in $\DblCat$ if and only if its underlying horizontal and vertical categories $U\bfH \bA$ and $U\bfV\bA$ are free.
\end{cor}

\begin{proof}
This is \cite[Corollary 3.13]{MSVsecond}.
\end{proof}

The horizontal embedding functor $\bbH$ is not right Quillen with respect to Lack's model structure on $\TwoCat$ and the model structure on $\DblCat$ of \cref{thm:WHI-MS} since the horizontal double category $\bbH\cA$ associated to a $2$-category $\cA$ is not always weakly horizontally invariant; see \cite[Remark 6.4]{MSVsecond}. However, its better suited homotopical version $\bbHsim$ is such a right Quillen functor and it gives a homotopically full embedding of $2$-categories into double categories. 

\begin{theorem} \label{thm:Hsim-rightQUillen}
The adjunction 
\begin{tz}
\node[](A) {$\TwoCat$};
\node[right of=A,xshift=1.2cm](B) {$\DblCat$};
\draw[->] ($(B.west)+(0,.25cm)$) to [bend right=25] node[above,la]{$\Lsim$} ($(A.east)+(0,.25cm)$);
\draw[->] ($(A.east)-(0,.25cm)$) to [bend right=25] node[below,la]{$\bbHsim$} ($(B.west)+(0,-.25cm)$);
\node[la] at ($(A.east)!0.5!(B.west)$) {$\bot$};
\end{tz}
is a Quillen pair between Lack's model structure on $\TwoCat$ and the model structure on $\DblCat$ for weakly horizontally invariant double categories. Moreover, the derived counit of this adjunction is level-wise a biequivalence in $\TwoCat$. 
\end{theorem}

\begin{proof}
This is \cite[Theorem 6.6]{MSVsecond}.
\end{proof}

\begin{rem}\label{prop:fibreplHA}
By \cite[Theorem 6.5]{MSVsecond}, the inclusion $\bbH \cA\to \bbHsim \cA$ is a double biequivalence and hence exhibits $\bbHsim \cA$ as a fibrant replacement of $\bbH \cA$ in the model structure on $\DblCat$ for weakly horizontally invariant double categories. 
\end{rem}

\section{Model structures for \texorpdfstring{$(\infty,2)$}{(infinity,2)}-categories and double \texorpdfstring{$(\infty,1)$}{(infinity,1)}-categories} \label{sec:MS-infinity}

The model for $(\infty,1)$-categories we are considering here is that of complete Segal spaces, due to Rezk \cite{Rezk}. An $(\infty,2)$-category can then be defined as a complete Segal object in complete Segal spaces; this is the notion of $2$-fold complete Segal space, due to Barwick \cite{Bar}. Haugseng then defines double $(\infty,1)$-categories as the Segal objects in complete Segal spaces in \cite{HaugsengThesis}, where the completeness condition is consequently in the vertical direction. However, the model of double $(\infty,1)$-categories we use here requires completeness in the horizontal direction instead, so that the embedding of $(\infty,2)$-categories into double $(\infty,1)$-categories is compatible with the homotopical horizontal embedding of $2$-categories into double categories after applying the nerves. Nevertheless, these two models of double $(\infty,1)$-categories are Quillen equivalent through a transpose functor.  

In \cref{subsec:doubleinf}, we introduce horizontally complete double $(\infty,1)$-categories and show that they are the fibrant objects in a model structure on bisimplicial spaces. Then, in \cref{subsec:2CSS}, we recall the definition of a $2$-fold complete Segal space and show how to obtain the model structure on bisimplicial spaces in which they are the fibrant objects as a localization of the model structure for horizontally complete double $(\infty,1)$-categories. The construction of these two model structures are inspired from constructions given by Bergner and Rezk in \cite{BR2}. In particular, the model structure for $2$-fold complete Segal spaces is precisely the model structure of \cite[Corollary~7.2]{BR2} for $n=2$ and $i=1$.

\subsection{Model structures for double \texorpdfstring{$(\infty,1)$}{(infinity,1)}-categories} \label{subsec:doubleinf}

Let us denote by $\Delta$ the simplex category and by $\sSet=\Set^{\Delta^\op}$ the category of simplicial sets. We endow the category $\sSet$ with the Quillen model structure, constructed in \cite{Qui}, and consider the Reedy or injective model structure on $\sSet^{\twoDop}$, which coincide; see, for example, \cite[Proposition~3.10]{BergnerRezk}. This allows us to describe both the (trivial) cofibrations and the fibrant objects of this model structure. 

The objects of study here are bisimplicial spaces, i.e., trisimplicial sets, and we introduce notations for the representables in each of the three copies of $\Delta^\op$. 

\begin{notation}
We denote by $\Reph{m}$, $\Repv{k}$, and $\Reps{n}$ the representable bisimplicial spaces in the first, second, and third variable, respectively. We refer to the first direction as the \emph{horizontal} direction, the second as the \emph{vertical} direction, and the third as the \emph{space} direction. We also denote by $\iotah{m}\colon \partial \Reph{m}\to \Reph{m}$, $\iotav{k}\colon \partial \Repv{k}\to \Repv{k}$, and $\iotas{n}\colon \partial \Reps{n}\to \Reps{n}$ the boundary inclusions, and by $\ell^s_{n,t}\colon \Lambda^t[n]\to \Reps{n}$ the $(n,t)$-horn inclusion in $\Reps{n}$.
\end{notation}

\begin{notation}
Given two maps $f\colon X\to Y$ and $f'\colon X'\to Y'$ in $\sSet^{\twoDop}$, we denote by $f\,\square_\times\, f'$ their pushout-product 
\[ f\,\square_\times\, f'\colon Y\times X'\bigsqcup_{X\times X'} X\times Y'\to Y\times Y'. \]
\end{notation}

\begin{rem} \label{rem:gencofReedy}
A set of generating cofibrations for the Reedy/injective model structure on $\sSet^{\twoDop}$ is given by the collection of maps
\[ \pushout{(\iotah{m}\colon \partial \Reph{m}\to \Reph{m})}{(\iotav{k}\colon \partial \Repv{k}\to \Repv{k})}{(\iotas{n}\colon \partial \Reps{n}\to \Reps{n})}{\times}\]
for $m,k,n\geq 0$, and a set of generating trivial cofibrations by the collection of maps
\[ \pushout{(\iotah{m}\colon \partial \Reph{m}\to \Reph{m})}{(\iotav{k}\colon \partial \Repv{k}\to \Repv{k})}{(\ell^s_{n,t}\colon \Lambda^t[n]\to \Reps{n})}{\times} \]
for $m,k\geq 0$, $n\geq 1$, and $0\leq t \leq n$. In particular, the cofibrations are precisely the monomorphisms. 
\end{rem}

\begin{defn} \label{rem:Reedyfibrant}
A bisimplicial space $X\colon \twoDop\to \sSet$ is \textbf{Reedy fibrant} if the map
\[ X_{m,k}\cong \Map(\Reph{m}\times \Repv{k},X)\to \Map(\partial \Reph{m}\times \Repv{k}\bigsqcup_{\partial \Reph{m}\times \partial \Repv{k}} \Reph{m}\times \partial \Repv{k},X) \]
induced by $\iotah{m}\,\square_\times\,\iotav{k}$ is a Kan fibration in $\sSet$, for all $m,k\geq 0$, where $\Map(-,-)$ denotes the mapping simplicial set in $\sSet^{\twoDop}$. In other words, this says that the bisimplicial space $X$ has the right lifting property against all generating trivial cofibrations $\pushout{\iotah{m}}{\iotav{k}}{\ell^s_{n,t}}{\times}$ of \cref{rem:gencofReedy}.
\end{defn}

We also introduce the following notation. 

\begin{notation}
We denote by $\Nh\colon \Cat\to \Set^{\threeDop}$ the discrete nerve constant in the vertical and space directions. It is given by $(\Nh\CC)_{m,k,n}=\Cat([m],\CC)$ at a category $\CC$.
\end{notation}

\begin{ex} \label{ex:NRI}
Let $I=\{ x\cong y\}\in \Cat$ be the ``free-living isomorphism''. Its discrete nerve is given by $(\Nh I)_{m,k,n}=\Cat([m],I)$. In particular, a functor $[m]\to I$ can be described as a word of $m+1$ letters in $\{x,y\}$. For example, when $m=0$, we have that $(\Nh I)_{0,k,n}=\{x,y\}$; and, when $m=1$, $(\Nh I)_{1,k,n}=\{xx,xy,yx,yy\}$ where $xx$ and $yy$ are degenerate and represent the identities at $x$ and $y$, and $xy$ and $yx$ represent the two inverse morphisms between $x$ and $y$. In particular, for $X\in \sSet^{\twoDop}$ and $k\geq 0$ such that $X_{-,k}\in \sSet^{\Delta^{\op}}$ is a Segal space, then 
\[ \Map(\Nh I\times \Repv{k}, X)\cong (X_{1,k})^\mathrm{heq} \]
is the space of \emph{homotopy equivalences} in $X_{1,k}$, as described in \cite[\S 5.7]{Rezk}.
\end{ex}

We now present the $\infty$-version of double categories of use in this paper. 

\begin{notation}
For $k\geq 0$, we write 
\[ \gv{k}\colon \Spv{k}\coloneqq \Repv{1}\sqcup_{\Repv{0}}\ldots \sqcup_{\Repv{0}} \Repv{1}\longrightarrow \Repv{k} \]
for the spine inclusion of $\Repv{k}$ induced by the maps $\{i,i+1\}\colon [1]\to [k]$ of $\Delta$, for all $0\leq i\leq k-1$. Similarly, for $m\geq 0$, we write
\[ \gh{m}\colon \Sph{m}\coloneqq \Reph{1}\sqcup_{\Reph{0}}\ldots \sqcup_{\Reph{0}} \Reph{1}\longrightarrow \Reph{m}, \]
for the spine inclusion of $\Reph{m}$ induced by the maps $\{j,j+1\}\colon [1]\to [m]$ of $\Delta$, for all $0\leq j\leq m-1$. Finally, we write
\[ \eh\colon \Reph{0}\to \Nh I \]
for the inclusion induced by the functor $x\colon [0]\to I=\{ x\cong y\}$, where $I$ is the ``free-living isomorphism''.
\end{notation}

\begin{defn} \label{def:doubleinfinitycat}
A \textbf{horizontally complete double $(\infty,1)$-category} is a bisimplicial space $X\colon \twoDop\to \sSet$ such that 
\begin{rome}
\item $X$ is Reedy/injective fibrant,
\item $X_{m,-}\colon \Delta^\op\to \sSet$ is a Segal space, for all $m\geq 0$, i.e., the Segal maps 
\[ \Map(\Reph{m}\times \Repv{k},X)\cong X_{m,k}\stackrel{\simeq}{\longrightarrow} X_{m,1}\times_{X_{m,0}}\ldots\times_{X_{m,0}} X_{m,1}\cong \Map(\Reph{m}\times \Spv{k},X)  \]
induced by $\id_{\Reph{m}}\times \gv{k}$ are weak equivalences in $\sSet$, for all $m,k\geq 0$,
\item $X_{-,k}\colon \Delta^\op\to \sSet$ is a Segal space, for all $k\geq 0$, i.e., the Segal maps
\[ \Map(\Reph{m}\times \Repv{k}, X)\cong X_{m,k}\stackrel{\simeq}{\longrightarrow} X_{1,k} \times_{X_{0,k}}\ldots\times_{X_{0,k}} X_{1,k}\cong \Map(\Sph{m}\times \Repv{k}, X)  \]
induced by $\gh{m}\times \id_{\Repv{k}}$ are weak equivalences in $\sSet$, for all $m\geq 0$, 
\item the Segal space $X_{-,k}\colon \Delta^\op\to \sSet$ is complete, i.e., the map \[ \Map(\Nh I\times \Repv{k},X)\cong (X_{1,k})^\mathrm{heq}\stackrel{\simeq}{\longrightarrow} X_{0,k}\cong \Map(\Reph{0}\times \Repv{k},X) \] 
induced by $\eh\times \id_{\Repv{k}}$ is a weak equivalence in $\sSet$, for all $k\geq 0$. 
\end{rome}
\end{defn}

We obtain a model structure on $\sSet^{\twoDop}$ for horizontally complete double $(\infty,1)$-category by localizing the Reedy/injective model structure with respect to monomorphisms, i.e., cofibrations, with respect to which being local corresponds precisely to satisfying conditions (ii) and (iii) of the above definition. 

\begin{theorem} \label{thm:MSdoubleInf}
There is a model structure on $\sSet^{\twoDop}$, denoted by $\DblInfH$, obtained as a left Bousfield localization of the Reedy/injective model structure in which the fibrant objects are precisely the horizontally complete double $(\infty,1)$-categories. 
\end{theorem}

\begin{proof}
We localize the Reedy/injective model structure with respect to the cofibrations
\begin{itemize}
\item $\id_{\Reph{m}}\times \gv{k}\colon \Reph{m}\times \Spv{k}\to \Reph{m}\times \Repv{k}$, for all $m,k\geq 0$,
\item $\gh{m}\times \id_{\Repv{k}}\colon \Sph{m}\times \Repv{k}\to \Reph{m}\times \Repv{k}$, for all $m,k\geq 0$,
\item $\eh\times \id_{\Repv{k}}\colon \Repv{k}\cong \Reph{0}\times \Repv{k}\to \Nh I\times \Repv{k}$, for all $k\geq 0$.
\end{itemize}
The existence of this model structure is given by \cite[Theorem 4.1.1]{Hirschhorn}. Moreover, a Reedy/injective fibrant bisimplicial set is local with respect to this collection of maps if and only if it is a horizontally complete double $(\infty,1)$-category.
\end{proof}

\begin{rem}
We could also have defined a notion of double $(\infty,1)$-category, where the completeness is in the vertical direction. These correspond to the Segal objects in complete Segal spaces defined by Haugseng in \cite[Definition 2.2.2.1]{HaugsengThesis}. Let us denote by $\DblInfV$ the model structure for these vertically complete double $(\infty,1)$-category. Then the functor 
\[ t\colon \twoDop \to \twoDop, \ \ ([m],[k])\mapsto ([k],[m]) \]
swapping the two copies of $\Delta^\op$ induces a functor $t^*\colon \sSet^{\twoDop}\to \sSet^{\twoDop}$, and we get a Quillen equivalence
\begin{tz}
\node[](A) {$\DblInfH$};
\node[right of=A,xshift=1.8cm](B) {$\DblInfV$};
\draw[->] ($(B.west)+(0,.25cm)$) to [bend right=25] node[above,la]{$t^*$} ($(A.east)+(0,.25cm)$);
\draw[->] ($(A.east)-(0,.25cm)$) to [bend right=25] node[below,la]{$t^*$} ($(B.west)+(0,-.25cm)$);
\node[la] at ($(A.east)!0.5!(B.west)$) {$\bot$};
\end{tz}
between the two model structures for double $(\infty,1)$-categories. This functor $t^*$ can be thought of as a \emph{transpose functor}. 
\end{rem}

\subsection{Model structure for \texorpdfstring{$2$}{2}-fold complete Segal spaces} \label{subsec:2CSS}

We now recall the definition of a $2$-fold complete Segal space. 

\begin{notation}
We denote by $\Nv\colon \Cat\to \Set^{\threeDop}$ the discrete nerve constant in the horizontal and space directions. It is given by $(\Nv\CC)_{m,k,n}=\Cat([k],\CC)$ at a category $\CC$.
\end{notation}

\begin{notation}
We write $\ev\colon \Repv{0}\to \Nv I$ for the inclusion induced by the functor $x\colon [0]\to I=\{ x\cong y\}$, where $I$ is the ``free-living isomorphism'', and $c^v_k\colon \Repv{0}\to \Repv{k}$ for the inclusion induced by the map $0\colon [0]\to [k]$ of $\Delta$, for $k\geq 0$.
\end{notation}

\begin{defn}
A \textbf{$2$-fold complete Segal space} (or \textbf{$(\infty,2)$-category}) is a bisimplicial space $X\colon \twoDop\to \sSet$ such that 
\begin{rome}
\item $X$ is Reedy/injective fibrant,
\item $X_{m,-}\colon \Delta^\op\to \sSet$ is a complete Segal space, for all $m\geq 0$, i.e., we have the Segal condition as in \cref{def:doubleinfinitycat} (ii), and the map
\[ \Map(\Reph{m}\times \Nv I,X)\cong (X_{m,1})^\mathrm{heq}\stackrel{\simeq}{\longrightarrow} X_{m,0}\cong \Map(\Reph{m}\times \Repv{0},X) \]
induced by $\id_{\Reph{m}}\times \ev$ is a weak equivalence in $\sSet$, for all $m\geq 0$,
\item $X_{-,k}\colon \Delta^\op\to \sSet$ is a complete Segal space, for every $k\geq 0$, 
\item $X_{0,-}\colon \Delta^\op\to \sSet$ is essentially constant, for all $k\geq 0$, i.e., the map 
\[ \Map(\Repv{k}, X)\cong X_{0,k}\stackrel{\simeq}{\longrightarrow} X_{0,0}\cong \Map(\Repv{0},X) \]
induced by $c^v_k$ is a weak equivalence in $\sSet$, for all $k\geq 0$.
\end{rome}
\end{defn}
 
We obtain a model structure for $2$-fold complete Segal spaces as a left Bousfield localization of the model structure for horizontally complete double $(\infty,1)$-categories. 

\begin{theorem} \label{thm:MS2CSS}
There is a model structure on $\sSet^{\twoDop}$, denoted by $\TwoCSS$, obtained as a left Bousfield localization of the model structure $\DblInfH$ for horizontally complete double categories in which the fibrant objects are precisely the $2$-fold complete Segal spaces, i.e., the $(\infty,2)$-categories. 
\end{theorem}

\begin{proof}
We localize the model structure $\DblInfH$ with respect to the cofibrations 
\begin{itemize}
    \item $\id_{\Reph{m}}\times \ev\colon \Reph{m}\cong \Reph{m}\times \Repv{0}\to \Reph{m}\times \Nv I$, for all $m\geq 0$, 
    \item $c^v_k\colon \Repv{0}\to \Repv{k}$, for all $k\geq 0$.
\end{itemize}
The existence of this model structure is given by \cite[Theorem 4.1.1]{Hirschhorn}. Moreover, a horizontally complete double $(\infty,1)$-category is local with respect to this collection of maps if and only if it is a $2$-fold complete Segal space. 
\end{proof}

The following result is obtained as a direct consequence of the fact that $\TwoCSS$ is a localization of $\DblInfH$, and tells us that the identity functor $\id\colon \TwoCSS\to \DblInfH$ can be interpreted as the $\infty$-version of the horizontal embedding. 

\begin{cor}
The identity adjunction on $\sSet^{\twoDop}$ is a Quillen pair
\begin{tz}
\node[](A) {$\TwoCSS$};
\node[right of=A,xshift=1.8cm](B) {$\DblInfH$};
\node at ($(B.east)-(0,4pt)$) {.};
\draw[->] ($(B.west)+(0,.25cm)$) to [bend right=25] node[above,la]{$\id$} ($(A.east)+(0,.25cm)$);
\draw[->] ($(A.east)-(0,.25cm)$) to [bend right=25] node[below,la]{$\id$} ($(B.west)+(0,-.25cm)$);
\node[la] at ($(A.east)!0.5!(B.west)$) {$\bot$};
\end{tz}
Moreover, the derived counit is level-wise a weak equivalence. In particular, this gives a homotopically full embedding of $\TwoCSS$ into $\DblInfH$. 
\end{cor}

\section{Nerve of double categories} \label{sec:nerveDblCat}

We now give the construction of a nerve functor from double categories to bisimplicial spaces. In \cref{subsec:defnerve}, we define the nerve and its left adjoint, and  in \cref{subsec:ND-Quillen}, we show that they form a Quillen pair between the model structure on $\DblCat$ for weakly horizontally invariant double categories and the model structure $\DblInfH$ for horizontally complete double $(\infty,1)$-categories. Once this fact is established, we prove in \cref{subsec:ND-homotopicff} that the nerve functor is homotopically fully faithful, by showing that the derived counit of the adjunction is level-wise a weak equivalence in $\DblCat$. Finally, in \cref{subsec:fibrantnerve}, we show that the nerve of a double category is almost fibrant; namely, it satisfies all conditions of a horizontally complete double $(\infty,1)$-category except for the Reedy/injective fibrancy condition in the vertical direction. We show that the latter condition is satisfied by the nerve if and only if the double category considered is weakly horizontally invariant.

\subsection{Definition of the nerve} \label{subsec:defnerve}

To define the nerve we make use of truncated versions of the $n$-orientals $O(n)$, introduced by Street in \cite{Street}. More precisely: 

\begin{defn}
For $n\geq 0$, we define the \textbf{$2$-truncated $n$-oriental} $O_2(n)$ to be the $2$-category described by the following data.
\begin{rome}
\item Its set of objects is given by $\{0,\ldots,n\}$, 
\item For $0\leq x,x'\leq n$, its hom-category $O_2(n)(x,x')$ is given by the poset 
\begin{align*}
    O_2(n)(x,x') =  \begin{cases} 
    \{ I\subseteq [x,x']\mid x,x'\in I\} & \text{if } x'\leq x, \\
    \emptyset & \text{if } x>x',
    \end{cases} 
\end{align*}
where $[x,x']=\{ y\in \{0,\ldots,n\}\mid x\leq y\leq x'\}$. 
\end{rome}

We also define $\OM{n}$ as the $2$-category obtained from $O_2(n)$ by formally inverting every $2$-morphism, and we define $\ON{n}$ as the $2$-category obtained from $\OM{n}$ by formally making every morphism into an adjoint equivalence. The $2$-categories $\OM{n}$ and $\ON{n}$ can be obtained as the following pushouts, respectively.
\begin{tz}
\node[](1) {$\underset{0\leq x<x'<x''\leq n}{\bigsqcup} \Sigma[1]$};
\node[right of=1,xshift=1.5cm,yshift=5pt](2) {$O_2(n)$};
\node[below of=1](3) {$\underset{0\leq x<x'<x''\leq n}{\bigsqcup} \Sigma I$};
\node[below of=2](4) {$\OM{n}$};
\draw[->] ($(1.east)+(0,5pt)$) to (2);
\draw[->] ($(3.east)+(0,5pt)$) to (4);
\draw[->] (1) to (3);
\draw[->] (2) to (4);

\node at ($(4)-(15pt,-15pt)$) {$\ulcorner$};

\node[right of=2,xshift=2cm,yshift=-5pt](1) {$\underset{0\leq x<x'\leq n}{\bigsqcup} [1]$};
\node[right of=1,xshift=1.5cm,yshift=5pt](2) {$\OM{n}$};
\node[below of=1](3) {$\underset{0\leq x<x'\leq n}{\bigsqcup} \Eadj$};
\node[below of=2](4) {$\ON{n}$};
\draw[->] ($(1.east)+(0,5pt)$) to (2);
\draw[->] ($(3.east)+(0,5pt)$) to (4);
\draw[->] (1) to (3);
\draw[->] (2) to (4);

\node at ($(4)-(15pt,-15pt)$) {$\ulcorner$};
\end{tz}
\end{defn}

In order to have a better sense of what these $2$-categories look like, we describe the lower cases.

\begin{ex} \label{ex:O2N}
For $n=0$, the $2$-categories $O_2(0)$, $\OM{0}$, and $\ON{0}$ are all given by the terminal ($2$-)category $[0]$. 

For $n=1$, the $2$-categories $O_2(1)$ and $\OM{1}$ are both given by the free ($2$-)category $[1]$ on a morphism, while the $2$-category $\ON{1}$ is the ``free-living adjoint equivalence'' $\Eadj$. 

For $n=2$, the $2$-categories $O_2(2)$, $\OM{2}$, and $\ON{2}$ are generated, respectively, by the following data,
\begin{tz}
\node[](1) {$0$};
\node[above right of=1](2) {$1$};
\node[below right of=2](3) {$2$};
\draw[->] (1) to (2);
\draw[->] (2) to (3);
\draw[->] (1) to (3);
\coordinate(a) at ($(1)!0.5!(3)$);
\cell[la,right][n][0.37]{a}{2}{};

\node[right of=3,xshift=.5cm](1) {$0$};
\node[above right of=1](2) {$1$};
\node[below right of=2](3) {$2$};
\draw[->] (1) to (2);
\draw[->] (2) to (3);
\draw[->] (1) to (3);
\coordinate(a) at ($(1)!0.5!(3)$);
\cell[la,right][n][0.37]{a}{2}{$\cong$};

\node[right of=3,xshift=.5cm](1) {$0$};
\node[above right of=1](2) {$1$};
\node[below right of=2](3) {$2$};
\draw[->] (1) to node[la,left]{$\simeq$} (2);
\draw[->] (2) to node[la,right]{$\simeq$} (3);
\draw[->] (1) to node[la,below]{$\simeq$} (3);
\coordinate(a) at ($(1)!0.5!(3)$);
\cell[la,right][n][0.37]{a}{2}{$\cong$};
\end{tz}
where $\xrightarrow{\simeq}$ denotes the data of an adjoint equivalence. 

For $n=3$, the $2$-category $O_2(3)$ is generated by the following data
\begin{tz}
\node[](1) {$0$};
\node[above of=1](2) {$1$};
\node[right of=2](3) {$2$};
\node[below of=3](4) {$3$};
\draw[->] (1) to (2);
\draw[->] (2) to (3);
\draw[->] (3) to (4);
\draw[->] (1) to (4);
\draw[->] (1) to (3);
\coordinate(a) at ($(1)!0.5!(3)$);
\cell[la,right,yshift=7pt][n][0.4]{a}{2}{};
\coordinate(a) at ($(1)!0.7!(4)$);
\coordinate(b) at ($(a)+(0,1cm)$);
\cell[la,left][n][0.37]{a}{b}{};

\node[right of=4](1) {$0$};
\node[la] at ($(3)!0.5!(1)$) {$=$};
\node[above of=1](2) {$1$};
\node[right of=2](3) {$2$};
\node[below of=3](4) {$3$};
\draw[->] (1) to (2);
\draw[->] (2) to (3);
\draw[->] (3) to (4);
\draw[->] (1) to (4);
\draw[->] (2) to (4);
\coordinate(a) at ($(2)!0.5!(4)$);
\cell[la,right,yshift=-7pt][n][0.37]{a}{3}{};
\coordinate(a) at ($(1)!0.3!(4)$);
\coordinate(b) at ($(a)+(0,1cm)$);
\cell[la,right][n][0.4]{a}{b}{};
\end{tz}
and the $2$-category $\OM{3}$ is generated by the corresponding $2$-category with all $2$-mor\-phisms invertible, while the $2$-category $\ON{3}$ is generated by the corresponding $2$-category with all morphisms being adjoint equivalences and all $2$-morphisms being invertible. 
\end{ex}

The nerve functor is then defined as the right adjoint of the left Kan extension of the following tricosimplicial double category along the Yoneda embedding. Recall the tensoring functor $\otimes\colon \DblCat\times \TwoCat\to \DblCat$ introduced in \cref{def:tensor}.

\begin{defn}
We define the tricosimplicial double category
\[  \XD\colon \threeD\to \DblCat, \quad
    ([m],[k],[n]) \mapsto \XD_{m,k,n} \coloneqq (\VK{k}\otimes \OM{m})\otimes \ON{n}, \]
where the cosimplicial maps are induced by the ones of the cosimplicial objects
 \[ \Delta \to \DblCat, \quad [k] \mapsto \VK{k}, \hspace{1.2cm} \Delta \to \TwoCat, \quad [m]\mapsto \OM{m}, \ \ \text{and} \ \ [n]\mapsto \ON{n}. \]
\end{defn}

\begin{prop}
The tricosimplicial double category $\XD$ induces an adjunction
\begin{tz}
\node[](1) {$\threeD$};
\node[below of=1](2) {$\Set^{\threeDop}$};
\node[right of=1,rr](3) {$\DblCat$};
\node at ($(3.east)-(0,4pt)$) {,};
\draw[->] (1) to node[above,la]{$\XD$} (3);
\draw[->] (1) to (2);
\draw[->] (2.north east) to node(c)[left,la,yshift=5pt] {$\CD$} (3.south west);
\coordinate(a) at ($(3.south west)+(.4cm,0)$);
\draw[->,bend left=40] (a) to node(n)[right,la,yshift=-5pt] {$\ND$} (2.east);
\node[la] at ($(c)!0.5!(n)$) {$\rotatebox{220}{$\top$}$};
\end{tz}
where $\CD$ is the left Kan extension of $\XD$ along the Yoneda embedding, and we have that
\[ (\ND\bA)_{m,k,n}\cong \DblCat((\VK{k}\otimes \OM{m})\otimes \ON{n},\bA), \]
for all $\bA\in \DblCat$ and all $m,k,n\geq 0$,
\end{prop}

\begin{proof}
This is a direct application of \cite[Theorem 1.1.10]{Cisinski}.
\end{proof}

\begin{rem}
As expected from a nerve construction, the $0$-simplices of the simplicial set~$(\ND\bA)_{0,0}$ are given by the objects of $\bA$, the ones of $(\ND\bA)_{1,0}$ by the horizontal morphisms of $\bA$, the ones of $(\ND\bA)_{0,1}$ by the vertical morphisms of $\bA$, and the ones of $(\ND\bA)_{1,1}$ by the squares of $\bA$. These can therefore be thought of as the \emph{spaces of objects, horizontal morphisms, vertical morphisms, and squares}. For a description of the $1$- and $2$-simplices of these simplicial sets, we refer the reader to \cref{subsec:descr-ND}. For $m\geq 2$ or $k\geq 2$, the simplicial sets $(\ND\bA)_{m,k}$ witness ``compositions'' in~$\bA$ of the above data.
\end{rem}

\begin{rem} \label{rem:CD-representables}
Since $\CD$ is the left Kan extension of $\XD$ along the Yoneda embedding, it is given on representables by $\CD(\Repv{k}\times \Reph{m}\times \Reps{n})=\XD_{m,k,n}$. In particular, we have that 
\[ \CD(\Repv{k})=\VK{k}, \quad \CD(\Reph{m})=\bbH\OM{m} \ \ \text{and} \ \ \CD(\Reps{n})=\bbH\ON{n}. \]
\end{rem}

We also introduce a functor $\overline{\CD}$, which takes values in $2$-categories and coincides with $\CD$ in the horizontal and space directions.

\begin{notation}
We denote by $\overline{\XD}\colon \threeD\to \TwoCat$ the tricosimplicial $2$-category given by $\overline{\XD}_{m,k,n}\coloneqq \OM{m}\otimes_2\ON{n}$, and by $\overline{\CD}\colon \Set^{\threeDop}\to \TwoCat$ the left Kan extension of $\overline{\XD}$ along the Yoneda embedding, where $\otimes_2\colon \TwoCat\times \TwoCat\to \TwoCat$ is the Gray tensor product; see \cref{def:otimes2}. 
\end{notation}

\begin{rem} \label{rem:overlineCD}
Note that $\XD_{m,0,n}= \bbH\overline{\XD}_{m,0,n}$. Therefore, if $X\in \Set^{\threeDop}$ is constant in the vertical direction, then $\CD X=\bbH \overline{\CD} X$. In particular, we have that $\CD(\Reph{m})=\bbH\overline{\CD}(\Reph{m})$ and $\CD(\Reps{n})=\bbH\overline{\CD}(\Reps{n})$, where $\overline{\CD}(\Reph{m})=\OM{m}$ and $\overline{\CD}(\Reps{n})=\ON{n}$.
\end{rem}

Finally, we comment on why we choose to define the nerve in this specific way instead of using the more direct inclusion of double categories into bisimplicial spaces. 

\begin{rem}
    Using simplices instead of their fattening given by the orientals, one can define a nerve $N\colon \DblCat\to \Set^{\threeDop}$given at a double category $\bA$ and $m,k,n\geq 0$ by
    \[ (N\bA)_{m,k,n}\cong \DblCat(\bbH[m]\times \bbV[k]\times \bbH I[n], \bA), \]
    where $I[n]$ is the contractible groupoid on $n+1$ points. While this defines a right adjoint at the point-set level, it will not have the required homotopical properties: it does not define a right Quillen functor from the model structure on $\DblCat$ for weakly horizontally invariant double categories to the model structure on $\sSet^{\twoDop}$ for horizontally complete double $(\infty,1)$-categories.
\end{rem}

\subsection{The nerve \texorpdfstring{$\ND$}{N} is right Quillen} \label{subsec:ND-Quillen}

We now want to prove that the adjunction $\CD\dashv \ND$ is a Quillen pair between $\DblCat$ and $\DblInfH$. To prove this result, we make use of the following theorem. 

\begin{theorem} \label{thm:Quillen-loc}
Let $\mathcal M$ and $\mathcal N$ be model categories and suppose that 
\begin{tz}
\node[](A) {$\mathcal N$};
\node[right of=A,xshift=.5cm](B) {$\mathcal M$};
\draw[->] ($(B.west)+(0,.25cm)$) to [bend right=25] node[above,la]{$F$} ($(A.east)+(0,.25cm)$);
\draw[->] ($(A.east)-(0,.25cm)$) to [bend right=25] node[below,la]{$U$} ($(B.west)+(0,-.25cm)$);
\node[la] at ($(A.east)!0.5!(B.west)$) {$\bot$};
\end{tz}
is a Quillen pair. Let $\CC$ be a set of cofibrations in $\mathcal M$ such that the left Bousfield localization $L_\CC \mathcal M$ of $\mathcal M$ with respect to $\CC$ exists. If $F$ sends every morphism in~$\CC$ to a weak equivalence in $\mathcal N$, then the adjunction 
\begin{tz}
\node[](A) {$\mathcal N$};
\node[right of=A,xshift=.8cm](B) {$L_\CC \mathcal M$};
\draw[->] ($(B.west)+(0,.25cm)$) to [bend right=25] node[above,la]{$F$} ($(A.east)+(0,.25cm)$);
\draw[->] ($(A.east)-(0,.25cm)$) to [bend right=25] node[below,la]{$U$} ($(B.west)+(0,-.25cm)$);
\node[la] at ($(A.east)!0.5!(B.west)$) {$\bot$};
\end{tz}
is also a Quillen pair. 
\end{theorem}

\begin{proof}
This is a direct consequence of \cite[Theorem 3.3.20]{Hirschhorn}, since the localization of $\mathcal N$ with respect to maps in $F\CC$ is $\mathcal N$ itself as maps in $F\CC$ are already weak equivalences in~$\mathcal N$. 
\end{proof}

To apply this theorem, we first show that $\CD\dashv \ND$ is a Quillen pair between the model structure on $\DblCat$ and the Reedy/injective model structure on $\sSet^{\twoDop}$. 

\begin{prop} \label{prop:CD-ND-Quillen-Reedy}
The adjunction 
\begin{tz}
\node[](A) {$\DblCat$};
\node[right of=A,xshift=1.8cm](B) {$\sSet^{\twoDop}$};
\draw[->] ($(B.west)+(0,.25cm)$) to [bend right=25] node[above,la]{$\CD$} ($(A.east)+(0,.25cm)$);
\draw[->] ($(A.east)-(0,.25cm)$) to [bend right=25] node[below,la]{$\ND$} ($(B.west)+(0,-.25cm)$);
\node[la] at ($(A.east)!0.5!(B.west)$) {$\bot$};
\end{tz}
is a Quillen pair between the model structure on $\DblCat$ of \cref{thm:WHI-MS} and the Reedy/injective model structure on $\sSet^{\twoDop}$.
\end{prop}

\begin{proof}
It is enough to show that $\CD$ sends generating cofibrations and generating trivial cofibrations in $\sSet^{\twoDop}$ to cofibrations and trivial cofibrations in $\DblCat$, respectively. Recall from \cref{rem:gencofReedy} that generating cofibrations and generating trivial cofibrations are given by pushout-product of maps $\pushout{\iotav{k}}{\iotah{m}}{\iotas{n}}{\times}$ and $\pushout{\iotav{k}}{\iotah{m}}{\ell^s_{n,t}}{\times}$, respectively. Note that the map $\iotav{k}$ is constant in the horizontal and space directions, the map $\iotah{m}$ is constant in the vertical and space directions, and the maps $\iotas{n}$ and $\ell^s_{n,t}$ are constant in the horizontal and vertical directions. Therefore, since the functor $\CD$ preserves colimits and by \cref{rem:overlineCD}, we have that
\[ \CD(\pushout{\iotav{k}}{\iotah{m}}{\iotas{n}}{\times})\cong \pushout{\CD \iotav{k}}{\CD \iotah{m}}{\CD \iotas{n}}{\otimes_\mathrm{G}}\cong \pushout{\CD \iotav{k}}{\overline{\CD} \iotah{m}}{\overline{\CD} \iotas{n}}{\otimes}, \]
and similarly for $\ell^s_{n,t}$ in place of $\iotas{n}$. Since the model structure $\DblCat$ is enriched over $\TwoCat$, pushout-products of cofibrations with respect to $\otimes$ are cofibrations, which are trivial if one of the morphisms involved is a weak equivalence, by \cref{rem:pushprodDblCat}. Therefore, it is enough to show that $\CD \iotav{k}$ is a cofibration in $\DblCat$, for all $k\geq 0$, that $\overline{\CD} \iotah{m}$ and $\overline{\CD} \iotas{n}$ are cofibrations in $\TwoCat$, for all $m,n\geq 0$, and that $\overline{\CD} \ell^s_{n,t}$ is a trivial cofibration in $\TwoCat$, for all $n\geq 1$, $0\leq t\leq n$. These statements are verified in \cref{lem:iotaF,lem:iotaRDelta,lem:ellDelta} below. 
\end{proof}

To prove that the boundary and horn inclusions mentioned above are sent to cofibrations in $\TwoCat$ and $\DblCat$, we introduce the following definitions of the boundary and  $(n,t)$-horn of $O_2(n)$, which are used to describe the images under $\CD$ of the boundary and horn inclusions. 

\begin{defn} \label{def:boundary}
For $n\geq 0$, we define the \textbf{boundary $2$-category} $\partial O_2(n)$ as the coequalizer in $\TwoCat$
\begin{tz}
\node[](1) {$\underset{0\leq i<j\leq n}{\bigsqcup} O_2(n-2)$};
\node[right of=1,xshift=2.5cm](2) {$\underset{0\leq i\leq n}{\bigsqcup} O_2(n-1)$}; 
\node[right of=2,xshift=1.5cm,yshift=4pt](3) {$\partial O_2(n)$}; 
\node at ($(3.east)-(0,4pt)$) {,};
\draw[->] ($(1.east)+(0,8pt)$) to ($(2.west)+(0,8pt)$);
\draw[->] ($(1.east)+(0,2pt)$) to ($(2.west)+(0,2pt)$);
\draw[->] ($(2.east)+(0,5pt)$) to ($(3.west)+(0,1pt)$);
\end{tz}
where the maps in the $(i,j)$-copy are induced by the cosimplicial identities ${d^i d^j=d^{j-1} d^i}$, where $d^r\colon O_2(n-2) \to O_2(n-1)$ and $d^s\colon O_2(n-1)\to O_2(n)$ denote the face maps for $r=i,j$ and $s=i,j-1$. In particular, there is an inclusion $\partial O_2(n)\to O_2(n)$ induced by the face maps $d^i\colon O_2(n-1)\to O_2(n)$ for $0\leq i\leq n$. More explicitly, these $2$-categories are given by the following: 
\begin{itemize}
    \item for $n=0$, $\partial O_2(0)=\emptyset$ with $\partial O_2(0)=\emptyset\to O_2(0)=[0]$ given by the unique map,
    \item for $n=1$, $\partial O_2(1)=[0]\sqcup [0]$ with $\partial O_2(1)=[0]\sqcup [0]\to O_2(1)=[1]$ given by including the two copies of $[0]$ as the two endpoints of the morphism in $[1]$, 
    \item for $n=2$, $\partial O_2(2)$ is the sub-$2$-category of $O_2(2)$ where the $2$-morphism is missing and the inclusion $\partial O_2(2)\to O_2(2)$ is given by the following $2$-functor,
\begin{tz}
\node[](1) {$0$};
\node[above right of=1](2) {$1$};
\node[below right of=2](3) {$2$};
\draw[->] (1) to (2);
\draw[->] (2) to (3);
\draw[->] (1) to (3);

\node[right of=3,xshift=.5cm](1) {$0$};
\node at ($(3)!0.5!(1)+(0,.5cm)$) {$\longrightarrow$};
\node[above right of=1](2) {$1$};
\node[below right of=2](3) {$2$};
\draw[->] (1) to (2);
\draw[->] (2) to (3);
\draw[->] (1) to (3);
\coordinate(a) at ($(1)!0.5!(3)$);
\cell[la,right][n][0.37]{a}{2}{};
\end{tz}
\item for $n=3$, $\partial O_2(3)$ is the sub-$2$-category of $O_2(3)$ where only the equality between the two pasting diagrams in $O_2(3)$ -- as depicted in \cref{ex:O2N} -- is missing,
\item for $n\geq 4$, $\partial O_2(n)=O_2(n)$. 
\end{itemize}

Similarly, we define the boundary $2$-categories $\partial \OM{n}$ and $\partial \ON{n}$. 
\end{defn}

\begin{defn} \label{def:horn}
For $n\geq 1$ and $0\leq t\leq n$, we define the \textbf{$(n,t)$-horn $2$-category} $\Lambda^t O_2(n)$ as the coequalizer in $\TwoCat$
\begin{tz}
\node[](1) {$\underset{\substack{0\leq i<j\leq n\\
i\neq t, j\neq t}}{\bigsqcup} O_2(n-2)$};
\node[right of=1,xshift=2.5cm](2) {$\underset{\substack{0\leq i\leq n\\ i\neq t}}{\bigsqcup} O_2(n-1)$}; 
\node[right of=2,xshift=1.5cm,yshift=8pt](3) {$\Lambda^t O_2(n)$}; 
\node at ($(3.east)-(0,4pt)$) {,};
\draw[->] ($(1.east)+(0,12pt)$) to ($(2.west)+(0,12pt)$);
\draw[->] ($(1.east)+(0,6pt)$) to ($(2.west)+(0,6pt)$);
\draw[->] ($(2.east)+(0,9pt)$) to ($(3.west)+(0,1pt)$);
\end{tz}
where the maps in the $(i,j)$-copy are induced by the cosimplicial identities ${d^i d^j=d^{j-1} d^i}$, where $d^r\colon O_2(n-2) \to O_2(n-1)$ and $d^s\colon O_2(n-1)\to O_2(n)$ denote the face maps for $r=i,j$ and $s=i,j-1$. In particular, there is an inclusion $\Lambda^t O_2(n)\to O_2(n)$ induced by the face maps $d^i\colon O_2(n-1)\to O_2(n)$ for $0\leq i\leq n$, $i\neq t$. More explicitly, these $2$-categories are given by the following: 
\begin{itemize}
    \item for $n=1$, $\Lambda^t O_2(1)=[0]$ with $\Lambda^t O_2(1)=[0]\to O_2(1)=[1]$ given by the inclusion of $[0]$ at the source of the morphism in $[1]$ if $t=1$ and at the target if $t=0$, 
    \item for $n=2$, $\Lambda^2 O_2(2)$, $\Lambda^1 O_2(2)$, and $\Lambda^0 O_2(2)$ are generated, respectively, by the following data
\begin{tz}
\node[](1) {$0$};
\node[above right of=1](2) {$1$};
\node[below right of=2](3) {$2$};
\draw[->] (2) to (3);
\draw[->] (1) to (3);

\node[right of=3,xshift=.5cm](1) {$0$};
\node[above right of=1](2) {$1$};
\node[below right of=2](3) {$2$};
\draw[->] (1) to (2);
\draw[->] (2) to (3);

\node[right of=3,xshift=.5cm](1) {$0$};
\node[above right of=1](2) {$1$};
\node[below right of=2](3) {$2$};
\draw[->] (1) to (2);
\draw[->] (1) to (3);
\end{tz}
with the obvious inclusions into $O_2(2)$, 
\item for $n=3$ and $0\leq t\leq 3$, $\Lambda^t O_2(3)$ is the sub-$2$-category where the equality between the two pasting diagrams in $O_2(3)$ and the $2$-morphism opposite to the object~$t$ are missing. For example, when $t=0$, the inclusion $\Lambda^0 O_2(3)\to O_2(3)$ is given by the following. 
\begin{tz}
\node[](1) {$0$};
\node[above of=1](2) {$1$};
\node[right of=2](3) {$2$};
\node[below of=3](4) {$3$};
\draw[->] (1) to (2);
\draw[->] (2) to (3);
\draw[->] (3) to (4);
\draw[->] (1) to (4);
\draw[->] (1) to (3);
\coordinate(a) at ($(1)!0.5!(3)$);
\cell[la,right,yshift=7pt][n][0.4]{a}{2}{};
\coordinate(a) at ($(1)!0.7!(4)$);
\coordinate(b) at ($(a)+(0,1cm)$);
\cell[la,left][n][0.37]{a}{b}{};

\node[right of=4](1) {$0$};
\node[above of=1](2) {$1$};
\node[right of=2](3) {$2$};
\node[below of=3](4) {$3$};
\draw[->] (1) to (2);
\draw[->] (2) to (3);
\draw[->] (3) to (4);
\draw[->] (1) to (4);
\draw[->] (2) to (4);
\coordinate(a) at ($(1)!0.3!(4)$);
\coordinate(b) at ($(a)+(0,1cm)$);
\cell[la,right][n][0.4]{a}{b}{};

\node[right of=4,xshift=.5cm](1) {$0$};
\node at ($(3)!0.5!(1)$) {$\longrightarrow$};
\node[above of=1](2) {$1$};
\node[right of=2](3) {$2$};
\node[below of=3](4) {$3$};
\draw[->] (1) to (2);
\draw[->] (2) to (3);
\draw[->] (3) to (4);
\draw[->] (1) to (4);
\draw[->] (1) to (3);
\coordinate(a) at ($(1)!0.5!(3)$);
\cell[la,right,yshift=7pt][n][0.4]{a}{2}{};
\coordinate(a) at ($(1)!0.7!(4)$);
\coordinate(b) at ($(a)+(0,1cm)$);
\cell[la,left][n][0.37]{a}{b}{};

\node[right of=4](1) {$0$};
\node[la] at ($(3)!0.5!(1)$) {$=$};
\node[above of=1](2) {$1$};
\node[right of=2](3) {$2$};
\node[below of=3](4) {$3$};
\draw[->] (1) to (2);
\draw[->] (2) to (3);
\draw[->] (3) to (4);
\draw[->] (1) to (4);
\draw[->] (2) to (4);
\coordinate(a) at ($(2)!0.5!(4)$);
\cell[la,right,yshift=-7pt][n][0.37]{a}{3}{};
\coordinate(a) at ($(1)!0.3!(4)$);
\coordinate(b) at ($(a)+(0,1cm)$);
\cell[la,right][n][0.4]{a}{b}{};
\end{tz}
\item for $n\geq 4$ and $0\leq t\leq n$, $\Lambda^t O_2(n)= O_2(n)$.
\end{itemize}

Similarly, we define the $(n,t)$-horn $2$-categories $\Lambda^t \OM{n}$ and $\Lambda^t \ON{n}$. 
\end{defn}

We are now ready to prove the promised lemmas which complete the proof of \cref{prop:CD-ND-Quillen-Reedy}. 

\begin{lemme} \label{lem:iotaF}
For all $k\geq 0$, the double functor $\CD(\iotav{k})\colon \CD(\partial \Repv{k})\to \CD(\Repv{k})$ is a cofibration in $\DblCat$.
\end{lemme}

\begin{proof}
The boundary $\partial \Repv{k}$ of the representable $\Repv{k}$ can be computed as the following coequalizer in $\sSet^{\twoDop}$ 
\begin{tz}
\node[](1) {$\underset{0\leq i<j\leq k}{\bigsqcup} F[k-2]$};
\node[right of=1,xshift=2cm](2) {$\underset{0\leq i\leq k}{\bigsqcup} F[k-1]$}; 
\node[right of=2,xshift=1cm,yshift=4pt](3) {$\partial \Repv{k}$}; 
\node at ($(3.east)-(0,4pt)$) {,};
\draw[->] ($(1.east)+(0,8pt)$) to ($(2.west)+(0,8pt)$);
\draw[->] ($(1.east)+(0,2pt)$) to ($(2.west)+(0,2pt)$);
\draw[->] ($(2.east)+(0,5pt)$) to ($(3.west)+(0,1pt)$);
\end{tz}
where the maps in the $(i,j)$-copy are induced by the cosimplicial identities ${d^i d^j=d^{j-1} d^i}$. By construction of $\partial \OM{k}$ (see \cref{def:boundary}), by \cref{rem:CD-representables}, and since $\CD$ preserves colimits, we find that
\[ \CD(\partial \Repv{k})=\bbV \partial \OM{k} \ \ \text{and} \ \ \CD(\Repv{k})=\VK{k}, \]
for all $k\geq 0$. Therefore, the double functors $\CD(\iotav{k})$ are given by 
\begin{itemize}
    \item for $k=0$, the generating cofibration $I_1\colon \emptyset\to [0]$,
    \item for $k=1$, the generating cofibration $I_3\colon [0]\sqcup [0]\to \bbV[1]$,
    \item for $k=2$, the inclusion 
\begin{tzsmall}
\node[](1) {$0$};
\node[below of=1](2) {$1$};
\node[below of=2](3) {$2$};
\node[left of=1](4) {$0$};
\node[left of=3](5) {$2$};
\draw[->,pro] (1) to (2);
\draw[->,pro] (2) to (3);
\draw[->,pro] (4) to (5);
\draw[d] (1) to (4);
\draw[d] (3) to (5);

\node[right of=1,xshift=2cm](1) {$0$};
\node at ($(5)!0.5!(1)$) {$\longrightarrow$};
\node[below of=1](2) {$1$};
\node[below of=2](3) {$2$};
\node[left of=1](4) {$0$};
\node[left of=3](5) {$2$};
\node at ($(3.east)-(0,4pt)$) {,};
\draw[->,pro] (1) to (2);
\draw[->,pro] (2) to (3);
\draw[->,pro] (4) to (5);
\draw[d] (1) to (4);
\draw[d] (3) to (5);
\node[la] at ($(1)!0.5!(5)$) {$\cong$};
\end{tzsmall}
which is a cofibration by \cref{prop:cofinDblCat'} since it is the identity on underlying horizontal and vertical categories, 
 \item for $k=3$, the inclusion $\bbV\partial \OM{3}\to \VK{3}$, which is a cofibration by \cref{prop:cofinDblCat'} since it is the identity on underlying horizontal and vertical categories, 
 \item for $k\geq 4$, the identity.
\end{itemize}
This shows that the double functor $\CD(\iotav{k})$ is a cofibration in $\DblCat$, for all $k\geq 0$.
\end{proof}

\begin{lemme} \label{lem:iotaRDelta}
For all $m,n\geq 0$, the $2$-functors $\overline{\CD}(\iotah{m})\colon \overline{\CD}(\partial \Reph{m})\to \overline{\CD}(\Reph{m})$ and  $\overline{\CD}(\iotas{n})\colon \overline{\CD}(\partial \Reps{n})\to \overline{\CD}(\Reps{n})$ are cofibrations in $\TwoCat$.
\end{lemme}

\begin{proof}
We first prove the statement for $\overline{\CD}(\iotah{m})$. As in the proof of \cref{lem:iotaF} and by \cref{rem:overlineCD}, we find that 
\[ \overline{\CD}(\partial \Reph{m})=\partial \OM{m} \ \ \text{and} \ \ \overline{\CD}(\Reph{m})=\OM{m}, \]
for all $m\geq 0$. Therefore, the $2$-functors $\overline{\CD}(\iotah{m})$ are given by 
\begin{itemize}
    \item for $m=0$, the generating cofibration $i_1\colon \emptyset\to [0]$,
    \item for $m=1$, the generating cofibration $i_2\colon [0]\sqcup [0]\to [1]$,
    \item for $m=2$, the inclusion $\partial \OM{2}\to \OM{2}$,
 which is a cofibration by \cref{prop:cofin2Cat} since it is the identity on underlying categories,
    \item for $m=3$, the inclusion $\partial \OM{3}\to \OM{3}$, which is a cofibration by \cref{prop:cofin2Cat} since it is the identity on underlying categories,
    \item for $m\geq 4$, the identity. 
\end{itemize}
Therefore, the $2$-functor $\overline{\CD}(\iotah{m})$ is a cofibration in $\TwoCat$, for all $m\geq 0$.

We now prove the statement for $\overline{\CD}(\iotas{n})$. As above, we find that 
\[ \overline{\CD}(\partial \Reps{n})=\partial \ON{n} \ \ \text{and} \ \ \overline{\CD}(\Reps{n})=\ON{n}, \]
for all $n\geq 0$. Therefore the $2$-functors $\overline{\CD}(\iotas{n})\colon \partial \ON{n}\to \ON{n}$ can be described as the $2$-functors $\overline{\CD}(\iotah{m})$ above, but where all the morphisms of the $2$-categories in play are adjoint equivalences. In particular, the $2$-functor $\overline{\CD}(\iotas{n})$ is also a cofibration in $\TwoCat$, for all $n\geq 0$.
\end{proof}

\begin{lemme} \label{lem:ellDelta}
For all $n\geq 1$ and $0\leq t\leq n$, the $2$-functor $\overline{\CD}(\ell^s_{n,t})\colon \overline{\CD}(\Lambda^t[n])\to \overline{\CD}(\Reps{n})$ is a trivial cofibration in $\TwoCat$. 
\end{lemme}

\begin{proof}
The $(n,t)$-horn $\Lambda^t[n]$ of the representable $\Reps{n}$ can be computed as the following coequalizer in $\sSet^{\twoDop}$ 
\begin{tz}
\node[](1) {$\underset{\substack{0\leq i<j\leq n\\
i\neq t, j\neq t}}{\bigsqcup} \Delta[n-2]$};
\node[right of=1,xshift=2cm](2) {$\underset{\substack{0\leq i\leq n\\ i\neq t}}{\bigsqcup} \Delta[n-1]$}; 
\node[right of=2,xshift=1cm,yshift=8pt](3) {$\Lambda^t [n]$}; 
\node at ($(3.east)-(0,4pt)$) {,};
\draw[->] ($(1.east)+(0,12pt)$) to ($(2.west)+(0,12pt)$);
\draw[->] ($(1.east)+(0,6pt)$) to ($(2.west)+(0,6pt)$);
\draw[->] ($(2.east)+(0,9pt)$) to ($(3.west)+(0,1pt)$);
\end{tz}
where the maps in the $(i,j)$-copy are induced by the cosimplicial identities ${d^i d^j=d^{j-1} d^i}$. By construction of $\Lambda^t \ON{n}$ (see \cref{def:horn}), by \cref{rem:overlineCD}, and since $\overline{\CD}$ preserves colimits, we find that
\[ \overline{\CD}(\Lambda^t[n])=\Lambda^t \ON{n} \ \ \text{and} \ \ \overline{\CD}(\Reps{n})=\ON{n}, \]
for all $n\geq 1$ and $0\leq t\leq n$. Therefore, the $2$-functors $\overline{\CD}(\ell^s_{n,t})\colon \Lambda^t \ON{n}\to \ON{n}$ are given by
\begin{itemize}
    \item for $n=1$ and $0\leq t\leq 1$, the generating trivial cofibration $j_1\colon [0]\to \ON{1}=\Eadj$, including $[0]$ as one of the two end points, 
    \item for $n=2$ and $0\leq t\leq 2$, the inclusion $\Lambda^t \ON{2}\to \ON{2}$, which is a cofibration by \cref{prop:cofin2Cat} since it is given by adding two morphisms $x\to y$ and $y\to x$ freely between objects $x<y\in \{0,1,2\}\setminus \{t\}$ on underlying categories. Moreover, it is a biequivalence, since it is bijective on objects, essentially full on morphisms, and fully faithful on $2$-morphisms, where essential fullness on morphisms can be shown using the fact that all the morphisms are adjoint equivalences. 
\item for $n=3$ and $0\leq t \leq 3$, the inclusion $\Lambda^t \ON{3}\to \ON{3}$, which is a cofibration by \cref{prop:cofin2Cat} since it is the identity on underlying categories. Moreover, it is a biequivalence, since it is bijective on objects and morphisms, and it is fully faithful on $2$-morphisms, where fully faithfulness follows from the fact that there is a unique invertible $2$-morphism filling the triangle of the missing invertible $2$-morphism and it is given by the obvious composite of the three other invertible $2$-morphisms. 
\item for $n\geq 4$ and $0\leq t\leq n$, the identity. 
\end{itemize}
Therefore, the $2$-functor $\overline{\CD}(\ell^s_{n,t})$ is a trivial cofibration in $\TwoCat$, for all $n\geq 1$, $0\leq t \leq n$. 
\end{proof}

We now show that the nerve functor is right Quillen from $\DblCat$ to $\DblInfH$. 

\begin{theorem} \label{thm:CD-ND-Quillen}
The adjunction 
\begin{tz}
\node[](A) {$\DblCat$};
\node[right of=A,xshift=1.8cm](B) {$\DblInfH$};
\draw[->] ($(B.west)+(0,.25cm)$) to [bend right=25] node[above,la]{$\CD$} ($(A.east)+(0,.25cm)$);
\draw[->] ($(A.east)-(0,.25cm)$) to [bend right=25] node[below,la]{$\ND$} ($(B.west)+(0,-.25cm)$);
\node[la] at ($(A.east)!0.5!(B.west)$) {$\bot$};
\end{tz}
is a Quillen pair between the model structure on $\DblCat$ for weakly horizontally invariant double categories and the model structure on $\sSet^{\twoDop}$ for horizontally complete double $(\infty,1)$-categories.
\end{theorem}

\begin{proof}
By \cref{thm:Quillen-loc} and \cref{prop:CD-ND-Quillen-Reedy}, it is enough to show that the cofibrations $\gv{k}\times \id_{\Reph{m}}$, $\id_{\Repv{k}}\times \gh{m}$, and $\id_{\Repv{k}}\times \eh$, with respect to which we localize the Reedy/injective model structure on $\sSet^{\twoDop}$ in order to obtain the model structure $\DblInfH$ of \cref{thm:MSdoubleInf}, are sent by $\CD$ to weak equivalences in $\DblCat$. By definition of $\CD$ and by \cref{rem:overlineCD}, we have that 
\[ \CD(\gv{k}\times \id_{\Reph{m}})\cong \CD(\gv{k})\otimes \id_{\overline{\CD} \Reph{m}} =\CD(\gv{k})\,\square_{\otimes}\,(\emptyset\to \overline{\CD} \Reph{m}),
\]
and similarly that
\[ \CD(\id_{\Repv{k}}\times \gh{m}) \cong (\emptyset\to \CD \Repv{k})\,\square_{\otimes}\,\overline{\CD}(\gh{m}), \quad \CD(\id_{\Repv{k}}\times \eh) \cong (\emptyset\to \CD \Repv{k})\,\square_{\otimes}\, \overline{\CD}(\eh). \]
Since $\CD$ is left Quillen from the Reedy/injective model structure on $\sSet^{\twoDop}$ in which every object is cofibrant, the unique maps $\emptyset\to \overline{\CD} \Reph{m}$ and $\emptyset\to \CD \Repv{k}$ are cofibrations in $\TwoCat$ and $\DblCat$, respectively. Moreover, the maps $\CD(\gv{k})$, $\overline{\CD}(\gh{m})$ and $\overline{\CD}(\eh)$ are cofibrations in $\DblCat$ and $\TwoCat$, since they are images of monomorphisms in $\sSet^{\twoDop}$. As the model structure on $\DblCat$ is $\TwoCat$-enriched, it is enough to show that $\CD(\gv{k})$ is a weak equivalence in $\DblCat$ and that $\overline{\CD}(\gh{m})$, and $\overline{\CD}(\eh)$ are biequivalences by \cref{rem:pushprodDblCat}. These statements are the content of \cref{lem:IF,lem:IR}, respectively.
\end{proof}

The following two lemmas complete the proof of \cref{thm:CD-ND-Quillen}. 

\begin{lemme} \label{lem:IF}
For all $k\geq 0$, the double functor $\CD(\gv{k})\colon \CD(\Spv{k})\to \CD(\Repv{k})$ is a double biequivalence in $\DblCat$. In particular, it is a weak equivalence in $\DblCat$.
\end{lemme}

\begin{proof}
Since $\CD$ preserve colimits and $[k]=[1]\sqcup_{[0]}\ldots \sqcup_{[0]} [1]$, we have that 
\[ \CD(\Spv{k})=\bbV[k] \ \ \text{and} \ \ \CD(\Repv{k})=\VK{k}, \]
for all $k\geq 0$. First note that, when $k=0,1$, the double functor $\CD(\gv{k})$ is an identity. For $k\geq 2$, let us give an example. When $k=2$, the double functor $\CD(\gv{2})$ is given by the inclusion
\begin{tzsmall}
\node[](1) {$0$};
\node[below of=1](2) {$1$};
\node[below of=2](3') {$2$};
\draw[->,pro] (1) to (2);
\draw[->,pro] (2) to (3');

\node[right of=1,xshift=2cm](1) {$0$};
\node[below of=1](2) {$1$};
\node[below of=2](3) {$2$};
\node at ($(3.east)-(0,4pt)$) {.};
\node[left of=1](4) {$0$};
\node at ($(3')!0.5!(4)$) {$\longrightarrow$};
\node[left of=3](5) {$2$};
\draw[->,pro] (1) to (2);
\draw[->,pro] (2) to (3);
\draw[->,pro] (4) to (5);
\draw[d] (1) to (4);
\draw[d] (3) to (5);
\node[la] at ($(1)!0.5!(5)$) {$\cong$};
\end{tzsmall}
Having this example in mind, we can see that, for all $k\geq 0$, $\CD(\gv{k})\colon\bbV[k]\to \VK{k}$ is the identity on objects and horizontal morphisms, and it is fully faithful on squares, since all squares in $\bbV[k]$ are trivial. The double functor $\CD(\gv{k})$ is also injective on vertical morphisms. Moreover, since every vertical morphism $i\arrowdot j$ in $\VK{k}$ is related by a horizontally invertible square to the composite $i\arrowdot i+1\arrowdot \ldots \arrowdot j$, then $\CD(\gv{k})$ is essentially full on vertical morphisms. This shows that the double functor $\CD(\gv{k})$ is a double biequivalence and hence a weak equivalence in $\DblCat$, for all $k\geq 0$.
\end{proof}

\begin{lemme} \label{lem:IR}
For all $m\geq 0$, the $2$-functors \[ \overline{\CD}(\gh{m})\colon \overline{\CD}(\Sph{m})\to \overline{\CD}(\Reph{m}) \ \ \text{and} \ \ \overline{\CD}(\eh)\colon \overline{\CD}(\Reph{0})\to \overline{\CD}(\Nh I) \]
are biequivalences in $\TwoCat$. 
\end{lemme}

\begin{proof}
We first show the result for $\overline{\CD}(\gh{m})$. As in the proof of \cref{lem:IF} and by \cref{rem:overlineCD}, we have that 
\[ \overline{\CD}(\Sph{m})=[m] \ \ \text{and} \ \ \overline{\CD}(\Reph{m})=\OM{m}, \]
for all $m\geq 0$. One can prove that the $2$-functor $\overline{\CD}(\gh{m})$ is the identity on objects, essentially full on morphisms, and fully faithful on squares as in the proof of \cref{lem:IF}. Hence the $2$-functor $\overline{\CD}(\gh{m})$ is a biequivalence, for all $m\geq 0$.

It remains to show that $\overline{\CD}(\eh)$ is a biequivalence. We have that $\overline{\CD}(\Reph{0})=[0]$, and we compute $\overline{\CD}(\Nh I)$. Recall from \cref{ex:NRI} that $m$-simplices of the bisimplicial space $\Nh I$ constant in the vertical and space directions are given by words of $m$ letters in $\{x,y\}$. Since $\overline{\CD}(\Nh I)$ is obtained by gluing a copy of $\OM{m}$ for each $m$-simplex of~$\Nh I$, we have that $\overline{\CD}(\Nh I)$ has 
\begin{itemize}
    \item two objects $0$ and $1$, given by the $0$-simplices $x$ and $y$, 
    \item two non-trivial morphisms $f\colon 0\to 1$ and $g\colon 1\to 0$, given by the $1$-simplices $xy$ and $yx$, 
    \item two non-trivial invertible $2$-morphisms $\eta\colon \id_x\cong gf$ and $\epsilon\colon \id_y \cong fg$, given by the $2$-simplices $xyx$ and $yxy$,
\end{itemize}
such that $\eta$ and $\epsilon$ satisfy the triangle identities, expressed by the $3$-simplices $yxyx$ and $xyxy$. Higher simplices of $\Nh I$ do not add any relations. Therefore, the $2$-category $\overline{\CD}(\Nh I)=\Eadj$ is the ``free-living adjoint equivalence'', and $\overline{\CD}(\eh)=j_1\colon [0]\to \Eadj$ is a generating trivial cofibration in $\TwoCat$. 
\end{proof}

\subsection{The nerve \texorpdfstring{$\ND$}{N} is homotopically fully faithful} \label{subsec:ND-homotopicff}

We now show that the nerve functor is homotopically fully faithful. For this, we show that the derived counit of the adjunction $\CD\dashv \ND$ is a weak equivalence in $\DblCat$. More precisely, we show that it is a trivial fibration, i.e., a double functor which is surjective on objects, full on horizontal and vertical morphisms, and fully faithful on squares. Note that, since all objects are cofibrant in $\DblInfH$, the derived counit coincides with the counit. 

\begin{theorem} \label{thm:counit-CD-ND}
The components $\epsilon_{\bA}\colon \CD\ND\bA\to \bA$ of the (derived) counit are trivial fibrations in $\DblCat$, for all (fibrant) double categories $\bA$. In particular, these are weak equivalences in $\DblCat$ and therefore the nerve functor $\ND\colon \DblCat\to \DblInfH$ is homotopically fully faithful. 
\end{theorem}

\begin{proof}
Let $\bA$ be a double category. We first compute the double category $\CD\ND\bA$. By a formula for left Kan extensions, we have that 
\[ \CD\ND\bA=\mathrm{colim}(\mathcal Y\downarrow \ND\bA\longrightarrow \threeD\stackrel{\XD}{\longrightarrow} \DblCat), \]
where $\mathcal Y\colon \threeD\to \Set^{\threeDop}$ denotes the Yoneda embedding and $\mathcal Y\downarrow \ND\bA$ is the slice category over $\ND\bA$. An object in $\mathcal Y\downarrow \ND\bA$ is a map $\Reph{m}\times \Repv{k}\times \Reps{n}\to \ND\bA$, or equivalently a double functor $(\VK{k}\otimes \OM{m})\otimes \ON{n}\to \bA$, by the adjunction $\CD\dashv \ND$. Therefore, for each double functor $(\VK{k}\otimes \OM{m})\otimes \ON{n}\to \bA$, we glue a copy of $\XD_{m,k,n}=(\VK{k}\otimes \OM{m})\otimes \ON{n}$ in $\CD\ND\bA$.

The double category $\CD\ND\bA$ is cofibrant, since every object in $\DblInfH$ is cofibrant and $\CD$ is left Quillen. Therefore its underlying horizontal and vertical categories are free by \cref{prop:cofibrantDblCat'} and it is enough to describe the generating morphisms. First note that $\CD\ND\bA$ has the same objects as~$\bA$. Its horizontal morphisms are freely generated by
\begin{itemize}
    \item a horizontal morphism $\overline f\colon A\to B$, for each horizontal morphism $f$ of $\bA$, 
    \item a horizontal morphism $\widetilde f_{(f,g,\eta,\epsilon)}\colon A\to B$ together with a horizontal morphism $\widetilde g_{(f,g,\eta,\epsilon)}\colon B\to A$, for each horizontal adjoint equivalence $(f,g,\eta,\epsilon)$ in $\bA$,
\end{itemize}
where $\overline{\id_A}$, $\widetilde f_{(\id_A,\id_A,\id_{\id_A},\id_{\id_A})}$, and $\widetilde g_{(\id_A,\id_A,\id_{\id_A},\id_{\id_A})}$ are identified with the identity $\id_A$ at the object $A$ of $\CD\ND\bA$. The vertical morphisms in $\CD\ND\bA$ are freely generated by a vertical morphism $\overline u\colon A\arrowdot A'$, for each vertical morphism $u$ of $\bA$, where $\overline{e_A}$ is identified with the identity $e_A$ at the object $A$ of $\CD\ND\bA$. Finally, the squares of $\CD\ND\bA$ are generated by:
\begin{itemize}
    \item vertically invertible squares $\sq{\widetilde \eta_{(f,g,\eta,\epsilon)}}{\id_A}{\widetilde g\,\widetilde f}{e_A}{e_A}$ and $\sq{\widetilde \epsilon_{(f,g,\eta,\epsilon)}}{\widetilde f\,\widetilde g}{\id_B}{e_B}{e_B}$ satisfying the triangle identities, for each horizontal adjoint equivalence $(f,g,\eta,\epsilon)$ in~$\bA$, 
    \item a square $\sq{\overline \alpha}{\overline f}{\overline f'}{\overline u}{\overline v}$, for each square $\alpha$ in $\bA$, 
    \item a square $\sq{\widetilde \alpha}{\widetilde f}{\widetilde f'}{\overline u}{\overline v}$, for each square $\alpha$ in $\bA$ whose horizontal boundaries are horizontal adjoint equivalences $(f,g,\eta,\epsilon)$ and $(f',g',\eta',\epsilon')$, 
    \item a vertically invertible square $\sq{\overline \theta_{f,k,g,h}}{\widetilde g\, \overline f}{\overline k\,\widetilde h}{e_A}{e_C}$, for each vertically invertible square $\theta$ in $\bA$ as depicted below,
    \begin{tz}
    \node[](1) {$A$};
    \node[below of=1](2) {$A$}; 
    \node[right of=2](3) {$B'$};
    \node[right of=3](4) {$C$};
    \node[right of=1](6) {$B$};
    \node[above of=4](5) {$C$};
    \draw[d,pro] (1) to node(a)[]{} (2);
    \draw[d,pro] (5) to node(b)[]{} (4);
    \draw[->] (1) to node[above,la] {$f$} (6);
    \draw[->] (6) to node[above,la] {$g$} node[below,la]{$\simeq$} (5);
    \draw[->] (2) to node[above,la] {$\simeq$} node[below,la] {$h$} (3);
    \draw[->] (3) to node[below,la] {$k$} (4);
    \node[la] at ($(a)!0.45!(b)$) {$\theta$};
    \node[la] at ($(a)!0.55!(b)$) {$\vcong$};
    \end{tz}
    where $g$ and $h$ are horizontal adjoint equivalences,
    \item a vertically invertible square $\sq{\overline  \varphi_{f,g,h}}{\overline h}{\overline g\overline f}{e_A}{e_C}$, for each vertically invertible square~$\varphi$ in $\bA$ as depicted below,
    \begin{tz}
    \node[](1) {$A$};
    \node[below of=1](2) {$A$}; 
    \node[right of=2](3) {$B$};
    \node[right of=3](4) {$C$};
    \node[above of=4](5) {$C$};
    \draw[d,pro] (1) to node(a)[]{} (2);
    \draw[d,pro] (5) to node(b)[]{} (4);
    \draw[->] (1) to node[above,la] {$h$} (5);
    \draw[->] (2) to node[below,la] {$f$} (3);
    \draw[->] (3) to node[below,la] {$g$} (4);
    \node[la] at ($(a)!0.45!(b)$) {$\varphi$};
    \node[la] at ($(a)!0.55!(b)$) {$\vcong$};
    \end{tz}
    \item a vertically invertible square $\sq{\widetilde  \varphi_{f,g,h}}{\widetilde h}{\widetilde g\,\widetilde f}{e_A}{e_C}$, for each vertically invertible square $\varphi$ in $\bA$ as above, but where the morphisms $f$, $g$, and $h$ are all horizontal adjoint equivalences,
    \item a horizontally invertible square $\sq{\overline \psi_{u,v,w}}{\id_A}{\id_{A''}}{\overline w}{\overline v\,\overline u}$, for each horizontally invertible square $\sq{\psi}{\id_A}{\id_{A''}}{w}{vu}$ in $\bA$.
\end{itemize}
Furthermore, these squares are submitted to relations represented by double functors $(\VK{k}\otimes \OM{m})\otimes \ON{n}\to \bA$, where $k+m+n\geq 3$. In particular, these relations hold for the squares that represent them in $\bA$. 

Then the double functor $\epsilon_\bA\colon \CD\ND\bA\to \bA$ is given by the identity on objects and by sending each horizontal morphism, vertical morphism, and square in $\CD\ND\bA$ to the horizontal morphism, vertical morphism, and square in $\bA$ representing it. This defines a double functor since the underlying horizontal and vertical categories are free, and the relations on squares in $\CD\ND\bA$ are satisfied by the squares representing them in $\bA$. Moreover, it is straightforward to see that this double functor is surjective on objects, full on horizontal morphisms, full on vertical morphisms, and full on squares. Faithfulness on squares follows from the fact that, given a boundary in $\CD\ND\bA$, for each square in $\bA$ in the representing boundary, we added a unique square in $\CD\ND\bA$ with that boundary, and the fact that the relations satisfied for squares in $\bA$ are also satisfied in $\CD\ND\bA$. 
\end{proof}

\begin{rem}
    Note that, since the functor $\epsilon_\bA\colon \CD\ND\bA\to \bA$ is fully faithful on squares, the relations imposed on the generating squares in $\CD\ND\bA$ are completely determined by their image in $\bA$ under the double functor $\epsilon_\bA$. 
\end{rem}

\begin{rem}
Since $\DblInfH$ is obtained as a localization of the Reedy/injective model structure on $\sSet^{\twoDop}$, all objects are cofibrant in $\DblInfH$, and hence the functor $\CD\colon \DblInfH\to \DblCat$ preserves weak equivalences by Ken Brown's Lemma (see \cite[Lemma 1.1.12]{Hovey}). Therefore, since the components $\epsilon_\bA\colon \CD\ND\bA\to \bA$ of the counit are weak equivalences by \cref{thm:counit-CD-ND}, for all $\bA\in \DblCat$, the nerve $\ND\colon \DblCat\to \DblInfH$ reflects weak equivalences by $2$-out-of-$3$. 
\end{rem}

\subsection{Level of fibrancy of nerves of double categories} \label{subsec:fibrantnerve}

The nerve of any double category is almost fibrant in the model structure $\DblInfH$ of \cref{thm:MSdoubleInf}. Indeed, aside from the vertical Reedy/injective fibrancy condition, the nerve of a double category satisfies the conditions of a horizontally complete double $(\infty,1)$-category.

\begin{theorem} \label{prop:properties-nerve}
The nerve of a double category $\bA$ is such that
\begin{rome}
\item $(\ND\bA)_{-,k}\colon \Delta^\op\to \sSet$ is Reedy/injective fibrant, for all $k\geq 0$, 
\item $(\ND\bA)_{m,-}\colon \Delta^\op\to \sSet$ satisfies the Segal condition, for all $m\geq 0$,
\item $(\ND\bA)_{-,k}\colon \Delta^\op\to \sSet$ is a complete Segal space, for all $k\geq 0$.
\end{rome}
\end{theorem}

To show this theorem we will need several technical results. The first piece is a Quillen pair between $\TwoCat$ and $\sSet$ whose left adjoint is given by the restriction of the functor $\overline{\CD}\colon \sSet^{\twoDop}\to \TwoCat$ to its space component. 

\begin{defn}
We define the cosimplicial $2$-category 
\[ \cX\colon \Delta \to \TwoCat, \quad
[n]\mapsto \ON{n}. \]
\end{defn}

\begin{prop} \label{prop:cCcNadj}
The cosimplicial $2$-category $\cX$ induces an adjunction
\begin{tz}
\node[](1) {$\Delta$};
\node[below of=1](2) {$\Set^{\Delta^\op}$};
\node[right of=1,xshift=1cm](3) {$\TwoCat$};
\node at ($(3.east)-(0,4pt)$) {,};
\draw[->] (1) to node[above,la]{$\cX$} (3);
\draw[->] (1) to (2);
\draw[->] (2.north east) to node(c)[left,la,yshift=8pt] {$\cC$} (3.south west);
\coordinate(a) at ($(3.south west)+(.4cm,0)$);
\draw[->,bend left=40] (a) to node(n)[right,la,yshift=-8pt] {$\cN$} (2.east);
\node[la] at ($(c)!0.5!(n)$) {$\rotatebox{220}{$\top$}$};
\end{tz}
where $\cC$ is the left Kan extension of $\cX$ along the Yoneda embedding, and we have that 
\[ (\cN \cA)_n\cong \TwoCat(\ON{n}, \cA), \]
for all $\cA\in \TwoCat$ and all $n\geq 0$.
\end{prop}

\begin{proof}
This is a direct application of \cite[Theorem 1.1.10]{Cisinski}.
\end{proof}

\begin{prop} \label{prop:QuillenC2N2}
The adjunction 
\begin{tz}
\node[](A) {$\TwoCat$};
\node[right of=A,xshift=1cm](B) {$\sSet$};
\draw[->] ($(B.west)+(0,.25cm)$) to [bend right=25] node[above,la]{$\cC$} ($(A.east)+(0,.25cm)$);
\draw[->] ($(A.east)-(0,.25cm)$) to [bend right=25] node[below,la]{$\cN$} ($(B.west)+(0,-.25cm)$);
\node[la] at ($(A.east)!0.5!(B.west)$) {$\bot$};
\end{tz}
is a Quillen pair between Lack's model structure on $\TwoCat$ and the Quillen model structure on $\sSet$.
\end{prop}

\begin{proof}
It is enough to show that $\cC$ sends generating cofibrations and generating trivial cofibrations in $\sSet$ to cofibrations and trivial cofibrations in $\TwoCat$, respectively. Recall that generating cofibrations and generating trivial cofibrations in $\sSet$ are given by the maps $\iotas{n}\colon \partial \Reps{n}\to \Reps{n}$, for $n\geq 0$, and $\ell^s_{n,t}\colon \Lambda^t[n]\to \Reps{n}$, for $n\geq 1$ and $0\leq t\leq n$, respectively. Note that we have $\cC(\iotas{n})=\overline{\CD}(\iotas{n})$ and $\cC(\ell^s_{n,t})=\overline{\CD}(\ell^s_{n,t})$. Therefore, by \cref{lem:iotaRDelta,lem:ellDelta}, we see that these are cofibrations and trivial cofibrations in $\TwoCat$, respectively. 
\end{proof}

We now reformulate conditions (i-iii) of \cref{prop:properties-nerve}, which are for now given in terms of weak equivalences between mapping spaces, using the right Quillen functor $\cN$ of the above proposition. This can be done by applying the following lemma. 

\begin{lemme} \label{lem:cNbfH}
Let $X\in \sSet^\twoDop$ be a bisimplicial space which is constant in the space direction. Then, for every double category $\bA$, we have an isomorphism of simplicial sets
\[ \Map(X,\ND\bA)\cong \cN(\bfH[\CD(X),\bA]_\ps) \]
natural in $X$ and $\bA$.
\end{lemme}

\begin{proof}
For all $n\geq 0$, we have natural isomorphisms of sets
\begin{align*}
    \Map(X,\ND\bA)_n &\cong \sSet^{\twoDop}(X\times \Reps{n},\ND\bA) \cong \DblCat(\CD(X\times \Reps{n}),\bA) \\
    & \cong \DblCat(\CD(X)\otimes \ON{n}, \bA) \cong \TwoCat(\ON{n}, \bfH[\CD(X),\bA]_\ps) \\
    & \cong \cN(\bfH[\CD(X),\bA]_\ps)_n, 
\end{align*}
where the first isomorphism holds by definition of the mapping space, the second by the adjunction $\CD\dashv \ND$, the third by definition of $\CD$ and the fact that $X$ is constant in the space direction, the fourth by the universal property of $\otimes$ (see \cref{def:tensor}), and the last isomorphism by definition of $\cN$. These isomorphisms of sets assemble into an isomorphism $\Map(X,\ND\bA)\cong \cN(\bfH[\CD(X),\bA]_\ps)$ of simplicial sets, natural in $X$ and $\bA$.
\end{proof}

We now prove \cref{prop:properties-nerve} assuming \cref{lem:bfHIR,lem:bfHiotaF} below.

\begin{proof}[Proof of \cref{prop:properties-nerve}]
Let $\bA$ be a double category. By \cref{lem:bfHIR,lem:bfHiotaF} below, for all $m,k\geq 0$, the $2$-functor $\bfH[\CD(\id_{\Repv{k}}\times \iotah{m}),\bA]_\ps$ is a fibration in $\TwoCat$, and the $2$-functors $\bfH[\CD(\id_{\Repv{k}}\times \gh{m}),\bA]_\ps$, $\bfH[\CD(\id_{\Repv{k}}\times \eh),\bA]_\ps$, and $\bfH[\CD(\gv{k}\times \id_{\Reph{m}}),\bA]_\ps$ are trivial fibrations in $\TwoCat$. As $\cN\colon \TwoCat\to \sSet$ is right Quillen by \cref{prop:QuillenC2N2}, these are sent by $\cN$ to fibrations and trivial fibrations in $\sSet$, respectively. As the map $\id_{\Repv{k}}\times \iotah{m}$ is constant in the space direction, by \cref{lem:cNbfH}, we have that 
\[ \Map(\id_{\Repv{k}}\times \iotah{m},\ND\bA)\cong \cN(\bfH[\CD(\id_{\Repv{k}}\times \iotah{m}),\bA]_\ps). \]
By the above arguments, this is a fibration in $\sSet$, for all $m,k\geq 0$, which shows (i) saying that $(\ND\bA)_{-,k}$ is Reedy/injective fibrant. Similarly, we have that
\[ \Map(\id_{\Repv{k}}\times \gh{m},\ND\bA)\cong \cN(\bfH[\CD(\id_{\Repv{k}}\times \gh{m}),\bA]_\ps), \]
\[ \Map(\id_{\Repv{k}}\times \eh,\ND\bA)\cong \cN(\bfH[\CD(\id_{\Repv{k}}\times \eh),\bA]_\ps), \]
\[ \Map(\gv{k}\times \id_{\Reph{m}},\ND\bA)\cong \cN(\bfH[\CD(\gv{k}\times \id_{\Reph{m}}),\bA]_\ps), \]
and these are trivial fibrations in $\sSet$ by the above arguments, for all $m,k\geq 0$. The fact that $\Map(\id_{\Repv{k}}\times \gh{m},\ND\bA)$ and $ \Map(\id_{\Repv{k}}\times \eh,\ND\bA)$ are weak equivalences in $\sSet$ shows that (iii) holds, i.e., we have the Segal and completeness conditions for $(\ND\bA)_{-,k}$, and the fact that $\Map(\gv{k}\times \id_{\Reph{m}},\ND\bA)$ is a weak equivalence in~$\sSet$ gives (ii), i.e., the Segal condition for $(\ND\bA)_{m,-}$. 
\end{proof}

\begin{lemme} \label{lem:bfHIR}
Let $\bA$ be a double category. The $2$-functor $\bfH[\CD(\id_{\Repv{k}}\times \iotah{m}),\bA]_\ps$ is a fibration in $\TwoCat$, and the $2$-functors $\bfH[\CD(\id_{\Repv{k}}\times \gh{m}),\bA]_\ps$ and $\bfH[\CD(\id_{\Repv{k}}\times \eh),\bA]_\ps$ are trivial fibrations in $\TwoCat$, for all $m,k\geq 0$.
\end{lemme}

\begin{proof}
By promoting the bijections in \cref{def:tensor} of the tensor $\otimes$, we get isomorphisms of $2$-categories as in the following commutative square. 
\begin{tz}
\node[](1) {$\bfH[\VK{k}\otimes \OM{m},\bA]_\ps$};
\node[right of=1,xshift=6.5cm](2) {$\bfH[\VK{k}\otimes \partial \OM{m},\bA]_\ps$};
\node[below of=1](3) {$[\OM{m},\bfH[\VK{k},\bA]_\ps]_{2,\ps}$};
\node[below of=2](4) {$[\partial \OM{m},\bfH[\VK{k},\bA]_\ps]_{2,\ps}$};
\draw[->] (1) to node[above,la] {$\bfH[\CD(\id_{\Repv{k}}\times \iotah{m}),\bA]_\ps$} (2);
\draw[->] (3) to node[below,la,yshift=-4pt,xshift=4pt] {$[\overline{\CD}(\iotah{m}),\bfH[\VK{k},\bA]_\ps]_{2,\ps}$} (4);
\draw[->] (1) to node[left,la] {$\cong$} (3);
\draw[->] (2) to node[right,la] {$\cong$} (4);
\end{tz}
As every $2$-category is fibrant and $\overline{\CD}(\iotah{m})$ is a cofibration in $\TwoCat$ by \cref{lem:iotaRDelta}, the $2$-functor $[\overline{\CD}(\iotah{m}),\bfH[\VK{k},\bA]_\ps]_{2,\ps}$ is a fibration in $\TwoCat$ by monoidality of Lack's model structure. Hence $\bfH[\CD(\id_{\Repv{k}}\times \iotah{m}),\bA]_\ps$ is also a fibration in $\TwoCat$. 

Similarly, we have isomorphisms $\bfH[\CD(\id_{\Repv{k}}\times \gh{m}),\bA]_\ps\cong [\overline{\CD}(\gh{m}),\bfH[\VK{k},\bA]_\ps]_{2,\ps}$ and $\bfH[\CD(\id_{\Repv{k}}\times \eh),\bA]_\ps\cong [\overline{\CD}(\eh),\bfH[\VK{k},\bA]_\ps]_{2,\ps}$. By \cref{lem:IR} and since $\overline{\CD}$ preserves cofibrations, the $2$-functors $\overline{\CD}(\gh{m})$ and $\overline{\CD}(\eh)$ are trivial cofibrations in $\TwoCat$. Therefore, by monoidality of Lack's model structure, the $2$-functors 
\[ [\overline{\CD}(\gh{m}),\bfH[\VK{k},\bA]_\ps]_{2,\ps} \ \ \text{and} \ \ [\overline{\CD}(\eh),\bfH[\VK{k},\bA]_\ps]_{2,\ps} \] are trivial fibrations in $\TwoCat$, which shows the second part of the statement.
\end{proof}

The last piece for the proof of \cref{prop:properties-nerve} makes use of the data in the $2$-category $\bfH[-,-]_\ps$ of double functors, horizontal pseudo-natural transformations, and modifications, whose definitions can be found in \cref{subsec:ps-equ}.

\begin{lemme} \label{lem:bfHiotaF}
Let $\bA$ be a double category. The $2$-functor $\bfH[\CD(\gv{k}\times \id_{\Reph{m}}),\bA]_\ps$ is a trivial fibration in $\TwoCat$, for all $m,k\geq 0$.
\end{lemme}

\begin{proof}
By promoting the bijections in \cref{def:Graytensordbl} of the Gray tensor $\otimes_G$, we get isomorphisms of $2$-categories as in the following commutative square.
\begin{tz}
\node[](1) {$\bfH[\VK{k}\otimes \OM{m},\bA]_\ps$};
\node[right of=1,xshift=6.5cm](2) {$\bfH[\bbV[k]\otimes \OM{m},\bA]_\ps$};
\node[below of=1](3) {$\bfH[\VK{k},[\bbH\OM{m},\bA]_\ps]_\ps$};
\node[below of=2](4) {$\bfH[\bbV[k],[\bbH\OM{m},\bA]_\ps]_\ps$};
\draw[->] (1) to node[above,la] {$\bfH[\CD(\gv{k}\times \id_{\Reph{m}}),\bA]_\ps$} (2);
\draw[->] (3) to node[below,la,yshift=-4pt,xshift=4pt] {$\bfH[\CD(\gv{k}),[\bbH\OM{m},\bA]_\ps]_\ps$} (4);
\draw[->] (1) to node[left,la] {$\cong$} (3);
\draw[->] (2) to node[right,la] {$\cong$} (4);
\end{tz}
Hence, to see that the $2$-functor $\bfH[\CD(\gv{k}\times \id_{\Reph{m}}),\bA]_\ps$ is a trivial fibration in $\TwoCat$, it is enough to show that the $2$-functor $\bfH[\CD(\gv{k}),\bB]_\ps\colon \bfH[\VK{k},\bB]_\ps\to \bfH[\bbV[k],\bB]_\ps$ is a trivial fibration in $\TwoCat$, for any $\bB\in \DblCat$.

We first describe the double functor $\CD(\gv{k})\colon \bbV[k]\to \VK{k}$ on objects and vertical morphisms. Since the horizontal morphisms and squares of $\bbV[k]$ are all trivial, this describes the image of $\CD(\gv{k})$ completely. We denote by ${u_i\colon i\arrowdot i+1}$, for $0\leq i<k$, the generating vertical morphisms of $\bbV[k]$. Then the double functor $\CD(\gv{k})$ is the identity on objects and sends a generating vertical morphism $u_i\colon i\arrowdot i+1$ of $\bbV[k]$ to the vertical morphism $i\arrowdot i+1$ of $\VK{k}$ represented by $\{i,i+1\}$. 

Now let $\bB$ be a double category. We show that the $2$-functor $\bfH[\CD(\gv{k}),\bB]_\ps$ is a trivial fibration in $\TwoCat$, by verifying that it is surjective on objects, full on morphisms, and fully faithful on $2$-morphisms.

Given a double functor $F\colon \bbV[k]\to \bB$, consider the composite 
\[ \VK{k}\stackrel{\bbV\pi}{\longrightarrow} \bbV[k]\stackrel{F}{\longrightarrow} \bB, \] 
where $\pi\colon \OM{k}\to [k]$ is the identity on objects and acts on hom-categories as the unique functor $\OM{k}(i,j)\to [k](i,j)=[0]$. The composite above is a double functor in $\bfH[\VK{k},\bB]$ such that $F\circ \bbV\pi\circ \CD(\gv{k})=F$, which proves surjectivity on objects.

Let $F,G\colon \VK{k}\to \bB$ be double functors, and $\varphi\colon F\CD(\gv{k})\Rightarrow G\CD(\gv{k})$ be a horizontal pseudo-natural transformation in $\bfH[\bbV[k],\bB]_\ps$. We want to define a horizontal pseudo-natural transformation $\overline{\varphi}\colon F\Rightarrow G$ in $\bfH[\VK{k},\bB]_\ps$ such that $\overline{\varphi}\CD(\gv{k})=\varphi$. It is enough to define $\overline{\varphi}$ on the generating vertical morphisms of $\VK{k}$ which are represented by $\{i,j\}$ for $i<j$. When $j=i+1$, we set $\overline{\varphi}_{\{i,i+1\}}\coloneqq \varphi_{u_i}$. For $j>i+1$, let $\theta$ denote the unique horizontally invertible square in $\VK{k}$ from the vertical morphism represented by $\{i,j\}$ to the vertical composite of morphisms represented by $[i,j]=\{l\mid i\leq l\leq j\}$. Then there is a unique way of defining $\overline{\varphi}_{\{i,j\}}$ so that $\overline{\varphi}$ is natural; namely as follows. 
\begin{tz}
\node[](1) {$Fi$}; 
\node[below of=1](2) {$Fj$}; 
\node[right of=1](3) {$Gi$}; 
\node[below of=3](4) {$Gj$}; 
\draw[->] (1) to (3);
\draw[->] (2) to (4); 
\draw[pro,->] (1) to node[left,la] {$F_{\{i,j\}}$} (2);
\draw[pro,->] (3) to node(a)[right,la] {$G_{\{i,j\}}$} (4);
\node[la] at ($(1)!0.5!(4)$) {$\overline{\varphi}_{\{i,j\}}$};

\node[right of=3,yshift=1.5cm,xshift=1cm](1) {$Fi$}; 
\node[right of=1,xshift=.5cm](2) {$Fi$}; 
\node[below of=2](3) {$F(i+1)$}; 
\node[below of=3](4) {$\vdots$}; 
\node[below of=4](5) {$Fj$}; 
\node[left of=5,xshift=-.5cm](6) {$Fj$}; 
\draw[pro,->] (1) to node(b)[left,la] {$F_{\{i,j\}}$} (6);
\node[la] at ($(a)!0.5!(b)$) {$=$};
\draw[pro,->] (2) to node[left,la] {$F_{\{i,i+1\}}$} (3);
\draw[pro,->] (3) to (4); 
\draw[pro,->] (4) to (5);
\draw[d] (1) to (2); 
\draw[d] (6) to (5); 
\node[la] at ($(1)!0.5!(5)$) {$F\theta$};
\node[yshift=-.4cm,la] at ($(1)!0.5!(5)$) {$\cong$};

\node[right of=2,xshift=.5cm](2') {$Gi$}; 
\node[right of=2',xshift=.5cm](1) {$Gi$}; 
\node[below of=2'](3') {$G(i+1)$}; 
\node[below of=3'](4') {$\vdots$}; 
\node[below of=4'](5') {$Gj$}; 
\node[right of=5',xshift=.5cm](6) {$Gj$}; 
\draw[pro,->] (1) to node[right,la] {$G_{\{i,j\}}$} (6);
\draw[pro,->] (2') to node[right,la] {$G_{\{i,i+1\}}$} (3');
\draw[pro,->] (3') to (4'); 
\draw[pro,->] (4') to (5');
\draw[d] (2') to (1); 
\draw[d] (5') to (6); 
\node[la,xshift=5pt] at ($(1)!0.5!(5')$) {$(G\theta)^{-1}$};
\node[yshift=-.4cm,la] at ($(1)!0.5!(5')$) {$\cong$};

\draw[->] (2) to (2');
\draw[->] (3) to (3');
\draw[->] (5) to (5');
\node[la] at ($(2)!0.5!(3')$) {$\varphi_{u_i}$}; 
\node(a)[la] at ($(3)!0.5!(4')$) {$\varphi_{u_{i+1}}$}; 
\node(b)[la] at ($(4)!0.5!(5')$) {$\varphi_{u_{j-1}}$}; 
\node at ($(a)!0.5!(b)$) {$\vdots$};
\end{tz}
This defines a horizontal pseudo-natural transformation $\overline{\varphi}\colon F\Rightarrow G$ which maps to $\varphi$ via $\bfH[\CD(\gv{k}),\bB]_\ps$, and hence shows fullness on morphisms. 

Let $\overline{\varphi},\overline{\psi}\colon F\Rightarrow G$ be horizontal pseudo-natural transformations in $\bfH[\VK{k},\bB]_\ps$, and let $\mu\colon \varphi\coloneqq \overline{\varphi}\CD(\gv{k})\to \psi\coloneqq \overline{\psi}\CD(\gv{k})$ be a modification in $\bfH[\bbV[k],\bB]_\ps$. The modification~$\mu$ comprises the data of squares $\sq{\mu_i}{\varphi_i}{\psi_i}{e_{Fi}}{e_{Gi}}$, for $0\leq i\leq k$, natural with respect to the square components of $\varphi$ and $\psi$. By the relations between the square components of $\overline{\varphi}$ and~$\varphi$, and the ones of $\overline{\psi}$ and $\psi$ as indicated in the pasting equality above, one can show that the squares $\mu_i$ of $\mu$ are also natural with respect to the square components of $\overline{\varphi}$ and~$\overline{\psi}$. Therefore $\mu$ also defines a modification $\mu\colon \overline{\varphi}\to \overline{\psi}$ in $\bfH[\VK{k},\bB]_\ps$. As it is the unique such modification in $\bfH[\VK{k},\bB]_\ps$ that maps to $\mu$ via $\bfH[\CD(\gv{k}),\bB]_\ps$, this shows fully faithfulness on $2$-morphisms. 
\end{proof}

Finally, we show that the nerve of a double category satisfies the missing condition of a horizontally complete double $(\infty,1)$-category in the list of \cref{prop:properties-nerve}, namely the Reedy/injective fibrancy in the vertical direction, precisely when the double category is weakly horizontally invariant. Recall that the weakly horizontally invariant double categories are the fibrant objects in the model structure on $\DblCat$ of \cref{thm:WHI-MS}.

\begin{theorem} \label{thm:fibrantnerve}
The nerve of a double category $\bA$ is such that $(\ND\bA)_{m,-}\colon \Delta^\op\to \sSet$ is Reedy/injective fibrant, for all $m\geq 0$, if and only if the double category $\bA$ is weakly horizontally invariant.
\end{theorem}

\begin{proof} 
Let $\bA$ be a double category. If $\bA$ is weakly horizontally invariant, then $\ND\bA$ is a horizontally complete double $(\infty,1)$-category since $\ND\colon \DblCat\to \DblInfH$ is right Quillen. In particular, this says that $(\ND\bA)_{m,-}\colon \Delta^\op\to \sSet$ is Reedy/injective fibrant, for all $m\geq 0$.

Conversely, suppose that $(\ND\bA)_{m,-}\colon \Delta^\op\to \sSet$ is Reedy/injective fibrant, for all ${m\geq 0}$. Then $(\ND\bA)_{0,-}$ is Reedy/injective fibrant and therefore the map
\[ (\iotav{1})^*\colon (\ND\bA)_{0,1}\cong \Map (\Repv{1},\ND\bA)\to \Map(\partial \Repv{1}, \ND\bA)\cong (\ND\bA)_{0,0}\times (\ND\bA)_{0,0}. \]
is a fibration in $\sSet$, by \cref{rem:Reedyfibrant}. In particular, it has the right lifting property with respect to $\ell^s_{1,1}\colon \Reps{0}\to \Reps{1}$, i.e., there is a lift in every commutative diagram as below left.
\begin{tz}
\node[](1) {$\Reps{0}$};
\node[right of=1,rr](2) {$(\ND\bA)_{0,1}$};
\node[below of=1,yshift=-.5cm](3) {$\Reps{1}$};
\node[below of=2,yshift=-.5cm](4) {$(\ND\bA)_{0,0}\times (\ND\bA)_{0,0}$};
\draw[->] (1) to node[left,la] {$\ell^s_{1,1}$} (3);
\draw[->] (2) to node[right,la]{$(\iotav{1})^*$} (4);
\draw[->] (1) to node[above,la]{$v$} (2);
\draw[->] (3) to node[below,la]{$(f,f')$} (4);
\draw[->,dashed] (3) to node[above,pos=0.4,la]{$\alpha$} (2);

\node[right of=2,xshift=2cm,yshift=-.25cm](1) {$A$};
\node[right of=1](2) {$B$};
\node[below of=1](3) {$A'$};
\node[below of=2](4) {$B'$};
\draw[->,pro] (1) to node[left,la] {$u$} (3);
\draw[->,pro] (2) to node[right,la]{$v$} (4);
\draw[->] (1) to node[below,la]{$\simeq$} node[above,la]{$f$} (2);
\draw[->] (3) to node[below,la]{$f'$} node[above,la]{$\simeq$} (4);
\node[la] at ($(1)!0.5!(4)-(5pt,0)$) {$\alpha$};
\node[la] at ($(1)!0.5!(4)+(5pt,0)$) {$\simeq$};
\end{tz}
By \cref{desc:double00,desc:double01}, the upper map $v$ is the data of a vertical morphism $v\colon B\arrowdot B'$ in $\bA$, while the bottom map $(f,f')$ is the data of a pair of horizontal adjoint equivalences $(f\colon A\xrightarrow{\simeq} B,f'\colon A'\xrightarrow{\simeq} B')$ in $\bA$. Therefore, by \cref{desc:double01} the existence of a lift in each diagram as above corresponds to the existence of a weakly horizontally invertible square in $\bA$ as depicted above right, for each such data $(v,f,f')$. In other words, this says that $\bA$ is weakly horizontally invariant.
\end{proof}

\begin{rem} \label{rem:NHnotfibrant}
In particular, since a horizontal double category is not generally weakly horizontally invariant (see \cite[Remark 6.4]{MSVsecond}), the nerve $\ND\bbH\cA$ of a $2$-category $\cA$ is not generally fibrant in $\DblInfH$. Since every $2$-category is fibrant in Lack's model structure on $\TwoCat$, this shows that the composite $\ND\bbH$ is not right Quillen from $\TwoCat$ to $\DblInfH$. Therefore, we will need to define the nerve for $2$-categories differently in the next section. 
\end{rem}

\section{Nerve of \texorpdfstring{$2$}{2}-categories} \label{sec:nerve2Cat}

As $2$-categories are horizontally embedded in double categories, we hope that the nerve functor $\ND\colon \DblCat\to \DblInfH$ restricts to a nerve functor $\TwoCat\to \TwoCSS$. Since the nerve of a double category $\bbH\cA$ associated to a $2$-category $\cA$ is not generally fibrant, as explained in \cref{rem:NHnotfibrant}, we need to define the nerve of a $2$-category as the nerve of the fibrant replacement $\bbHsim\cA$ of $\bbH\cA$ in $\DblCat$; see \cref{prop:fibreplHA}. In \cref{subsec:NH-Quillen-hff}, we show that the composite of the Quillen pairs $\Lsim\dashv \bbHsim$ and $\CD\dashv \ND$ restrict to a Quillen pair between $\TwoCat$ and $\TwoCSS$, whose derived counit is level-wise a biequivalence. Hence the nerve $\ND\bbHsim$ gives a homotopically full embedding of $\TwoCat$ into~$\TwoCSS$. We further show in \cref{subsec:rightinduced} that Lack's model on $\TwoCat$ is right-induced from~$\TwoCSS$ along~$\ND\bbHsim$, which implies that the homotopy theory of $2$-categories is completely determined by that of $2$-fold complete Segal spaces through its image under $\ND\bbHsim$. In \cref{subsec:NHvsNHsim}, we compare the nerve of the double categories $\bbH\cA$ and $\bbHsim\cA$, by showing that the nerve of the latter is a fibrant replacement of the nerve of the former in $\TwoCSS$, and hence also in $\DblInfH$. 

\subsection{The nerve \texorpdfstring{$\ND\bbHsim$}{NH-sim} is right Quillen and homotopically fully faithful} \label{subsec:NH-Quillen-hff}

We consider the composite of the Quillen pairs 
\begin{tz}
\node[](Z) {$\TwoCat$};
\node[right of=Z,xshift=1.2cm](A) {$\DblCat$};
\node[right of=A,xshift=1.8cm](B) {$\DblInfH$};
\node at ($(B.east)-(0,4pt)$) {,};
\draw[->] ($(A.west)+(0,.25cm)$) to [bend right=25] node[above,la]{$\Lsim$} ($(Z.east)+(0,.25cm)$);
\draw[->] ($(Z.east)-(0,.25cm)$) to [bend right=25] node[below,la]{$\bbHsim$} ($(A.west)+(0,-.25cm)$);
\node[la] at ($(Z.east)!0.5!(A.west)$) {$\bot$};
\draw[->] ($(B.west)+(0,.25cm)$) to [bend right=25] node[above,la]{$\CD$} ($(A.east)+(0,.25cm)$);
\draw[->] ($(A.east)-(0,.25cm)$) to [bend right=25] node[below,la]{$\ND$} ($(B.west)+(0,-.25cm)$);
\node[la] at ($(A.east)!0.5!(B.west)$) {$\bot$};
\end{tz}
and show that this gives a Quillen pair between $\TwoCat$ and the localization $\TwoCSS$ of $\DblInfH$.

\begin{theorem} \label{thm:Quillen-NDH}
The adjunction 
\begin{tz}
\node[](A) {$\TwoCat$};
\node[right of=A,xshift=1.2cm](B) {$\TwoCSS$};
\draw[->] ($(B.west)+(0,.25cm)$) to [bend right=25] node[above,la]{$\Lsim\CD$} ($(A.east)+(0,.25cm)$);
\draw[->] ($(A.east)-(0,.25cm)$) to [bend right=25] node[below,la]{$\ND\bbHsim$} ($(B.west)+(0,-.25cm)$);
\node[la] at ($(A.east)!0.5!(B.west)$) {$\bot$};
\end{tz}
is a Quillen pair between Lack's model structure on $\TwoCat$ and the model structure on $\sSet^{\twoDop}$ for $2$-fold complete Segal spaces, i.e., $(\infty,2)$-categories.
\end{theorem}

\begin{rem} \label{rem:noninv}
Note that the functor $\Lsim\colon \DblCat\to \TwoCat$ does not preserve tensors. For example, the $2$-category $\Lsim(\bbV[1]\otimes [1])$ is generated by a non-invertible $2$-morphism as below left, while the $2$-category $\Lsim(\bbV[1])\otimes_2 [1]$ is generated by an invertible $2$-morphism as below right.
\begin{tz}
\node[](1) {$0$};
\node[right of=1](2) {$1$};
\node[below of=1](3) {$0'$};
\node[below of=2](4) {$1'$};
\draw[->] (1) to (2);
\draw[->] (3) to (4);
\draw[->] (1) to node[left,la]{$\simeq$} (3);
\draw[->] (2) to node[right,la]{$\simeq$} (4);
\cell[]{2}{3}{};

\node[right of=2,xshift=1cm](1) {$0$};
\node[right of=1](2) {$1$};
\node[below of=1](3) {$0'$};
\node[below of=2](4) {$1'$};
\draw[->] (1) to (2);
\draw[->] (3) to (4);
\draw[->] (1) to node[left,la]{$\simeq$} (3);
\draw[->] (2) to node[right,la]{$\simeq$} (4);
\cell[la,left,yshift=8pt]{2}{3}{$\cong$};
\end{tz}
However, the fact that the left-hand $2$-morphism is not invertible in a square coming from a pair of a vertical morphism and a horizontal morphism is the only difference between $\Lsim(-\otimes -)$ and $\Lsim(-)\otimes_2 \Lsim(-)$.
\end{rem}

\begin{proof}
First note that the adjunction $\Lsim\CD\dashv \ND\bbHsim$ is a Quillen pair between $\TwoCat$ and $\DblInfH$, since it is a composite of two Quillen pairs. By \cref{thm:Quillen-loc}, it is enough to show that the functor $\Lsim\CD$ sends the cofibrations $\ev\times \id_{\Reph{m}}$ and $c^v_k$, with respect to which we localize $\DblInfH$ to obtain $\TwoCSS$, to weak equivalences in $\TwoCat$. 

We first show that $\Lsim\CD(\ev\times \id_{\Reph{m}})$ is a biequivalence. By a similar computation to the one of $\CD (\Nh I)$ in the proof of \cref{lem:IR}, we obtain that 
\[ \Lsim\CD(\Nv I\times \Reph{m})\cong \Lsim(\bbV \ON{1}\otimes \OM{m}). \]
Then the squares in the tensor $\bbV \ON{1}\otimes \OM{m}$ induced from vertical morphisms in $\bbV\ON{1}$ and morphisms in $\OM{m}$ must be weakly vertically invertible, since all vertical morphisms in $\bbV\ON{1}$ are vertical equivalences, and these correspond to invertible $2$-morphisms in $\Lsim(\bbV \ON{1}\otimes \OM{m})$, by a dual version of \cref{lem:whiinLsim}. By \cref{rem:noninv}, we deduce that $\Lsim$ preserves this tensor:
\[ \Lsim(\bbV \ON{1}\otimes \OM{m})\cong \ON{1}\otimes_2 \OM{m}\cong \Lsim\CD(\Nv I)\otimes_2 \Lsim\CD(\Reph{m}). \]
Therefore, $\Lsim\CD(\ev\times \id_{\Reph{m}})\cong \Lsim\CD(\ev)\,\square_{\otimes_2}\, (\emptyset\to \Lsim\CD \Reph{m})$. Both morphisms in this pushout-product are cofibrations in $\TwoCat$ since $\Lsim\CD$ is left Quillen from $\DblInfH$, and therefore, by \cref{rem:pushprod2Cat}, it is enough to show that $\Lsim\CD(\ev)$ is a biequivalence. But this is clear since the $2$-functor $\Lsim\CD(\ev)\colon \Lsim\CD(\Repv{0})\to \Lsim\CD(\Nv I)$ can be identified with the generating trivial cofibration $j_1\colon [0]\to \Eadj$ in $\TwoCat$.

We now show that the $2$-functor $\Lsim\CD(c^v_k)\colon \Lsim\CD(\Repv{0})\to \Lsim\CD(\Repv{k})$ is a biequivalence. It is given by the inclusion $[0]\to \ON{k}$ at $0$. First note that for $k=0$, this is the identity. For $k\geq 1$, it is a biequivalence since it is
\begin{itemize}
\item bi-essentially surjective on objects as every object in $\ON{k}$ is related by an adjoint equivalence to the object $0$,
\item essentially full on morphisms since every composite of adjoint equivalences $0\to 0$ in $\ON{k}$ is related by an invertible $2$-morphism to $\id_0$, which is given by a pasting of units and counits of the corresponding adjoint equivalences, 
\item fully faithful on $2$-morphisms since the only $2$-morphism $\id_0\Rightarrow\id_0$ in $\ON{k}$ is the identity. \qedhere
\end{itemize} 
\end{proof}

As in the double categorical case, the nerve $\ND\bbHsim$ is homotopically fully faithful.

\begin{theorem} \label{thm:counit-NDH}
The derived counit of the adjunction $\Lsim\CD\dashv \ND\bbHsim$ is level-wise a biequivalence. In particular, the nerve $\ND\bbHsim\colon \TwoCat\to \TwoCSS$ is homotopically fully faithful. 
\end{theorem}

\begin{proof}
This follows from the fact that the derived counits of the adjunctions $\CD\dashv \ND$ and $\Lsim\dashv \bbHsim$ are level-wise weak equivalences, by \cref{thm:counit-CD-ND,thm:Hsim-rightQUillen}, respectively. 
\end{proof}

\begin{rem} \label{rem:Quillen-pair-Cat}
Let us denote by $D\colon \Cat\to \TwoCat$ the functor sending a category $\CC$ to the locally discrete $2$-category $D\CC$ with the same objects and morphisms as $\CC$ and only trivial $2$-morphisms. The functor $D$ has a left adjoint $P\colon \TwoCat\to \Cat$ given by base change along the functor $\pi_0\colon \Cat\to \Set$ sending a category to its set of connected components. By \cite[Theorem~8.2]{Lack2Cat}, these functors form a Quillen pair between the canonical model structure on $\Cat$ and Lack's model structure on $\TwoCat$, and its derived counit is level-wise an equivalence of categories. 

By composing with the Quillen pair of \cref{thm:Quillen-NDH}, we obtain a Quillen pair
\begin{tz}
\node[](Z) {$\Cat$};
\node[right of=Z,xshift=.9cm](A) {$\TwoCat$};
\node[right of=A,xshift=1.2cm](B) {$\TwoCSS$};
\draw[->] ($(A.west)+(0,.25cm)$) to [bend right=25] node[above,la] {$P$} ($(Z.east)+(0,.25cm)$);
\draw[->] ($(Z.east)-(0,.25cm)$) to [bend right=25] node[below,la]{$D$} ($(A.west)+(0,-.25cm)$);
\node[la] at ($(Z.east)!0.5!(A.west)$) {$\bot$};
\draw[->] ($(B.west)+(0,.25cm)$) to [bend right=25] node[above,la]{$\Lsim\CD$} ($(A.east)+(0,.25cm)$);
\draw[->] ($(A.east)-(0,.25cm)$) to [bend right=25] node[below,la]{$\ND\bbHsim$} ($(B.west)+(0,-.25cm)$);
\node[la] at ($(A.east)!0.5!(B.west)$) {$\bot$};
\end{tz}
between the canonical model structure on $\Cat$ and the model structure on $\sSet^{\twoDop}$ for $2$-fold complete Segal spaces, i.e., $(\infty,2)$-categories, whose derived counit is level-wise an equivalence of categories.  
\end{rem}

\subsection{\texorpdfstring{$\TwoCat$}{2Cat} is right-induced from \texorpdfstring{$\TwoCSS$}{2CSS} along \texorpdfstring{$\ND\bbHsim$}{NH-sim}} \label{subsec:rightinduced}

We now show that Lack's model structure on $\TwoCat$ is right-induced from $\TwoCSS$ along the nerve $\ND\bbHsim$. In particular, this says that the homotopy theory of $2$-categories is determined by the homotopy theory of $2$-fold complete Segal spaces through its image under $\ND\bbHsim$.

\begin{theorem} \label{thm:2catRI2css}
Lack's model structure on $\TwoCat$ is right-induced along the adjunction
\begin{tz}
\node[](A) {$\TwoCat$};
\node[right of=A,xshift=1.2cm](B) {$\TwoCSS$};
\draw[->] ($(B.west)+(0,.25cm)$) to [bend right=25] node[above,la]{$\Lsim\CD$} ($(A.east)+(0,.25cm)$);
\node at ($(B.east)-(0,4pt)$) {,};
\draw[->] ($(A.east)-(0,.25cm)$) to [bend right=25] node[below,la]{$\ND\bbHsim$} ($(B.west)+(0,-.25cm)$);
\node[la] at ($(A.east)!0.5!(B.west)$) {$\bot$};
\end{tz}
where $\TwoCSS$ denotes the model structure on $\sSet^{\twoDop}$ for $2$-fold complete Segal spaces.
\end{theorem}

\begin{proof}
It is enough to show that a $2$-functor $F$ is a weak equivalence (resp.~fibration) in $\TwoCat$ if and only if $\ND\bbHsim F$ is a weak equivalence (resp.~fibration) in $\TwoCSS$, as model structures are uniquely determined by their classes of weak equivalences and fibrations. 

Since the functor $\ND\bbHsim$ is right Quillen, it preserves fibrations. Moreover, since all objects are fibrant in $\TwoCat$, the functor $\ND\bbHsim$ also preserves weak equivalences by Ken Brown's Lemma (see \cite[Lemma 1.1.12]{Hovey}). This shows that, if $F$ is a weak equivalence (resp.~fibration) in $\TwoCat$, then $\ND\bbHsim F$ is a weak equivalence (resp.~fibration) in $\TwoCSS$.

Now let $F\colon \cA\to \cB$ be a $2$-functor such that $\ND\bbHsim F\colon \ND\bbHsim \cA\to \ND\bbHsim \cB$ is a weak equivalence in $\TwoCSS$. Since all objects are cofibrant in $\TwoCSS$, by Ken Brown's Lemma, the left Quillen functor $\Lsim\CD$ preserves weak equivalences, and the $2$-functor $\Lsim\CD\ND\bbHsim F$ is a biequivalence. We then have a commutative square 
\begin{tz}
\node[](1) {$\Lsim\CD\ND\bbHsim \cA$}; 
\node[right of=1,xshift=2cm](2) {$\Lsim\CD\ND\bbHsim \cB$}; 
\node[below of=1](3) {$\cA$}; 
\node[below of=2](4) {$\cB$}; 
\node at ($(4.east)-(0,4pt)$) {,};

\draw[->] (1) to node[above,la] {$\Lsim\CD\ND\bbHsim F$} node[below,la] {$\simeq$} (2); 
\draw[->] (1) to node[left,la] {$\epsilon_\cA$} node[right,la] {$\simeq$} (3);
\draw[->] (2) to node[right,la] {$\epsilon_\cB$} node[left,la] {$\simeq$} (4); 
\draw[->] (3) to node[below,la] {$F$} (4); 
\end{tz}
where the vertical $2$-functors are biequivalences by \cref{thm:counit-NDH}, since the components of the counit coincide with that of the derived counit as all objects in $\TwoCSS$ are cofibrant. By $2$-out-of-$3$, we get that $F$ is also a biequivalence. 

Finally, let $F\colon \cA\to \cB$ be a $2$-functor such that $\ND\bbHsim F\colon \ND\bbHsim \cA\to \ND\bbHsim \cB$ is a fibration in~$\TwoCSS$. We show that $F$ has the right lifting property with respect to the generating trivial cofibrations $j_1\colon [0]\to \Eadj$ and $j_2\colon [1]\to \Sigma I$ in $\TwoCat$ as described in \cref{not:cofin2cat}. First note that $(\ND\bbHsim F)_{m,k}$ is a fibration in $\sSet$ for all $m,k\geq 0$, since fibrations between fibrant objects in $\TwoCSS$ are in particular level-wise fibrations. 

By taking $m=k=0$, as $(\ND\bbHsim F)_{0,0}$ is a fibration in $\sSet$, there is a lift in every commutative diagram as below left.
\begin{tz}
\node[](1) {$\Reps{0}$};
\node[right of=1,xshift=1cm](2) {$(\ND\bbHsim\cA)_{0,0}$};
\node[below of=1](3) {$\Reps{1}$};
\node[below of=2](4) {$(\ND\bbHsim\cB)_{0,0}$};
\draw[->] (1) to node[left,la] {$\ell^s_{1,1}$} (3);
\draw[->] (2) to node[right,la]{$(\ND\bbHsim F)_{0,0}$} (4);
\draw[->] (1) to (2);
\draw[->] (3) to (4);
\draw[->,dashed] (3) to (2);

\node[right of=2,xshift=2cm](1) {$[0]$};
\node[right of=1,xshift=.5cm](2) {$\cA$};
\node[below of=1](3) {$\Eadj$};
\node[below of=2](4) {$\cB$};
\draw[->] (1) to node[left,la] {$j_1$} (3);
\draw[->] (2) to node[right,la]{$F$} (4);
\draw[->] (1) to (2);
\draw[->] (3) to (4);
\draw[->,dashed] (3) to (2);
\end{tz}
By \cref{desc:2cat00}, a $0$-simplex in $(\ND\bbHsim \cA)_{0,0}$ is an object of $\cA$, and a $1$-simplex in $(\ND\bbHsim \cA)_{0,0}$ is an adjoint equivalence in $\cA$. Therefore, the existence of a lift in each diagram as above left corresponds to the existence of a lift in each diagram as above right. This shows that $F$ has the right lifting property with respect to $j_1$.

Now take $m=1$ and $k=0$. As $(\ND\bbHsim \cA)_{1,0}$ is a fibration in $\sSet$, there is a lift in every commutative diagram as below left.
\begin{tz}
\node[](1) {$\Reps{0}$};
\node[right of=1,xshift=1cm](2) {$(\ND\bbHsim\cA)_{1,0}$};
\node[below of=1](3) {$\Reps{1}$};
\node[below of=2](4) {$(\ND\bbHsim\cB)_{1,0}$};
\draw[->] (1) to node[left,la] {$\ell^s_{1,1}$} (3);
\draw[->] (2) to node[right,la]{$(\ND\bbHsim F)_{1,0}$} (4);
\draw[->] (1) to (2);
\draw[->] (3) to (4);
\draw[->,dashed] (3) to (2);

\node[right of=2,xshift=2cm](1) {$[1]$};
\node[right of=1,xshift=.5cm](2) {$\cA$};
\node[below of=1](3) {$[1]\otimes_2 \Eadj$};
\node[below of=2](4) {$\cB$};
\draw[->] (1) to node[left,la] {$[1]\otimes_2 j_1$} (3);
\draw[->] (2) to node[right,la]{$F$} (4);
\draw[->] (1) to (2);
\draw[->] (3) to (4);
\draw[->,dashed] (3) to (2);
\end{tz}
By \cref{desc:2cat10}, a $0$-simplex in $(\ND\bbHsim \cA)_{1,0}$ is a morphism of $\cA$, and a $1$-simplex in $(\ND\bbHsim \cA)_{1,0}$ is an invertible $2$-morphism in $\cA$, as depicted in \cref{desc:2cat10} (1). Therefore, the existence of a lift in each diagram as above left corresponds to the existence of a lift in each diagram as above right. 

Now we show that the generating trivial cofibration $j_2\colon [1]\to \Sigma I$ is a retract of the $2$-functor $[1]\otimes_2 j_1$ of the following form
\begin{tz}
\node[](1) {$[1]$}; 
\node[right of=1,xshift=.5cm](2) {$[1]$}; 
\node[right of=2,xshift=.5cm](3) {$[1]$}; 
\draw[d] (1) to (2); 
\draw[d] (2) to (3); 
\node[below of=1](4) {$\Sigma I$}; 
\node[below of=2](5) {$[1]\otimes_2 \Eadj$}; 
\node[below of=3](6) {$\Sigma I$}; 
\node at ($(6.east)-(0,4pt)$) {.};
\draw[->] (1) to node[left,la] {$j_2$} (4);
\draw[->] (2) to node[left,la] {$[1]\otimes_2 j_1$} (5);
\draw[->] (3) to node[right,la] {$j_2$} (6);

\draw[->] (4) to node[below,la] {$i$} (5);
\draw[->] (5) to node[below,la] {$r$} (6);
\end{tz}
If we denote the data of the $2$-categories $\Sigma I$ and $[1]\otimes_2 \Eadj$ as below left and right, respectively,
\begin{tz}
\node[](A) {$x$};
\node[right of=A,xshift=.5cm](B) {$y$};
\draw[->,bend left] (A.north east) to node(a)[above,la] {$f$} (B.north west);
\draw[->,bend right] (A.south east) to node(b)[below,la] {$f'$} (B.south west);
\cell[la,right]{a}{b}{$\cong$};

    \node[right of=B,xshift=2cm,yshift=.75cm](A) {$0$};
    \node[right of=A](B) {$1$};
    \node[below of=A](A') {$0'$};
    \node[right of=A'](B') {$1'$};
    \draw[->] (A) to node[above,la]{$\simeq$}  (B);
    \draw[->] (A') to node[below,la] {$\simeq$} (B');
    \draw[->] (A) to (A');
    \draw[->] (B) to (B');
    
    \cell[la,above,xshift=-.2cm]{B}{A'}{$\cong$};
\end{tz}
then the $2$-functor $i\colon \Sigma I\to [1]\otimes_2\Eadj$ is given by sending the object $x$, resp.~$y$, to the object $0$, resp.~$1'$; the morphism $f$, resp.~$f'$, to the composite $0\to 1\to 1'$, resp.~${0\to 0'\to 1'}$; and the invertible $2$-morphism of $\Sigma I$ to the invertible $2$-morphism of $[1]\otimes_2\Eadj$. On the other hand, the $2$-functor $r\colon [1]\otimes_2\Eadj\to \Sigma I$ is given by sending the objects $0,1$, resp.~$0',1'$, to the object $x$, resp.~$y$; the morphism $1\to 1'$, resp.~$0\to 0'$, to the morphism $f$, resp.~$f'$; the adjoint equivalences of $\Eadj$ to identities; and the invertible $2$-morphism of $[1]\otimes_2\Eadj$ to the invertible $2$-morphism of $\Sigma I$. 

Therefore, since $F$ has the right lifting property with respect to~$[1]\otimes_2 j_1$, then $F$ also has the right lifting property with respect to $j_2$. This shows that~$F$ is a fibration in $\TwoCat$ and concludes the proof.  
\end{proof}

\subsection{Comparison between the nerves \texorpdfstring{$\ND\bbH$}{NH} and \texorpdfstring{$\ND\bbHsim$}{NH-sim}} \label{subsec:NHvsNHsim}

We now want to compare the nerves $\ND\bbH\cA$ and $\ND\bbHsim\cA$ of a $2$-category $\cA$. For this, we will construct a homotopy equivalence between the spaces $(\ND\bbH\cA)_{m,k}$ and $(\ND\bbHsim\cA)_{m,k}$. Their sets of $n$-simplices are given by 
\[ (\ND\bbH\cA)_{m,k,n}=\DblCat(\XD_{m,k,n},\bbH\cA)\cong \TwoCat(L\XD_{m,k,n},\cA) \]
and 
\[ (\ND\bbHsim\cA)_{m,k,n}=\DblCat(\XD_{m,k,n},\bbHsim\cA)\cong \TwoCat(\Lsim\XD_{m,k,n},\cA). \]
Let us first describe the $2$-categories $\Lsim\XD_{m,k,n}$ and $L\XD_{m,k,n}$. 

\begin{descr} \label{descr:LX}
The $2$-category $L\XD_{m,k,n}$ is obtained from the double category 
\[ \XD_{m,k,n}=(\VK{k}\otimes\OM{m})\otimes \ON{n} \]
by identifying the objects $(x,y,z)\sim (x,y',z)$, for all $0\leq x\leq m$, $0\leq y,y'\leq k$, and ${0\leq z\leq n}$, and by identifying the vertical morphisms $(x,g,z)\colon (x,y,z)\arrowdot (x,y',z)$, where $g\in\OM{k}(y,y')$, with the identity at $(x,y,z)\sim (x,y',z)$. We denote by $[x,z]$ the equivalence class $\{(x,y,z)\mid 0\leq y\leq k\}$. Then, the $2$-category $L\XD_{m,k,n}$ has 
\begin{itemize}
    \item objects $[x,z]$ for all $0\leq x\leq m$ and $0\leq z\leq n$, 
    \item morphisms freely generated by
    \begin{itemize}
        \item a morphism $(f,y,z)\colon [x,z]\to [x',z]$ where $f\in\OM{m}(x,x')$ is represented by the set $\{x,x'\}$, for all $0\leq x<x'\leq m$, $0\leq y\leq k$, and $0\leq z\leq n$,
        \item a morphism $(x,y,h)\colon [x,z]\to [x,z']$ where $h\in\ON{n}(z,z')$ is represented by the set $\{z,z'\}$, for all $0\leq x\leq m$, $0\leq y\leq k$, and $0\leq z,z'\leq n$ with $z\neq z'$,
    \end{itemize}
    \item $2$-morphisms are generated by 
    \begin{itemize}
        \item a $2$-morphism $\alpha\colon \overline{f}\Rightarrow \overline{f'}$ for each square $\sq{\alpha}{\overline{f}}{\overline{f}'}{u}{v}$ in $\XD_{m,k,n}$
    \end{itemize} 
    subject to the minimal relations making the projection $\XD_{m,k,n}\to \bbH L\XD_{m,k,n}$ into a double functor. Here the projection $\XD_{m,k,n}\to \bbH L\XD_{m,k,n}$ sends an object $(x,y,z)$ to the object $[x,z]$, horizontal morphisms $(f,y,z)$ and $(x,y,h)$ to the morphisms $(f,y,z)$ and $(x,y,h)$, vertical morphisms $(x,u,z)$ to the identity at $[x,z]$, and squares $\sq{\alpha}{\overline{f}}{\overline{f}'}{u}{v}$ to the corresponding $2$-morphism $\alpha\colon \overline{f}\Rightarrow \overline{f'}$.
\end{itemize}
\end{descr}

\begin{descr} \label{descr:LsimX}
The $2$-category $\Lsim\XD_{m,k,n}$ has 
\begin{itemize}
    \item the same objects as the double category $\XD_{m,k,n}=(\VK{k}\otimes\OM{m})\otimes \ON{n}$, i.e., triples $(x,y,z)$ with $0\leq x\leq m$, $0\leq y\leq k$, $0\leq z\leq n$, 
    \item morphisms freely generated by 
    \begin{itemize}
        \item a morphism $(f,y,z)\colon (x,y,z)\to (x',y,z)$ where $f\in\OM{m}(x,x')$ is represented by the set $\{x,x'\}$, for all $0\leq x<x'\leq m$, $0\leq y\leq k$, and $0\leq z\leq n$,
        \item a morphism $(x,y,h)\colon (x,y,z)\to (x,y,z')$ where $h\in\ON{n}(z,z')$ is represented by the set $\{z,z'\}$, for all $0\leq x\leq m$, $0\leq y\leq k$, and $0\leq z,z'\leq n$ with $z\neq z'$,
        \item an adjoint equivalence $(x,g,z)\colon (x,y,z)\xrightarrow{\simeq} (x,y',z)$ where $g\in \OM{k}(y,y')$ is represented by the set $\{y,y'\}$, for all $0\leq x\leq m$, $0\leq y<y'\leq k$, and $0\leq z\leq n$,
    \end{itemize}
    \item $2$-morphisms are generated by 
    \begin{itemize}
        \item a $2$-morphism $\alpha\colon v\overline{f}\Rightarrow \overline{f}'u$ for each square $\sq{\alpha}{\overline{f}}{\overline{f}'}{u}{v}$ in $\XD_{m,k,n}$, 
    \end{itemize}
    subject to relations which are equivalent to requiring that the projection $2$-functor $\pi_{m,k,n}\colon \Lsim\XD_{m,k,n}\to L\XD_{m,k,n}$ is fully faithful on $2$-morphisms. Here the comparison $2$-functor $\pi_{m,k,n}\colon \Lsim\XD_{m,k,n}\to L\XD_{m,k,n}$ sends an object $(x,y,z)$ to the object $[x,z]$, morphisms $(f,y,z)$ and $(x,y,h)$ to the morphisms $(f,y,z)$ and $(x,y,h)$, adjoint equivalences $(x,g,z)$ to the identity at $[x,z]$, and $2$-morphisms $\alpha\colon v\overline{f}\Rightarrow \overline{f}'u$ to the corresponding $2$-morphism $\alpha\colon \overline{f}\Rightarrow \overline{f}'$.
\end{itemize}
\end{descr}

\begin{ex} \label{ex:LvsLsim}
We compute these $2$-categories in the case where $m=1$, $k=1$, and $n=0$. Let us denote by $u\colon 0'\arrowdot 1'$ the vertical morphism in $\bbV[1]$ and by $f\colon 0\to 1$ the morphism in $[1]$. We have that $L(\bbV[1]\otimes [1])$ is the free $2$-category on a $2$-morphism as depicted below left, while $\Lsim(\bbV[1]\otimes [1])$ is the $2$-category as depicted below right. We omit the $z$-component here since it is always $0$.
\begin{tz}
\node[](A) {$[0]$};
\node[right of=A,xshift=.5cm](B) {$[1]$};
\draw[->,bend left] (A.north east) to node(a)[pos=0.35]{} node[above,la] {$(f,0')$} (B.north west);
\draw[->,bend right] (A.south east) to node(b)[pos=0.35]{} node[below,la] {$(f,1')$} (B.south west);
\cell[la,right]{a}{b}{$(f,u)$};

\node[right of=B,yshift=.9cm,xshift=1cm](1) {$(0,0')$};
\node[right of=1,xshift=.5cm](2) {$(1,0')$};
\node[below of=1,yshift=-.3cm](3) {$(0,1')$};
\node[below of=2,yshift=-.3cm](4) {$(1,1')$};
\draw[->] (1) to node[above,la]{$(f,0')$} (2);
\draw[->] (3) to node[below,la]{$(f,1')$} (4);
\draw[->] (1) to node[left,la]{$(0,u)$} node[right,la] {$\simeq$} (3);
\draw[->] (2) to  node[left,la] {$\simeq$}node[right,la]{$(1,u)$} (4);
\cell[la,below,yshift=-3pt,xshift=8pt]{2}{3}{$(f,u)$};
\end{tz}
\end{ex}

\begin{rem}
Using these descriptions, we can see that the $0$-simplices of the simplicial sets $(\ND\bbH\cA)_{0,0}$ and $(\ND\bbHsim\cA)_{0,0}$ are the objects of $\cA$, and the ones of $(\ND\bbH\cA)_{1,0}$ and $(\ND\bbHsim\cA)_{1,0}$ the morphisms of $\cA$. The $0$-simplices in $(\ND\bbH\cA)_{1,1}$ are the $2$-morphisms of $\cA$ as in the left-hand diagram of \cref{ex:LvsLsim}, while the ones of $(\ND\bbHsim\cA)_{1,1}$ are the $2$-morphisms of~$\cA$ as in the right-hand diagram of \cref{ex:LvsLsim}. Finally, the $0$-simplices in $(\ND\bbH\cA)_{0,1}$ are just objects of~$\cA$, while the ones of $(\ND\bbHsim\cA)_{0,1}$ are adjoint equivalences in $\cA$. We describe these simplicial sets in greater detail in \cref{subsec:descr-NH,subsec:descr-NHsim}.
\end{rem}

Recall the comparison $2$-functor $\pi_{m,k,n}\colon \Lsim\XD_{m,k,n}\to L\XD_{m,k,n}$ introduced at the end of \cref{descr:LsimX}.
Then this $2$-functor is clearly surjective on objects, full on morphisms, and fully faithful on $2$-morphisms. By constructing an inverse $2$-functor up to pseudo-natural equivalence to this comparison $2$-functor $\pi_{m,k,n}$, we obtain the following result. 

\begin{theorem} \label{thm:homotopy}
Let $\cA$ be a $2$-category. The map $\pi^*\colon \ND\bbH\cA\to \ND\bbHsim \cA$ induced by the comparison $2$-functors $\pi_{m,k,n}\colon \Lsim\XD_{m,k,n}\to L\XD_{m,k,n}$ is level-wise a homotopy equivalence in $\sSet^{\twoDop}$. In particular, this exhibits $\ND\bbHsim \cA$ as a fibrant replacement of $\ND\bbH\cA$ in $\TwoCSS$ (or in $\DblInfH$).
\end{theorem}

\begin{proof} 
We first construct an inverse $2$-functor up to pseudo-natural equivalence
\[ \iota_{m,k,n}\colon L\XD_{m,k,n}\to \Lsim\XD_{m,k,n} \]
to the $2$-functor $\pi_{m,k,n}$ such that the composite $\pi_{m,k,n}\iota_{m,k,n}$ is the identity at $L\XD_{m,k,n}$. It sends an object $[x,z]$ to the object $(x,0,z)$, a generating morphism $(f,y,z)\colon [x,z]\to [x',z]$ with $f\in \OM{m}(x,x')$ represented by the set $\{x,x'\}$ to the composite 
\[ (x,0,z)\xrightarrow[\simeq]{(x,g,z)} (x,y,z)\xrightarrow{(f,y,z)} (x',y,z) \xrightarrow[\simeq]{(x',g',z)} (x',0,z), \]
and a generating morphism $(x,y,h)\colon [x,z]\to [x,z']$ with $h\in \ON{n}(z,z')$ represented by the set $\{z,z'\}$ to the composite 
\[ (x,0,z)\xrightarrow[\simeq]{(x,g,z)} (x,y,z)\xrightarrow{(x,y,h)} (x,y,z') \xrightarrow[\simeq]{(x,g',z')} (x,0,z'), \]
where $g\in \ON{k}(0,y)$ is represented by the set $\{0,y\}$ and $g'\in \ON{k}(y,0)$ is its weak inverse. The assignment on $2$-morphisms is uniquely determined by these assignments on objects and morphisms, since the $2$-functor $\pi_{m,k,n}$ is fully faithful on $2$-morphisms and we imposed that $\pi_{m,k,n}\iota_{m,k,n}=\id_{L\XD_{m,k,n}}$. In particular, since the morphisms in the $2$-category $L\XD_{m,k,n}$ are freely generated by the morphisms $(f,y,z)$ and $(x,y,h)$, this defines a $2$-functor $\iota_{m,k,n}\colon L\XD_{m,k,n}\to \Lsim\XD_{m,k,n}$. 

We now construct a pseudo-natural adjoint equivalence \[ \epsilon_{m,k,n}\colon \iota_{m,k,n}\pi_{m,k,n}\Rightarrow \id_{\Lsim\XD_{m,k,n}}. \] At an object $(x,y,z)\in \Lsim\XD_{m,k,n}$, we define $\epsilon_{(x,y,z)}$ to be the morphism
\[ \epsilon_{(x,y,z)}\coloneqq (x,g,z)\colon (x,0,z)\stackrel{\simeq}{\longrightarrow} (x,y,z), \]
 where $g\in \ON{k}(0,y)$ is represented by the set $\{0,y\}$. Note that the morphism $\epsilon_{(x,y,z)}$ as defined above is an adjoint equivalence. Given a morphism $(f,y,z)\colon (x,y,z)\to (x',y,z)$, we define $\epsilon_{(f,y,z)}$ to be the following invertible $2$-morphism 
\begin{tz}
\node[](1) {$(x,0,z)$};
\node[right of=1,xshift=2cm](5) {$(x,y,z)$};
\node[below of=1](2) {$(x,y,z)$};
\node[below of=2](3) {$(x',y,z)$};
\node[below of=3](4) {$(x',0,z)$};
\node[right of=4,xshift=2cm](6) {$(x',y,z)$};
\draw[->](1) to node[above,la]{$\epsilon_{(x,y,z)}=(x,g,z)$} node[below,la]{$\simeq$} (5);
\draw[->](1) to node[left,la]{$(x,g,z)$} node[right,la]{$\simeq$} (2);
\draw[->](2) to node[left,la]{$(f,y,z)$} (3);
\draw[->](3) to node[left,la]{$(x',g',z)$} node[right,la]{$\simeq$} (4);
\draw[->](4) to node[below,la,yshift=-2pt]{$\epsilon_{(x',y,z)}=(x',g,z)$} node[above,la]{$\simeq$} (6);
\draw[->](5) to node[right,la]{$(f,y,z)$} (6);
\draw[d] (3) to node(a)[pos=0.3]{} (6);
\cell[la,right,xshift=2pt,yshift=-3pt][n][0.6]{4.east}{a}{$\cong$};
\node[la] at ($(1)!0.5!(6)$) {$=$};
\end{tz}
induced by the counit $gg'\cong \id_y$ of the adjoint equivalence $(g,g')$. We define $\epsilon_{(x,y,h)}$ for a morphism $(x,y,h)\colon (x,y,z)\to (x,y,z')$  similarly. This defines a pseudo-natural adjoint equivalence $\epsilon_{m,k,n}\colon \iota_{m,k,n}\pi_{m,k,n}\Rightarrow \id_{\Lsim\XD_{m,k,n}}$, which can be represented by a $2$-functor $\ON{1}\to [\Lsim\XD_{m,k,n},\Lsim\XD_{m,k,n}]_{2,\ps}$ since it corresponds to an adjoint equivalence in the pseudo-hom $2$-category. By definition of the Gray tensor product~$\otimes_2$ (see \cref{def:otimes2}), this pseudo-natural adjoint equivalence can equivalently be seen as a $2$-functor 
\begin{tz}
\node[](1) {$\Lsim\XD_{m,k,n}\otimes_2\ON{1}$};
\node[below of=1](2) {$\Lsim\XD_{m,k,n}$};
\node[above of=1](3) {$\Lsim\XD_{m,k,n}$};
\node[right of=1,xshift=3cm](4) {$\Lsim\XD_{m,k,n}$};
\node at ($(4.east)-(0,4pt)$) {.};

\draw[->] (2) to node[left,la] {$\id \otimes_2 d^1$} (1);
\draw[->] (3) to node[left,la] {$\id \otimes_2 d^0$} (1);
\draw[->] (1) to node[over,la] {$\epsilon_{m,k,n}$} (4);
\draw[->] (3) to node[above,la,xshift=.8cm] {$\iota_{m,k,n}\circ \pi_{m,k,n}$} (4);
\draw[d] (2) to (4);
\end{tz} 

We claim that these $2$-functors $\epsilon_{m,k,n}$ induce a homotopy $\epsilon^*_{m,k}$ as in
\begin{tz}
\node[](1) {$(\ND\bbHsim\cA)_{m,k}\times \Reps{1}$};
\node[below of=1](2) {$(\ND\bbHsim\cA)_{m,k}$};
\node[above of=1](3) {$(\ND\bbHsim\cA)_{m,k}$};
\node[right of=1,xshift=3cm](4) {$(\ND\bbHsim\cA)_{m,k}$};
\node at ($(4.east)-(0,4pt)$) {,};

\draw[->] (2) to node[left,la] {$\id \times d^1$} (1);
\draw[->] (3) to node[left,la] {$\id \times d^0$} (1);
\draw[->] (1) to node[over,la] {$\epsilon^*_{m,k}$} (4);
\draw[->] (3) to node[above,la,xshift=.8cm] {$\pi^*_{m,k}\circ \iota^*_{m,k}$} (4);
\draw[d] (2) to (4);
\end{tz} 
where the $n$th component of $\epsilon^*_{m,k}$ is obtained by applying the functor $\TwoCat(-,\cA)$ to $\epsilon_{m,k,n}$, for all $n\geq 0$. 

For each $F\in(\ND\bbHsim\cA)_{m,k,n}$, we want to describe the corresponding $(\Reps{n}\times \Reps{1})$-prism of the homotopy, which coincide with $F\iota_{m,k,n}\pi_{m,k,n}$ at $0\in \Reps{1}$ and with $F$ at $1\in \Reps{1}$. Note that a $(\Reps{n}\times \Reps{1})$-prism in $(\ND\bbHsim\cA)_{m,k}$ corresponds to a $2$-functor 
\[ \Lsim((\VK{k}\otimes \OM{m})\otimes (\ON{n}\otimes_2 \ON{1}))\longrightarrow \cA. \]
The squares induced by vertical morphisms in $\VK{k}$ and morphisms in $\ON{1}$ must be weakly horizontally invertible in $(\VK{k}\otimes \OM{m})\otimes (\ON{n}\otimes_2 \ON{1})$, since the morphisms in $\ON{1}$ are adjoint equivalences. It follows from \cref{lem:whiinLsim} that the corresponding $2$-morphisms in $\Lsim((\VK{k}\otimes \OM{m})\otimes (\ON{n}\otimes_2 \ON{1}))$ are invertible and therefore, by \cref{rem:noninv}, we get that
\begin{align*}
    \Lsim((\VK{k}&\otimes \OM{m})\otimes (\ON{n}\otimes_2 \ON{1})) \\ &\cong \Lsim((\VK{k}\otimes \OM{m})\otimes \ON{n})\otimes_2 \ON{1} =\Lsim\XD_{m,k,n}\otimes_2 \ON{1}.
\end{align*} 
This says that a $(\Reps{n}\times \Reps{1})$-prism in $(\ND\bbHsim\cA)_{m,k}$ corresponds to a $2$-functor \[ \Lsim\XD_{m,k,n}\otimes_2 \ON{1}\to \cA. \]
We can therefore define the component of the homotopy at $F\in (\ND\bbHsim\cA)_{m,k,n}$ to be $F \epsilon_{m,k,n}$. This shows the claim.

Since $\iota^*_{m,k}\circ \pi^*_{m,k}=\id_{(\ND\bbH\cA)_{m,k}}$ and by the above homotopy, we see that $\iota^*_{m,k}$ and $\pi^*_{m,k}$ give a homotopy equivalence between $(\ND\bbH\cA)_{m,k}$ and $(\ND\bbHsim\cA)_{m,k}$, for all $m,k\geq 0$. These assemble into maps $\iota^*$ and $\pi^*$ of $\sSet^{\twoDop}$ which give a level-wise weak equivalence between $\ND\bbH\cA$ and $\ND\bbHsim\cA$. This is in particular a weak equivalence in $\TwoCSS$ and in $\DblInfH$. Since $\ND\bbHsim\cA$ is fibrant in $\TwoCSS$ and in $\DblInfH$, we conclude that it gives a fibrant replacement of $\ND\bbH\cA$. 
\end{proof}

\begin{rem}
    Note that the map $\pi^*\colon \ND\bbH\cA\to \ND\bbHsim\cA$ is also a monomorphism, hence $\pi^*$ is in fact level-wise a trivial cofibration in $\sSet^{\twoDop}$.
\end{rem}

\begin{rem}
Recall from \cref{rem:Quillen-pair-Cat} the Quillen pair $P\dashv D$ between $\Cat$ and $\TwoCat$ and let $\CC$ be a category. The nerve of the double category $\bbH D\CC$ is given by
\[    (\ND\bbH D\CC)_{m,k,n}=\TwoCat(L\XD_{m,k,n},D\CC) \cong \Cat(PL\XD_{m,k,n},\CC), \]
for all $m,k,n\geq 0$. By applying the functor $P$ to the $2$-category $L\XD_{m,k,n}$ as given in \cref{descr:LX}, we can see that $PL\XD_{m,k,n}\cong [m]\times I[n]$, where $I[n]$ is the category with object set $\{0,\ldots,n\}$ and a unique isomorphism between any two objects. Therefore, 
\[ (\ND\bbH D\CC)_{m,k,n}\cong \Cat([m]\times I[n],\CC) =N_{\mathrm{Rezk}}(\CC)_{m,n} \]
is given by the Rezk nerve (defined in \cite[\S 3.5]{Rezk}) constant in the vertical direction. On the other hand, the nerve of the double category $\bbHsim D\CC$ is given by
\begin{align*}
    (\ND\bbHsim D\CC)_{m,k,n}&=\TwoCat(\Lsim\XD_{m,k,n},D\CC) \cong \Cat(P\Lsim\XD_{m,k,n},\CC) \\
    &    \cong \Cat((I[k]\times [m])\times I[n], \CC),
\end{align*} 
for all $m,k,n\geq 0$. Then, by \cref{thm:homotopy}, there is a level-wise homotopy equivalence $\ND\bbH D\CC\to \ND\bbHsim D\CC$ which exhibits $\ND\bbHsim D\CC$ as a fibrant replacement of the Rezk nerve of~$\CC$ in $\TwoCSS$ (or $\DblInfH$).
\end{rem}

\appendix
\section{Weakly horizontally invertible squares} \label{app:weakhorinv}

In this first appendix, we give some technical results about weakly horizontally invertible squares, which will be of use to describe the nerves in low dimensions in \cref{app:nerveinlow}. 
These results also find their utility in the papers \cite{MSVfirst,MSVsecond} by the author, Sarazola, and Verdugo. Some of the lemmas presented here (\cref{lem:uniqueweakinv,lem:whiiffvi,lem:pseudoeq}) were also proven independently in another context by proven by Grandis and Par\'e in \cite{GraPar19} -- their terminology for weakly horizontally invertible squares is that of \emph{equivalence cells}. In \cref{subsec:uniquewhi}, we first prove that the weak inverse of a weakly horizontally invertible square is unique when one first fixes horizontal \emph{adjoint} equivalence data. In \cref{subsec:specialwhi}, we consider weakly horizontally invertible squares of special forms and characterize them. Finally, in \cref{subsec:ps-equ}, we give a definition of horizontal pseudo-natural transformations and modifications, which correspond to the morphisms and $2$-morphisms in the pseudo-hom $2$-categories $\bfH[-,-]_\ps$ of the $\TwoCat$-enrichment of $\DblCat$ given in \cref{def:tensor}. We then characterize the equivalences in these pseudo-hom $2$-categories.

\subsection{Unique inverse lemma} \label{subsec:uniquewhi}

We first show the existence and uniqueness of a weak inverse for a weakly horizontally invertible square with respect to fixed horizontal adjoint equivalence data. 

\begin{lemme} \label{lem:uniqueweakinv}
Let $\sq{\alpha}{f}{f'}{u}{v}$ be a weakly horizontally invertible square in a double category $\bA$. Suppose $(f,g,\eta,\epsilon)$ and $(f',g',\eta',\epsilon')$ are horizontal adjoint equivalences. Then there is a unique square $\sq{\beta}{g}{g'}{v}{u}$ in $\bA$ which is the weak inverse of $\alpha$ with respect to the horizontal adjoint equivalence data $(f,g,\eta,\epsilon)$ and $(f',g',\eta',\epsilon')$. 
\end{lemme}

\begin{proof}
Since $\alpha$ is weakly horizontally invertible, by definition, there is a weak inverse $\gamma$ of $\alpha$ with respect to horizontal adjoint equivalence data $(f,h,\mu,\delta)$ and $(f',h',\mu',\delta')$. We define $\beta$ to be given by the following pasting. 
\begin{tz}
\node[](1) {$B$};
\node[right of=1](2) {$A$};
\node[right of =2,rr](3) {$A$};
\draw[->] (1) to node[above,la]{$g$} (2);
\draw[d] (2) to (3);

\node[below of=1](4) {$B$};
\node[right of=4](5) {$A$};
\node[right of=5](6) {$B$};
\node[right of=6](7) {$A$};
\draw[d,pro] (1) to (4);
\draw[d,pro] (2) to node(a)[]{} (5);
\draw[d,pro] (3) to node(b)[]{} (7);
\draw[->] (4) to node[over,la]{$g$} (5);
\draw[->] (5) to node[over,la]{$f$} (6);
\draw[->] (6) to node[over,la]{$h$} (7);

\node[la] at ($(1)!0.5!(5)$) {$e_g$};
\node[la] at ($(a)!0.45!(b)$) {$\mu$};
\node[la] at ($(a)!0.55!(b)$) {$\vcong$};

\node[below of=4](1) {$B$};
\node[right of=1,rr](2) {$B$};
\node[right of =2](3) {$A$};
\draw[->] (2) to node[over,la]{$h$} (3);
\draw[d] (1) to (2);
\draw[d,pro] (4) to node(a)[]{} (1);
\draw[d,pro] (6) to node(b)[]{} (2);
\draw[d,pro] (7) to (3);

\node[la] at ($(6)!0.5!(3)$) {$e_h$};
\node[la] at ($(a)!0.45!(b)$) {$\epsilon$};
\node[la] at ($(a)!0.55!(b)$) {$\vcong$};

\node[left of=1,xshift=-.5cm](A) {$A$};
\node[left of=A](B) {$B$};
\node[below of=B](B') {$B'$};
\node[below of=A](A') {$A'$};
\draw[->] (B) to node[above,la] {$g$} (A);
\draw[->,pro] (B) to node[left,la] {$v$} (B');
\draw[->] (B') to node[below,la] {$g'$} (A');
\draw[->,pro] (A) to node[right,la] {$u$} (A');

\node[la] at ($(B)!0.5!(A')$) {$\beta$};
\node[la] at ($(A')!0.5!(1)$) {$=$};

\node[below of=1](4) {$B'$};
\node[right of=4,rr](5) {$B'$};
\node[right of=5](6) {$A'$};
\draw[->] (5) to node[over,la]{$h'$} (6);
\draw[d] (4) to (5);
\draw[->,pro] (1) to node[left,la]{$v$} (4);
\draw[->,pro] (2) to node[left,la]{$v$} (5);
\draw[->,pro] (3) to node[right,la]{$u$} (6);

\node[la] at ($(1)!0.5!(5)$) {$\id_v$};
\node[la] at ($(2)!0.5!(6)$) {$\gamma$};

\node[below of=4](1) {$B'$};
\node[right of=1](2) {$A'$};
\node[right of=2](3) {$B'$};
\node[right of=3](9) {$A'$};
\draw[d,pro] (4) to node(a)[]{} (1);
\draw[d,pro] (5) to node(b)[]{} (3);
\draw[d,pro] (6) to (9);
\draw[->] (1) to node[over,la]{$g'$} (2);
\draw[->] (2) to node[over,la]{$f'$} (3);
\draw[->] (3) to node[over,la]{$h'$} (9);

\node[la] at ($(5)!0.5!(9)$) {$e_{h'}$};
\node[la] at ($(a)!0.4!(b)$) {$(\epsilon')^{-1}$};
\node[la] at ($(a)!0.6!(b)$) {$\vcong$};

\node[below of=1](4) {$B'$};
\node[right of=4](5) {$A'$};
\node[right of=5,rr](6) {$A'$};
\draw[->] (4) to node[below,la]{$g'$} (5);
\draw[d] (5) to (6);
\draw[d,pro] (1) to (4);
\draw[d,pro] (2) to node(a)[]{} (5);
\draw[d,pro] (9) to node(b)[]{} (6);

\node[la] at ($(1)!0.5!(5)$) {$e_{g'}$};
\node[la] at ($(a)!0.4!(b)$) {$(\mu')^{-1}$};
\node[la] at ($(a)!0.6!(b)$) {$\vcong$};
\end{tz}
We check that $\beta$ is a weak inverse of $\alpha$ with respect to the horizontal adjoint equivalence data $(f,g,\eta,\epsilon)$ and $(f',g',\eta',\epsilon')$. We have that 
\begin{tz}
\node[](0) {$A$};
\node[right of=0,rr](1) {$A$};
\node[below of=0](2) {$A$};
\node[right of=2](3) {$B$};
\node[right of=3](4) {$A$};
\node[below of=2](5) {$A'$};
\node[right of=5](6) {$B'$};
\node[right of=6](7) {$A'$};
\node[below of=5](8) {$A'$};
\node[right of=8,rr](9) {$A'$};
\draw[d] (0) to (1);
\draw[d] (8) to (9);
\draw[d,pro] (0) to node(a)[]{} (2);
\draw[d,pro] (1) to node(b)[]{} (4);
\draw[d,pro] (5) to node(c)[]{} (8);
\draw[d,pro] (7) to node(d)[]{} (9);
\draw[->] (2) to node[over,la]{$f$} (3);
\draw[->] (3) to node[over,la]{$g$} (4);
\draw[->] (5) to node[over,la]{$f'$} (6);
\draw[->] (6) to node[over,la]{$g'$} (7);
\draw[->,pro] (2) to node[left,la] {$u$} (5);
\draw[->,pro] (3) to node[left,la] {$v$} (6);
\draw[->,pro] (4) to node[right,la] {$u$} (7);

\node[la] at ($(a)!0.45!(b)$) {$\eta$};
\node[la] at ($(a)!0.55!(b)$) {$\vcong$};
\node[la] at ($(2)!0.5!(6)$) {$\alpha$};
\node[la] at ($(3)!0.5!(7)$) {$\beta$};
\node[la] at ($(c)!0.4!(d)$) {$(\eta')^{-1}$};
\node[la] at ($(c)!0.6!(d)$) {$\vcong$};

\node[right of=4,xshift=.5cm](0) {$A$};
\node[la] at ($(7)!0.5!(0)$) {$=$};
\node[right of=0](1') {$B$};
\node[right of=1',rr](2') {$B$};
\node[right of=2'](3') {$A$};
\node[below of=0](4) {$A'$};
\node[right of=4](5) {$B'$};
\node[right of=5,rr](6) {$B'$};
\node[right of=6](7) {$A'$};
\draw[->] (0) to node[over,la]{$f$} (1');
\draw[->] (2') to node[over,la]{$h$} (3');
\draw[->] (4) to node[over,la]{$f'$} (5);
\draw[->] (6) to node[over,la]{$h'$} (7);
\draw[d] (1') to (2');
\draw[d] (5) to (6);
\draw[->,pro] (0) to node[left,la] {$u$} (4);
\draw[->,pro] (1') to node[right,la] {$v$} (5);
\draw[->,pro] (2') to node[left,la] {$v$} (6);
\draw[->,pro] (3') to node[right,la] {$u$} (7);

\node[la] at ($(0)!0.5!(5)$) {$\alpha$};
\node[la] at ($(2')!0.5!(7)$) {$\gamma$};
\node[la] at ($(1')!0.5!(6)$) {$\id_v$};

\node[below of=4](1) {$A'$};
\node[right of=1](2) {$B'$};
\node[right of=2](3) {$A'$};
\node[right of=3](8) {$B'$};
\node[right of=8](9) {$A'$};
\draw[->] (1) to node[over,la]{$f'$} (2);
\draw[->] (2) to node[over,la]{$g'$} (3);
\draw[->] (3) to node[over,la]{$f'$} (8);
\draw[->] (8) to node[over,la]{$h'$} (9);
\draw[d,pro] (4) to (1);
\draw[d,pro] (5) to node(a)[]{} (2);
\draw[d,pro] (6) to node(b)[]{} (8);
\draw[d,pro] (7) to (9);

\node[la] at ($(a)!0.4!(b)$) {$(\epsilon')^{-1}$};
\node[la] at ($(a)!0.6!(b)$) {$\vcong$};
\node[la] at ($(4)!0.5!(2)$) {$e_{f'}$};
\node[la] at ($(6)!0.5!(9)$) {$e_{h'}$};

\node[below of=1](4) {$A'$};
\node[right of=4,rr](5) {$A'$};
\node[right of=5,rr](6) {$A'$};
\draw[d] (4) to (5);
\draw[d] (5) to (6);
\draw[d,pro] (1) to node(a)[]{} (4);
\draw[d,pro] (3) to node(b)[]{} (5);
\draw[d,pro] (9) to node(c)[]{} (6);

\node[la] at ($(a)!0.4!(b)$) {$(\eta')^{-1}$};
\node[la] at ($(a)!0.6!(b)$) {$\vcong$};
\node[la] at ($(b)!0.4!(c)$) {$(\mu')^{-1}$};
\node[la] at ($(b)!0.6!(c)$) {$\vcong$};

\node[above of=0](1) {$A$};
\node[right of=1](2) {$B$};
\node[right of=2](3) {$A$};
\node[right of=3](8) {$B$};
\node[right of=8](9) {$A$};
\draw[->] (1) to node[over,la]{$f$} (2);
\draw[->] (2) to node[over,la]{$g$} (3);
\draw[->] (3) to node[over,la]{$f$} (8);
\draw[->] (8) to node[over,la]{$h$} (9);
\draw[d,pro] (1) to (0);
\draw[d,pro] (2) to node(a)[]{} (1');
\draw[d,pro] (8) to node(b)[]{} (2');
\draw[d,pro] (9) to (3');

\node[la] at ($(a)!0.45!(b)$) {$\epsilon$};
\node[la] at ($(a)!0.55!(b)$) {$\vcong$};
\node[la] at ($(1)!0.5!(1')$) {$e_{f}$};
\node[la] at ($(8)!0.5!(3')$) {$e_{h}$};

\node[above of=1](4) {$A$};
\node[right of=4,rr](5) {$A$};
\node[right of=5,rr](6) {$A$};
\draw[d] (4) to (5);
\draw[d] (5) to (6);
\draw[d,pro] (4) to node(a)[]{} (1);
\draw[d,pro] (5) to node(b)[]{} (3);
\draw[d,pro] (6) to node(c)[]{} (9);

\node[la] at ($(a)!0.45!(b)$) {$\eta$};
\node[la] at ($(a)!0.55!(b)$) {$\vcong$};
\node[la] at ($(b)!0.45!(c)$) {$\mu$};
\node[la] at ($(b)!0.55!(c)$) {$\vcong$};
\end{tz}
\begin{tz}
\node[](0) {$A$};
\node[right of=0,rr](1) {$A$};
\node[below of=0](2) {$A$};
\node[right of=2](3) {$B$};
\node[right of=3](4) {$A$};
\node[below of=2](5) {$A'$};
\node[right of=5](6) {$B'$};
\node[right of=6](7) {$A'$};
\node[below of=5](8) {$A'$};
\node[right of=8,rr](9) {$A'$};
\draw[d] (0) to (1);
\draw[d] (8) to (9);
\draw[d,pro] (0) to node(a)[]{} (2);
\draw[d,pro] (1) to node(b)[]{} (4);
\draw[d,pro] (5) to node(c)[]{} (8);
\draw[d,pro] (7) to node(d)[]{} (9);
\draw[->] (2) to node[over,la]{$f$} (3);
\draw[->] (3) to node[over,la]{$h$} (4);
\draw[->] (5) to node[over,la]{$f'$} (6);
\draw[->] (6) to node[over,la]{$h'$} (7);
\draw[->,pro] (2) to node(u)[left,la] {$u$} (5);
\draw[->,pro] (3) to node[left,la] {$v$} (6);
\draw[->,pro] (4) to node[right,la] {$u$} (7);

\node[left of=u,la,xshift=.5cm] {$=$};
\node[la] at ($(a)!0.45!(b)$) {$\mu$};
\node[la] at ($(a)!0.55!(b)$) {$\vcong$};
\node[la] at ($(2)!0.5!(6)$) {$\alpha$};
\node[la] at ($(3)!0.5!(7)$) {$\gamma$};
\node[la] at ($(c)!0.4!(d)$) {$(\mu')^{-1}$};
\node[la] at ($(c)!0.6!(d)$) {$\vcong$};

\node[right of=4,xshift=.5cm](B) {$A$};
\node[right of=B](A) {$A$};
\node[below of=B](B') {$A'$};
\node[below of=A](A') {$A'$};
\draw[d] (B) to (A);
\draw[->,pro] (B) to node[left,la] {$u$} (B');
\draw[d] (B') to (A');
\draw[->,pro] (A) to node[right,la] {$u$} (A');

\node[la] at ($(B)!0.5!(A')$) {$\id_u$};
\node[la] at ($(4)!0.5!(B')$) {$=$};
\end{tz}
where the first equality holds by definition of $\beta$, the second by the triangle identities for $(\eta,\epsilon)$ and $(\eta',\epsilon')$, and the last by definition of $\gamma$ being a weak inverse of $\alpha$ with respect to the horizontal adjoint equivalence data $(f,h,\mu,\delta)$ and $(f',h',\mu',\delta')$. The other pasting equality for $\alpha$, $\beta$, $\epsilon^{-1}$, and $\epsilon'$ also holds by definition of $\gamma$ being a weak inverse of $\alpha$, and by the triangle identities for $(\mu,\delta)$ and $(\mu',\delta')$. This shows that $\beta$ is a weak inverse of $\alpha$ with respect to the horizontal adjoint equivalence data $(f,g,\eta,\epsilon)$ and $(f',g',\eta',\epsilon')$. 

Now suppose that $\sq{\beta'}{g}{g'}{v}{u}$ is another weak inverse of $\alpha$ with respect to the horizontal adjoint equivalence data $(f,g,\eta,\epsilon)$ and $(f',g',\eta',\epsilon')$. Then we have that
\begin{tz}
\node[](A) {$A$};
\node[left of=A](B) {$B$};
\node[below of=B](B') {$B'$};
\node[below of=A](A') {$A'$};
\draw[->] (B) to node[above,la] {$g$} (A);
\draw[->,pro] (B) to node[left,la] {$v$} (B');
\draw[->] (B') to node[below,la] {$g'$} (A');
\draw[->,pro] (A) to node[right,la] {$u$} (A');

\node[la] at ($(B)!0.5!(A')$) {$\beta'$};

\node[right of=A,yshift=1.5cm,xshift=.5cm](0) {$B$};
\node[right of=0,rr](1) {$B$};
\node[right of=1](1') {$A$};
\node[below of=0](2) {$B$};
\node[la] at ($(A')!0.5!(2)$) {$=$};
\node[right of=2](3) {$A$};
\node[right of=3](4) {$B$};
\node[right of=4](4') {$A$};
\node[below of=2](5) {$B'$};
\node[right of=5](6) {$A'$};
\node[right of=6](7) {$B'$};
\node[right of=7](7') {$A'$};
\node[below of=5](8) {$B'$};
\node[right of=8,rr](9) {$B'$};
\node[right of=9](9') {$A'$};
\draw[d] (0) to (1);
\draw[d] (8) to (9);
\draw[d,pro] (0) to node(a)[]{} (2);
\draw[d,pro] (1) to node(b)[]{} (4);
\draw[d,pro] (1') to (4');
\draw[d,pro] (5) to node(c)[]{} (8);
\draw[d,pro] (7) to node(d)[]{} (9);
\draw[d,pro] (7') to (9');
\draw[->] (1) to node[above,la]{$g$} (1');
\draw[->] (4) to node[over,la]{$g$} (4');
\draw[->] (7) to node[over,la]{$g'$} (7');
\draw[->] (9) to node[below,la]{$g'$} (9');
\draw[->] (2) to node[over,la]{$g$} (3);
\draw[->] (3) to node[over,la]{$f$} (4);
\draw[->] (5) to node[over,la]{$g'$} (6);
\draw[->] (6) to node[over,la]{$f'$} (7);
\draw[->,pro] (2) to node[left,la] {$v$} (5);
\draw[->,pro] (3) to node[left,la] {$u$} (6);
\draw[->,pro] (4) to node[right,la] {$v$} (7);
\draw[->,pro] (4') to node[right,la] {$v$} (7');

\node[la] at ($(a)!0.4!(b)$) {$\epsilon^{-1}$};
\node[la] at ($(a)!0.6!(b)$) {$\vcong$};
\node[la] at ($(2)!0.5!(6)$) {$\beta$};
\node[la] at ($(3)!0.5!(7)$) {$\alpha$};
\node[la] at ($(4)!0.5!(7')$) {$\beta'$};
\node[la] at ($(c)!0.45!(d)$) {$\epsilon'$};
\node[la] at ($(c)!0.55!(d)$) {$\vcong$};
\node[la] at ($(1)!0.5!(4')$) {$e_g$};
\node[la] at ($(7)!0.5!(9')$) {$e_{g'}$};

\node[right of=7',xshift=.5cm](B') {$B'$};
\node[above of=B'](B) {$B$};
\node[right of=B](A) {$A$};
\node[la] at ($(7')!0.5!(B)$) {$=$};
\node[below of=A](A') {$A'$};
\draw[->] (B) to node[above,la] {$g$} (A);
\draw[->,pro] (B) to node[left,la] {$v$} (B');
\draw[->] (B') to node[below,la] {$g'$} (A');
\draw[->,pro] (A) to node[right,la] {$u$} (A');

\node[la] at ($(B)!0.5!(A')$) {$\beta$};
\end{tz}
where the first equality holds by definition of $\beta$ being a weak inverse of $\alpha$ with respect to the horizontal adjoint equivalence data $(f,g,\eta,\epsilon)$ and $(f',g',\eta',\epsilon')$, the third by definition of $\beta'$ being a weak inverse of $\alpha$ with respect to the horizontal adjoint equivalence data $(f,g,\eta,\epsilon)$ and $(f',g',\eta',\epsilon')$ and by the triangle identities for $(\eta,\epsilon)$ and $(\eta',\epsilon')$. This shows that $\beta'=\beta$ and therefore such a weak inverse is unique. 
\end{proof}

\subsection{Weakly horizontally invertible square in \texorpdfstring{$\bbH\cA$}{HA}, \texorpdfstring{$\bbHsim\cA$}{H-sim-A}, and \texorpdfstring{$\Lsim\bA$}{L-sim-A}} \label{subsec:specialwhi}

We first show that weakly horizontally invertible squares with trivial vertical boundaries correspond to vertically invertible squares between horizontal equivalences. 

\begin{lemme} \label{lem:whiiffvi}
Let $\alpha$ be a square in a double category $\bA$ of the form
\begin{tz}
\node[](A) {$A$};
\node[right of=A](B) {$B$};
\node[below of=A](A') {$A$};
\node[right of=A'](B') {$B$};
\draw[d,pro] (A) to (A');
\draw[d,pro] (B) to (B');
\draw[->] (A) to node[above,la] {$f$} (B);
\draw[->] (A') to node[below,la] {$f'$} (B');

\node[la] at ($(A)!0.5!(B')$) {$\alpha$};
\end{tz}
where $f$ and $f'$ are horizontal equivalences in $\bA$. Then the square $\alpha$ is weakly horizontally invertible if and only if it is vertically invertible. 
\end{lemme}

\begin{proof}
Suppose first that $\alpha$ is weakly horizontally invertible. Let $\beta$ be its weak inverse with respect to horizontal adjoint equivalence data $(f,g,\eta,\epsilon)$ and $(f',g',\eta',\epsilon')$. We define $\gamma$ to be given by the following pasting. 
\begin{tz}
\node[](A) {$A$};
\node[right of=A](B) {$B$};
\node[below of=A](A') {$A$};
\node[right of=A'](B') {$B$};
\draw[d,pro] (A) to (A');
\draw[d,pro] (B) to (B');
\draw[->] (A) to node[above,la] {$f'$} (B);
\draw[->] (A') to node[below,la] {$f$} (B');

\node[la] at ($(A)!0.5!(B')$) {$\gamma$};

\node[right of=B,xshift=.5cm](1) {$A$};
\node[la] at ($(B')!0.5!(1)$) {$=$};
\node[right of=1](2) {$B$};
\node[right of=2](3) {$A$};
\draw[->] (1) to node[over,la] {$f$} (2);
\draw[->] (2) to node[over,la] {$g$} (3);
\node[below of=2](2') {$B$};
\node[right of=2'](3') {$A$};
\node[right of=3'](4') {$B$};
\draw[->] (2') to node[over,la] {$g'$} (3');
\draw[->] (3') to node[over,la] {$f'$} (4');
\node[above of=1](4) {$A$};
\node[right of=4,rr](5) {$A$};
\node[right of=5](6) {$B$};
\draw[d] (4) to (5);
\draw[->] (5) to node[above,la] {$f'$} (6);
\node[below of=2'](5') {$B$};
\node[left of=5'](4'') {$A$};
\node[right of=5',rr](6') {$B$};
\draw[d] (5') to (6');
\draw[->] (4'') to node[below,la] {$f$} (5');
\draw[d,pro] (4) to node(a)[]{} (1);
\draw[d,pro] (5) to node(b)[]{} (3);
\draw[d,pro] (6) to node(y)[]{} (4');
\draw[d,pro] (1) to node(x)[]{} (4'');
\draw[d,pro] (2) to (2');
\draw[d,pro] (3) to (3');
\draw[d,pro] (2') to node(c)[]{} (5');
\draw[d,pro] (4') to node(d)[]{} (6');

\node[la] at ($(a)!0.45!(b)$) {$\eta$};
\node[la] at ($(a)!0.55!(b)$) {$\vcong$};
\node[la] at ($(c)!0.45!(d)$) {$\epsilon'$};
\node[la] at ($(c)!0.55!(d)$) {$\vcong$};
\node[la] at ($(x)!0.5!(2')$) {$e_f$};
\node[la] at ($(3)!0.5!(y)$) {$e_{f'}$};
\node[la] at ($(2)!0.5!(3')$) {$\beta$};
\end{tz}
Then one can show that $\gamma=\alpha^{-1}$ is the vertical inverse of $\alpha$ by using the definition of $\beta$ being a weak inverse of $\alpha$ with respect to horizontal adjoint equivalence data $(f,g,\eta,\epsilon)$ and $(f',g',\eta',\epsilon')$, and the triangle identities for $(\eta,\epsilon)$ and $(\eta',\epsilon')$.

Suppose now that $\alpha$ is vertically invertible. Let $(f,g,\eta,\epsilon)$ be an adjoint equivalence data and define $\eta'$ and $\epsilon'$ to be the following pasting, respectively. 
\begin{tz}
\node[](1) {$A$};
\node[right of=1,rr](2) {$A$};
\node[below of=1](4) {$A$};
\node[right of=4](5) {$B$};
\node[right of=5](6) {$A$};
\draw[d] (1) to (2);
\draw[->] (4) to node[over,la] {$f$} (5);
\draw[->] (5) to node[over,la] {$g$} (6);
\draw[d,pro] (1) to node(a)[]{} (4);
\draw[d,pro] (2) to node(b)[]{} (6);

\node[la] at ($(a)!0.45!(b)$) {$\eta$};
\node[la] at ($(a)!0.55!(b)$) {$\vcong$};

\node[below of=4](4') {$A$};
\node[right of=4'](5') {$B$};
\node[right of=5'](6') {$A$};
\draw[->] (4') to node[below,la] {$f'$} (5');
\draw[->] (5') to node[below,la] {$g$} (6');
\draw[d,pro] (4) to node(a)[]{} (4');
\draw[d,pro] (5) to node(b)[]{} (5');
\draw[d,pro] (6) to (6');

\node[la] at ($(a)!0.4!(b)$) {$\alpha$};
\node[la] at ($(a)!0.6!(b)$) {$\vcong$};
\node[la] at ($(5)!0.5!(6')$) {$e_g$};

\node[right of=6,xshift=1cm](5) {$B$};
\node[right of=5](6) {$A$};
\node[right of=6](7) {$B$};
\node[below of=5](2') {$B$};
\node[right of=2',rr](3') {$B$};
\draw[d] (2') to (3');
\draw[->] (5) to node[over,la] {$g$} (6);
\draw[->] (6) to node[over,la] {$f$} (7);
\draw[d,pro] (5) to node(c)[]{} (2');
\draw[d,pro] (7) to node(d)[]{} (3');

\node[la] at ($(c)!0.45!(d)$) {$\epsilon$};
\node[la] at ($(c)!0.55!(d)$) {$\vcong$};

\node[above of=5](5') {$B$};
\node[right of=5'](6') {$A$};
\node[right of=6'](7') {$B$};
\draw[->] (5') to node[above,la] {$g$} (6');
\draw[->] (6') to node[above,la] {$f'$} (7');
\draw[d,pro] (5') to (5);
\draw[d,pro] (6') to node(c)[]{} (6);
\draw[d,pro] (7') to node(d)[]{} (7);

\node[la] at ($(c)!0.35!(d)$) {$\alpha^{-1}$};
\node[la] at ($(c)!0.65!(d)$) {$\vcong$};
\node[la] at ($(5')!0.5!(6)$) {$e_g$};
\end{tz}
Then $(f',g,\eta',\epsilon')$ is a horizontal adjoint equivalence, and $e_g$ is a weak inverse of $\alpha$ with respect to the horizontal adjoint equivalence data $(f,g,\eta,\epsilon)$ and $(f',g,\eta',\epsilon')$. This shows that $\alpha$ is weakly horizontally invertible. 
\end{proof}

\begin{rem}
Given a $2$-category $\cA$, \cref{lem:whiiffvi} shows that a square in $\bbH\cA$ is weakly horizontally invertible if and only if its associated $2$-morphism is invertible.
\end{rem}

We now use the result above to characterize the weakly horizontally invertible squares in $\bbHsim\cA$ as invertible $2$-morphisms.

\begin{lemme} \label{lem:weakinvinHsim}
Let $\cA$ be a $2$-category and let $\sq{\alpha}{f}{f'}{u}{v}$ be a square in $\bbHsim \cA$, where $f$ and $f'$ are equivalences in $\cA$. Then $\alpha$ is weakly horizontally invertible if and only if its associated $2$-morphism $\alpha\colon vf\Rightarrow f'u$ is invertible. 
\end{lemme}

\begin{proof}
Consider a square $\alpha$ in $\bbHsim \cA$ as below right, where $f$ and $f'$ are horizontal equivalences.
\begin{tz}
    \node[](A) {$A$};
    \node[right of=A](B) {$B$};
    \node[below of=A](A') {$A'$};
    \node[right of=A'](B') {$B'$};
    \draw[->] (A) to node[above,la] {$f$} (B);
    \draw[->] (A') to node[below,la] {$f'$} (B');
    \draw[->] (A) to node[left,la]{$\underline{u}=(u,u',\eta_u,\epsilon_u)$} node[right,la]{$\lsimeq$} (A');
    \draw[->] (B) to node[right,la]{$\underline{v}=(v,v',\eta_v,\epsilon_v)$} node[left,la]{$\rsimeq$} (B');
    
    \cell[la,above,xshift=-.2cm]{B}{A'}{$\alpha$};
    
    \node[right of=B, xshift=2.5cm](A) {$A$};
    \node[right of=A](B) {$B$};
    \node[right of=B](C) {$B'$};
    \node[below of=A](A') {$A$};
    \node[right of=A'](B') {$A'$};
    \node[right of=B'](C') {$B'$};
    \draw[->] (A) to node[above,la] {$f$} (B);
    \draw[->] (B') to node[below,la] {$f'$} (C');
    \draw[->] (A') to node[below,la]{$u$} (B');
    \draw[->] (B) to node[above,la]{$v$} (C);
    \draw[d] (A) to (A');
    \draw[d] (C) to (C');
    
    \cell[la,left,xshift=-2pt]{B}{B'}{$\overline\alpha$};
\end{tz}
Then the corresponding $2$-morphism $\alpha\colon vf\Rightarrow f'u$ also gives rise to a square $\overline{\alpha}$ in $\bbHsim\cA$ as above right, where the composites $vf$ and $f'u$ are horizontal equivalences. We show that $\alpha$ is weakly horizontally invertible if and only if its associated square $\overline \alpha$ is weakly horizontally invertible. We can then conclude by applying \cref{lem:whiiffvi}. 

Let us fix adjoint equivalence data $(f,g,\eta,\epsilon)$ and $(f',g',\eta',\epsilon')$. Suppose first that the square $\beta$ in $\bbHsim\cA$, as depicted below left, is a weak inverse of $\alpha$ with respect to the adjoint equivalence data $(f,g,\eta,\epsilon)$ and $(f',g',\eta',\epsilon')$.
\begin{tz}
    \node[](A) {$B$};
    \node[right of=A](B) {$A$};
    \node[below of=A](A') {$B'$};
    \node[right of=A'](B') {$A'$};
    \draw[->] (A) to node[above,la] {$g$} (B);
    \draw[->] (A') to node[below,la] {$g'$} (B');
    \draw[->] (A) to node[left,la]{$\underline{v}$} node[right,la]{$\lsimeq$} (A');
    \draw[->] (B) to node[right,la]{$\underline{u}$} node[left,la]{$\rsimeq$} (B');
    
    \cell[la,above,xshift=-.2cm]{B}{A'}{$\beta$};
    
    \node[right of=B, xshift=1cm](A) {$B'$};
    \node[right of=A](B) {$B$};
    \node[right of=B](C) {$A$};
    \node[below of=A](A') {$B'$};
    \node[right of=A'](B') {$A'$};
    \node[right of=B'](C') {$A$};
    \draw[->] (A) to node[above,la] {$v'$} (B);
    \draw[->] (B') to node[below,la] {$u'$} (C');
    \draw[->] (A') to node[below,la]{$g'$} (B');
    \draw[->] (B) to node[above,la]{$g$} (C);
    \draw[d] (A) to (A');
    \draw[d] (C) to (C');
    
    \cell[la,left,xshift=-2pt]{B}{B'}{$\beta_*$};

    \node[right of=C](A) {$B'$};
    \node[la] at ($(C')!0.5!(A)$) {$=$};
    \node[right of=A](B) {$B$};
    \node[right of=B](C) {$A$};
    \node[below of=B](A') {$B'$};
    \node[right of=A'](B') {$A'$};
    \node[right of=B'](C') {$A$};
    \draw[->] (A) to node[above,la] {$v'$} (B);
    \draw[->] (B') to node[below,la] {$u'$} (C');
    \draw[->] (A') to node[below,la]{$g'$} (B');
    \draw[->] (B) to node[above,la]{$g$} (C);
    \draw[d] (A) to node(a)[]{} (A');
    \draw[d] (C) to node(b)[]{} (C');
    \draw[->] (B) to node[over,la]{$v$} (A');
    \draw[->] (C) to node[over,la]{$u$} (B');
    
    \cell[la,above,xshift=-.2cm]{C}{A'}{$\beta$};
    \cell[la,above,xshift=-.2cm][n][0.6]{B}{a}{$\epsilon_v$};
    \cell[la,below,xshift=.3cm][n][0.4]{b}{B'}{$\eta_u$};
\end{tz}
Then its mate $\beta_*$, given by the pasting as above right, is a weak inverse for the square $\overline\alpha$ with respect to the composite of the adjoint equivalence data $(f,g,\eta,\epsilon)$ with $(v,v',\eta_v,\epsilon_v)$, and of $(u,u',\eta_u,\epsilon_u)$ with $(f',g',\eta',\epsilon')$. This follows from the triangle identities for $(\eta_u,\epsilon_u)$ and $(\eta_v,\epsilon_v)$ and the definition of $\beta$ being a weak inverse of $\alpha$ with respect to the adjoint equivalence data $(f,g,\eta,\epsilon)$, $(f',g',\eta',\epsilon')$. 

Conversely, suppose that the square $\overline{\beta}$ in $\bbHsim\cA$, as depicted below left, is a weak inverse of $\overline \alpha$ with respect to the composite of the adjoint equivalence data $(f,g,\eta,\epsilon)$ with $(v,v',\eta_v,\epsilon_v)$, and of $(u,u',\eta_u,\epsilon_u)$ with $(f',g',\eta',\epsilon')$.
\begin{tz}
    \node[](A) {$B'$};
    \node[right of=A](B) {$B$};
    \node[right of=B](C) {$A$};
    \node[below of=A](A') {$B'$};
    \node[right of=A'](B') {$A'$};
    \node[right of=B'](C') {$A$};
    \draw[->] (A) to node[above,la] {$v'$} (B);
    \draw[->] (B') to node[below,la] {$u'$} (C');
    \draw[->] (A') to node[below,la]{$g'$} (B');
    \draw[->] (B) to node[above,la]{$g$} (C);
    \draw[d] (A) to (A');
    \draw[d] (C) to (C');
    
    \cell[la,left,xshift=-2pt]{B}{B'}{$\overline{\beta}$};
    
     \node[right of=C,xshift=1cm](A) {$B$};
    \node[right of=A](B) {$A$};
    \node[below of=A](A') {$B'$};
    \node[right of=A'](B') {$A'$};
    \draw[->] (A) to node[above,la] {$g$} (B);
    \draw[->] (A') to node[below,la] {$g'$} (B');
    \draw[->] (A) to node[left,la]{$v$} node[right,la]{$\lsimeq$} (A');
    \draw[->] (B) to node[right,la]{$u$} node[left,la]{$\rsimeq$} (B');
    
    \cell[la,above,xshift=-.2cm]{B}{A'}{$\overline \beta_!$};
    
    \node[right of=B'](X) {$B$};
    \node[la] at ($(B)!0.5!(X)$) {$=$};
    \node[right of=X](A) {$B'$};
    \node[above of=A](B) {$B$};
    \node[right of=A](A') {$A'$};
    \node[above of=A'](B') {$A$};
    \node[right of=B'](Y) {$A'$};
    \draw[->] (A) to node[over,la] {$v'$} (B);
    \draw[->] (X) to node[below,la] {$v$} (A);
    \draw[->] (B') to node[above,la] {$u$} (Y);
    \draw[->] (A') to node[over,la] {$u'$} (B');
    \draw[->] (A) to node[below,la]{$g'$} (A');
    \draw[->] (B) to node[above,la]{$g$} (B');
    \draw[d] (X) to node(a)[]{} (B);
    \draw[d] (A') to node(b)[]{} (Y);
    
    \cell[la,right,yshift=.3cm]{B}{A'}{$\overline\beta$};
    \cell[la,left,yshift=-.2cm][n][0.4]{a}{A}{$\eta_v$};
    \cell[la,right,yshift=.2cm][n][0.6]{B'}{b}{$\epsilon_u$};
\end{tz}
Then its mate $\overline{\beta}_!$, given by the pasting as above right, is a weak inverse of $\alpha$ with respect to the adjoint equivalence data $(f,g,\eta,\epsilon)$, $(f',g',\eta',\epsilon')$. 
\end{proof}

In particular, we can see that a $2$-morphism in $\Lsim\bA$ corresponding to a weakly horizontally invertible square in a double category $\bA$ is invertible, where $\Lsim\colon \DblCat\to \TwoCat$ of the functor $\bbHsim$.

\begin{lemme} \label{lem:whiinLsim}
Let $\bA$ be a double category. 
\begin{rome}
\item If $f\colon A\to B$ is a horizontal equivalence in $\bA$, then the corresponding morphism $f\colon A\to B$ in $\Lsim\bA$ is an equivalence. 
\item If $\sq{\alpha}{f}{f'}{u}{v}$ is a weakly horizontally invertible square in $\bA$, then the corresponding $2$-morphism $\alpha\colon vf\Rightarrow f'u$ in $\Lsim\bA$ is invertible. 
\end{rome}
\end{lemme}

\begin{proof}
Given a horizontal equivalence $(f,g,\eta,\epsilon)$ in $\bA$, then there are corresponding morphisms $f$ and $g$ and corresponding invertible $2$-morphisms $\eta\colon \id\cong gf$ and $\epsilon\colon fg\cong \id$ in $\Lsim\bA$, i.e., this is the data of an equivalence in $\Lsim\bA$. This proves (i).

Now, given a weakly horizontally invertible square $\sq{\alpha}{f}{f'}{u}{v}$ in $\bA$, then the corresponding morphisms $f$ and $f'$ are equivalences in $\Lsim\bA$ by (i). The relations expressing the fact that $\alpha$ is a weakly horizontally invertible square in $\bA$ translate to relations in $\bbHsim\Lsim\bA$ implying that the corresponding square 
\begin{tz}
    \node[](A) {$A$};
    \node[right of=A](B) {$B$};
    \node[below of=A](A') {$A'$};
    \node[right of=A'](B') {$B'$};
    \draw[->] (A) to node[above,la] {$f$} (B);
    \draw[->] (A') to node[below,la] {$f'$} (B');
    \draw[->] (A) to node[left,la]{$u$} node[right,la]{$\lsimeq$} (A');
    \draw[->] (B) to node[right,la]{$v$} node[left,la]{$\rsimeq$} (B');
    
    \cell[la,above,xshift=-.2cm]{B}{A'}{$\alpha$};
\end{tz}
is weakly horizontally invertible in $\bbHsim\Lsim\bA$. By \cref{lem:weakinvinHsim}, we obtain that the associated $2$-morphism $\alpha\colon vf\Rightarrow f'u$ is invertible. 
\end{proof}

\subsection{Horizontal pseudo-natural equivalences} \label{subsec:ps-equ}

We now give complete definitions of the morphisms and $2$-morphisms of the pseudo-hom $2$-category $\bfH[\bI, \bA]_\ps$ of double functors defined in \cref{def:tensor}. 

\begin{defn} \label{def:pseudohor}
Let $F,G\colon \bI\to \bA$ be double functors. A \textbf{horizontal pseudo-natural transformation} $\varphi\colon F\Rightarrow G$ consists of 
\begin{rome}
\item a horizontal morphism $\varphi_i\colon Fi\to Gi$ in $\bA$, for each object $i\in \bI$,
\item a square $\sq{\varphi_u}{\varphi_i}{\varphi_{i'}}{Fu}{Gu}$ in $\bA$, for each vertical morphism $u\colon i\arrowdot i'$ in $\bI$, 
\item a vertically invertible square $\sq{\varphi_f}{(Gf)\varphi_i}{\varphi_j(Ff)}{e_{Fi}}{e_{Gj}}$ in $\bA$, for each horizontal morphism $f\colon i\to j$ in $\bI$,
\end{rome}
such that the following conditions hold:
\begin{enumerate}
\item for every object $i\in \bI$, $\sq{\varphi_{e_i}=e_{\varphi_i}}{\varphi_i}{\varphi_i}{e_{Fi}}{e_{Gi}}$, 
\item for every pair of composable vertical morphisms $u\colon i\arrowdot i'$ and $v\colon i'\arrowdot i''$ in $\bI$, the vertical composite of the squares $\varphi_u$ and $\varphi_v$ is given by the square $\varphi_{vu}$,
\item for every object $i\in \bI$, $\sq{\varphi_{\id_i}=e_{\varphi_i}}{\varphi_i}{\varphi_i}{e_{Fi}}{e_{Gi}}$, 
\item for every pair of composable horizontal morphisms $f\colon i\to j$ and $g\colon j\to k$ in $\bI$,
\begin{tz}
\node[](A) {$Fi$};
\node[right of=A](B) {$Gi$};
\node[right of=B](C) {$Gj$};
\node[below of=A](A') {$Fi$};
\node[right of=A'](B') {$Fj$};
\node[right of=B'](C') {$Gj$};
\node[right of=C](D) {$Gk$};
\node[right of=C'](D') {$Gk$};
\draw[->] (A') to node[over,la] {$Ff$} (B');
\draw[->] (B') to node[over,la] {$\varphi_j$} (C');
\draw[->] (C) to node[above,la] {$Gg$} (D);
\draw[->] (C') to node[over,la] {$Gg$} (D');
\draw[->] (A) to node[above,la]{$\varphi_i$} (B);
\draw[->] (B) to node[above,la] {$Gf$} (C);
\draw[d,pro] (A) to node(a)[]{} (A');
\draw[d,pro] (C) to node(b)[]{} (C');
\draw[d,pro] (D) to (D');

\node[la] at ($(a)!0.45!(b)$) {$\varphi_f$};
\node[la] at ($(a)!0.55!(b)$) {$\vcong$};
\node[la] at ($(C)!0.5!(D')$) {$e_{Gg}$};

\node[below of=A'](A'') {$Fi$};
\node[below of=B'](B'') {$Fj$};
\node[below of=C'](C'') {$Fk$};
\node[below of=D'](D'') {$Gk$};
\draw[->] (A'') to node[below,la] {$Ff$} (B'');
\draw[->] (B'') to node[below,la] {$Fg$} (C'');
\draw[->] (C'') to node[below,la] {$\varphi_k$} (D'');
\draw[d,pro] (A') to (A'');
\draw[d,pro] (B') to node(b)[]{} (B'');
\draw[d,pro] (D') to node(c)[]{} (D'');

\node[la] at ($(A')!0.5!(B'')$) {$e_{Ff}$};
\node[la] at ($(b)!0.45!(c)$) {$\varphi_g$};
\node[la] at ($(b)!0.55!(c)$) {$\vcong$};

\node[right of=D',yshift=.75cm,xshift=.5cm](A) {$Fi$};
\node[right of=A](B) {$Gi$};
\node[right of=B](C) {$Gk$};
\node[below of=A](A') {$Fi$};
\node[right of=A'](B') {$Fk$};
\node[right of=B'](C') {$Gk$};
\node at ($(C'.east)-(0,4pt)$) {,};
\draw[->] (A) to node[above,la] {$\varphi_i$} (B);
\draw[->] (B) to node[above,la] {$G(gf)$} (C);
\draw[->] (A') to node[below,la] {$F(gf)$} (B');
\draw[->] (B') to node[below,la] {$\varphi_k$} (C');
\draw[d,pro] (A) to node(a)[]{} (A');
\node[la] at ($(D')!0.5!(a)$) {$=$};
\draw[d,pro] (C) to node(b)[]{} (C');

\node[la] at ($(a)!0.42!(b)$) {$\varphi_{gf}$};
\node[la] at ($(a)!0.58!(b)$) {$\vcong$};
\end{tz}
\item for every square $\sq{\alpha}{f}{f'}{u}{v}$ in $\bI$,
\begin{tz}
\node[](1) {$Fi$};
\node[right of=1](2) {$Gi$};
\node[right of=2](3) {$Gj$};
\node[below of=1](4) {$Fi$};
\node[right of=4](5) {$Fj$};
\node[right of=5](6) {$Gj$};
\draw[->] (1) to node[above,la] {$\varphi_i$} (2);
\draw[->] (2) to node[above,la] {$Gf$} (3);
\draw[d,pro] (1) to node(a)[]{} (4);
\draw[d,pro] (3) to node(b)[]{} (6);
\draw[->] (4) to node[above,la] {$Ff$}  (5);
\draw[->] (5) to node[over,la] {$\varphi_j$}  (6);

\node[la] at ($(a)!0.45!(b)$) {$\varphi_f$};
\node[la] at ($(a)!0.55!(b)$) {$\vcong$};

\node[below of=4](7) {$Fi'$};
\node[right of=7](8) {$Fj'$};
\node[right of=8](9) {$Gj'$};
\draw[->,pro] (4) to node[left,la] {$Fu$} (7);
\draw[->,pro] (5) to node(b)[right,la] {$Fv$} (8);
\draw[->,pro] (6) to node(c)[right,la] {$Gv$} (9);
\draw[->] (7) to node[below,la] {$Ff'$} (8);
\draw[->] (8) to node[below,la] {$\varphi_{j'}$} (9);

\node[la] at ($(4)!0.5!(8)$) {$F\alpha$};
\node[la] at ($(b)!0.4!(c)$) {$\varphi_v$};

\node[right of=3,xshift=.5cm](1) {$Fi$};
\node[la] at ($(9)!0.5!(1)$) {$=$};
\node[right of=1](2) {$Gi$};
\node[right of=2](3') {$Gj$};
\node[below of=1](3) {$Fi'$};
\node[right of=3](4) {$Gi'$};
\node[right of=4](5) {$Gj'$};
\draw[->] (1) to node[above,la] {$\varphi_i$} (2);
\draw[->] (2) to node[above,la] {$Gf$} (3');
\draw[->,pro] (1) to node(a)[left,la] {$Fu$} (3);
\draw[->,pro] (2) to node(b)[left,la] {$Gu$} (4);
\draw[->,pro] (3') to node[right,la] {$Gv$} (5);
\draw[->] (3) to node[over,la] {$\varphi_{i'}$} (4);
\draw[->] (4) to node[below,la] {$Gf'$} (5);

\node[la] at ($(a)!0.6!(b)$) {$\varphi_u$};
\node[la] at ($(2)!0.5!(5)$) {$G\alpha$};

\node[below of=3](7) {$Fi'$};
\node[right of=7](8) {$Fj'$};
\node[right of=8](9) {$Gj'$};
\node at ($(9.east)-(0,4pt)$) {.};
\draw[->] (7) to node[below,la] {$Ff'$} (8);
\draw[->] (8) to node[below,la] {$\varphi_{j'}$} (9);
\draw[d,pro] (3) to node(a)[]{} (7);
\draw[d,pro] (5) to node(b)[]{} (9);

\node[la] at ($(a)!0.45!(b)$) {$\varphi_{f'}$};
\node[la] at ($(a)!0.55!(b)$) {$\vcong$};
\end{tz}
\end{enumerate}
\end{defn}

\begin{defn} \label{def:modif}
Let $\varphi,\psi\colon F\Rightarrow G$ be horizontal pseudo-natural transformations between double functors $F,G\colon \bI\to \bA$. A \textbf{modification} $\mu\colon \varphi\to \psi$ consists of a square $\sq{\mu_i}{\varphi_i}{\psi_i}{e_{Fi}}{e_{Gi}}$ in $\bA$, for each object $i\in \bI$, such that
\begin{enumerate}
\item for every horizontal morphisms $f\colon i\to j$ in $\bI$,
\begin{tz}
\node[](1) {$Fi$};
\node[right of=1](2) {$Gi$};
\node[right of=2](3) {$Gj$};
\node[below of=1](4) {$Fi$};
\node[right of=4](5) {$Fj$};
\node[right of=5](6) {$Gj$};
\draw[->] (1) to node[above,la] {$\varphi_i$} (2);
\draw[->] (2) to node[above,la] {$Gf$} (3);
\draw[d,pro] (1) to node(a)[]{} (4);
\draw[d,pro] (3) to node(b)[]{} (6);
\draw[->] (4) to node[over,la] {$Ff$}  (5);
\draw[->] (5) to node[over,la] {$\varphi_j$}  (6);

\node[la] at ($(a)!0.45!(b)$) {$\varphi_f$};
\node[la] at ($(a)!0.55!(b)$) {$\vcong$};

\node[below of=4](7) {$Fi$};
\node[right of=7](8) {$Fj$};
\node[right of=8](9) {$Gj$};
\draw[d,pro] (4) to (7);
\draw[d,pro] (5) to (8);
\draw[d,pro] (6) to (9);
\draw[->] (7) to node[below,la] {$Ff$} (8);
\draw[->] (8) to node[below,la] {$\psi_{j}$} (9);

\node[la] at ($(4)!0.5!(8)$) {$e_{Ff}$};
\node[la] at ($(5)!0.5!(9)$) {$\mu_j$};

\node[right of=3,xshift=.5cm](1) {$Fi$};
\node[la] at ($(9)!0.5!(1)$) {$=$};
\node[right of=1](2) {$Gi$};
\node[right of=2](3') {$Gj$};
\node[below of=1](3) {$Fi$};
\node[right of=3](4) {$Gi$};
\node[right of=4](5) {$Gj$};
\draw[->] (1) to node[above,la] {$\varphi_i$} (2);
\draw[->] (2) to node[above,la] {$Gf$} (3');
\draw[d,pro] (1) to (3);
\draw[d,pro] (2) to (4);
\draw[d,pro] (3') to (5);
\draw[->] (3) to node[over,la] {$\psi_i$} (4);
\draw[->] (4) to node[over,la] {$Gf$} (5);

\node[la] at ($(1)!0.5!(4)$) {$\mu_i$};
\node[la] at ($(2)!0.5!(5)$) {$e_{Gf}$};

\node[below of=3](7) {$Fi$};
\node[right of=7](8) {$Fj$};
\node[right of=8](9) {$Gj$};
\node at ($(9.east)-(0,4pt)$) {,};
\draw[->] (7) to node[below,la] {$Ff$} (8);
\draw[->] (8) to node[below,la] {$\psi_{j}$} (9);
\draw[d,pro] (3) to node(a)[]{} (7);
\draw[d,pro] (5) to node(b)[]{} (9);

\node[la] at ($(a)!0.45!(b)$) {$\psi_{f}$};
\node[la] at ($(a)!0.55!(b)$) {$\vcong$};
\end{tz}
\item for every vertical morphism $u\colon i\arrowdot i'$ in $\bI$,
\begin{tz}
    \node[](1) {$Fi$};
    \node[right of=1](2) {$Gi$};
    \node[below of=1](3) {$Fi'$};
    \node[right of=3](4) {$Gi'$};
    \node[below of=3](5) {$Fi'$};
    \node[right of=5](6) {$Gi'$};
    \draw[->] (1) to node[above,la] {$\varphi_{i}$} (2);
    \draw[->] (3) to node[over,la] {$\varphi_{i'}$} (4);
    \draw[->] (5) to node[below,la] {$\psi_{i'}$} (6);
    \draw[->,pro] (1) to node[left,la] {$Fu$} (3);
    \draw[->,pro] (2) to node[right,la] {$Gu$} (4);
    \draw[d,pro] (3) to (5);
    \draw[d,pro] (4) to (6);
    
    \node[la] at ($(1)!0.5!(4)$) {$\varphi_u$};
    \node[la] at ($(3)!0.5!(6)$) {$\mu_{i'}$};
    
    \node[right of=2,xshift=.5cm](1) {$Fi$};
    \node[right of=1](2) {$Gi$};
    \node[below of=1](3) {$Fi$};
    \node[la] at ($(4)!0.5!(3)$) {$=$};
    \node[right of=3](4) {$Gi$};
    \node[below of=3](5) {$Fi'$};
    \node[right of=5](6) {$Gi'$};
\node at ($(6.east)-(0,4pt)$) {.};
    \draw[->] (1) to node[above,la] {$\varphi_{i}$} (2);
    \draw[->] (3) to node[over,la] {$\psi_{i}$} (4);
    \draw[->] (5) to node[below,la] {$\psi_{i'}$} (6);
    \draw[d,pro] (1) to (3);
    \draw[d,pro] (2) to (4);
    \draw[->,pro] (3) to node[left,la] {$Fu$} (5);
    \draw[->,pro] (4) to node[right,la] {$Gu$} (6);
    
    \node[la] at ($(1)!0.5!(4)$) {$\mu_i$};
    \node[la] at ($(3)!0.5!(6)$) {$\psi_u$};
\end{tz}
\end{enumerate}
\end{defn}

In particular, we show that an equivalence in the $2$-category $\bfH[\bI,\bA]_\ps$ is precisely a horizontal pseudo-natural transformation whose square components are weakly horizontally invertible squares. 

\begin{lemme} \label{lem:pseudoeq}
Let $\varphi\colon F\Rightarrow G$ be a horizontal pseudo-natural transformation between double functors $F,G\colon \bI\to \bA$. Then $\varphi$ is an equivalence in the $2$-category $\bfH[\bI,\bA]_\ps$ if and only if the square $\sq{\varphi_u}{\varphi_i}{\varphi_{i'}}{Fu}{Gu}$ is weakly horizontally invertible, for every vertical morphism $u\colon i\arrowdot i'$ in $\bI$. In particular, the horizontal morphism $\varphi_i\colon Fi\to Gi$ is a horizontal equivalence, for every object $i\in \bI$. 
\end{lemme}

\begin{proof}
Suppose first that $(\varphi,\psi,\eta,\epsilon)$ is an equivalence in the $2$-category $\bfH[\bI,\bA]_\ps$, i.e., we have the data of horizontal pseudo-natural transformations $\varphi\colon F\Rightarrow G$ and $\psi\colon G\Rightarrow F$ together with invertible modifications $\eta\colon \id_F\cong \psi\varphi$ and $\epsilon\colon \varphi\psi\cong \id_G$. By applying condition~(2) of \cref{def:modif} to the modifications $\eta$ and $\epsilon$, we directly get that $(\varphi_u,\psi_u)$ are weak inverses with respect to the horizontal equivalence data $(\varphi_i,\psi_i,\eta_i,\epsilon_i)$ and  $(\varphi_{i'},\psi_{i'},\eta_{i'},\epsilon_{i'})$, for  every vertical morphism $u\colon i\arrowdot i'$ in $\bA$. This shows that every square $\varphi_u$ is weakly horizontally invertible. 

Now suppose that the square $\sq{\varphi_u}{\varphi_i}{\varphi_{i'}}{Fu}{Gu}$ is weakly horizontally invertible, for every vertical morphism $u\colon i\arrowdot i'$ in $\bI$. For each object $i\in \bI$, let us fix a horizontal adjoint equivalence data $(\varphi_i,\psi_i,\eta_i,\epsilon_i)$. For each vertical morphism $u\colon i\arrowdot i'$ in $\bI$, we denote by $\sq{\psi_u}{\psi_i}{\psi_{i'}}{Gu}{Fu}$ the unique weak inverse of $\varphi_u$ given by \cref{lem:uniqueweakinv} with respect to the horizontal adjoint equivalence data $(\varphi_i,\psi_i,\eta_i,\epsilon_i)$ and  $(\varphi_{i'},\psi_{i'},\eta_{i'},\epsilon_{i'})$.

We define a horizontal pseudo-natural transformation $\psi\colon G\Rightarrow F$ which is given by the horizontal morphism $\psi_i\colon Gi\to Fi$, at each object $i\in \bI$, the square $\sq{\psi_u}{\psi_i}{\psi_{i'}}{Gu}{Fu}$, at each vertical morphism $u\colon i\arrowdot i'$ in $\bI$, and by the vertically invertible square $\psi_f$
\begin{tz}
    \node[](A) {$Gi$};
    \node[right of=A](B) {$Fi$};
    \node[right of=B](C) {$Fj$};
    \node[below of=A](A') {$Gi$};
    \node[right of=A'](B') {$Gj$};
    \node[right of=B'](C') {$Fj$};
    \draw[->] (A) to node[above,la] {$\psi_i$} (B);
    \draw[->] (B) to node[above,la]{$Ff$} (C);
    \draw[->] (B') to node[below,la] {$\psi_j$} (C');
    \draw[->] (A') to node[below,la]{$Gf$} (B');
    \draw[d,pro] (A) to node(a)[]{} (A');
    \draw[d,pro] (C) to node(c)[]{} (C');
    
    \node[la] at ($(a)!0.45!(c)$) {$\psi_f$};
    \node[la] at ($(a)!0.55!(c)$) {$\vcong$};
    
    \node[right of=C',xshift=.5cm](1) {$Gi$};
    \node[la] at ($(C)!0.5!(1)$) {$=$};
    \node[right of =1](2) {$Fi$};
    \node[right of=2](3) {$Gi$};
    \node[right of=3](4) {$Gj$};
    \node[below of=1](5) {$Gi$};
    \node[below of=3](6) {$Gi$};
    \node[right of=6](7) {$Gj$};
    \node[right of=7](8) {$Fj$};
    \draw[->] (1) to node[over,la] {$\psi_i$} (2);
    \draw[->] (2) to node[over,la] {$\varphi_i$} (3);
    \draw[->] (3) to node[over,la] {$Gf$} (4);
    \draw[d] (5) to (6);
    \draw[->] (6) to node[below,la] {$Gf$} (7);
    \draw[->] (7) to node[below,la]{$\psi_j$} (8);
    \draw[d,pro] (1) to node(a)[]{} (5);
    \draw[d,pro] (3) to node(b)[]{} (6);
    \draw[d,pro] (4) to (7);
    
    \node[la] at ($(a)!0.45!(b)$) {$\epsilon_i$};
    \node[la] at ($(a)!0.55!(b)$) {$\vcong$};
    \node[la] at ($(3)!0.55!(7)$) {$e_{Gf}$};
    
    \node[above of=2](1') {$Fi$};
    \node[right of =1'](2') {$Fj$};
    \node[right of=2'](3') {$Gj$};
    \node[right of=3'](4') {$Fj$};
    \draw[->] (1') to node[over,la] {$Ff$} (2');
    \draw[->] (2') to node[over,la] {$\varphi_j$} (3');
    \draw[->] (3') to node[over,la] {$\psi_j$} (4');
    \draw[d,pro] (1') to node(a)[]{} (2);
    \draw[d,pro] (3') to node(b)[]{} (4);
    \draw[d,pro] (4') to node(z)[]{} (8);
    
    \node[la] at ($(a)!0.43!(b)$) {$\varphi^{-1}_f$};
    \node[la] at ($(a)!0.57!(b)$) {$\vcong$};
    \node[la] at ($(4)!0.55!(z)$) {$e_{\psi_j}$};
    
    \node[above of=1'](6) {$Fi$};
    \node[left of=6](5) {$Gi$};
    \node[right of=6](7) {$Fj$};
    \node[above of=4'](8) {$Fj$};
    \draw[d] (7) to (8);
    \draw[->] (5) to node[above,la] {$\psi_i$} (6);
    \draw[->] (6) to node[above,la]{$Ff$} (7);
    \draw[d,pro] (5) to node(z)[]{} (1);
    \draw[d,pro] (6) to (1');
    \draw[d,pro] (7) to node(a)[]{} (2');
    \draw[d,pro] (8) to node(b)[]{} (4');
    
    \node[la] at ($(a)!0.45!(b)$) {$\eta_j$};
    \node[la] at ($(a)!0.55!(b)$) {$\vcong$};
    \node[la] at ($(6)!0.55!(2')$) {$e_{Ff}$};
    \node[la] at ($(z)!0.55!(1')$) {$e_{\psi_i}$};
    \end{tz}
at each horizontal morphism $f\colon i\to j$ in $\bI$. We show that this data assemble into a horizontal pseudo-natural transformation $\psi\colon G\Rightarrow F$ by verifying conditions (1)-(5) of \cref{def:pseudohor}. We have (1), since $\psi_{e_i}$ is the inverse of $\varphi_{e_i}$, which is unique by \cref{lem:uniqueweakinv} and therefore must be equal to $e_{\psi_i}$. Condition (2) follows from the fact that the vertical composite of $\psi_u$ and $\psi_{v}$, and the square $\psi_{vu}$ are both weak inverse of $\varphi_{vu}$ with respect to the horizontal adjoint equivalence data $(\varphi_i,\psi_i,\eta_i,\epsilon_i)$ and  $(\varphi_{i''},\psi_{i''},\eta_{i''},\epsilon_{i''})$; they must therefore be equal since such a weak inverse is unique by \cref{lem:uniqueweakinv}. Conditions (3) and (4) follow from the definition of $\psi_f$ and the triangle identities for $(\eta_i,\epsilon_i)$, for each $i\in \bI$. The last condition follows from the definition of $\psi_f$ and condition (5) for the horizontal pseudo-natural transformation $\varphi$. Moreover, it is straightforward to check that the vertically invertible squares $\eta_i$ and $\epsilon_i$ assemble into invertible modifications $\eta\colon \id_F\cong \psi\varphi$ and $\epsilon\colon \varphi\psi\cong \id_G$. This shows that $(\varphi,\psi,\eta,\epsilon)$ is an equivalence in $\bfH[\bA,\bB]_\ps$.
\end{proof}

\section{Explicit description of the nerves in lower dimensions} \label{app:nerveinlow}

In this appendix, we describe the nerves of the different double categories considered in this paper in lower dimensions; namely, for $0\leq m,k\leq 1$ and $0\leq n\leq 2$. The aim of these descriptions is to give the intuition that the space of the nerve at $(m,k)=(0,0)$ models the \emph{space of objects}, the one at $(m,k)=(1,0)$ the \emph{space of horizontal morphisms}, the one at $(m,k)=(0,1)$ the \emph{space of vertical morphisms}, and the one at $(m,k)=(1,1)$ the \emph{space of squares} of the corresponding double category. In \cref{subsec:descr-ND}, we first describe the nerve $\ND$ of a general double category. Then, in \cref{subsec:descr-NHsim}, we describe the nerve $\ND\bbHsim$ of a $2$-category. Finally, in \cref{subsec:descr-NH}, we also describe the nerve $\ND\bbH$ of a $2$-category, in order to compare it with its fibrant replacement~$\ND\bbHsim$.

\subsection{Nerve of a double category} \label{subsec:descr-ND}

Let $\bA$ be a double category. We first want to describe the $0$-, $1$-, and $2$-simplices of the space $(\ND\bA)_{m,k}$ for $0\leq m,k\leq 1$.

\begin{descr}
 By definition of $\ND$, we have that 
\begin{align*} (\ND\bA)_{m,k,n}&=\DblCat((\VK{k}\otimes\OM{m})\otimes \ON{n},\ND\bA) \\
&\cong \TwoCat(\ON{n},\bfH[\VK{k}\otimes\OM{m},\bA]_\ps). 
\end{align*}
Therefore we can describe the $0$-, $1$-, and $2$-simplices of the space $(\ND\bA)_{m,k}$ as follows. 
\begin{enumerate}\addtocounter{enumi}{-1}
    \item A $0$-simplex in $(\ND\bA)_{m,k}$ is a double functor $F\colon \VK{k}\otimes\OM{m}\to\bA$.
    \item A $1$-simplex in $(\ND\bA)_{m,k}$ is an adjoint equivalence in $\bfH[\VK{k}\otimes\OM{m},\bA]_\ps$, i.e., by \cref{lem:pseudoeq}, a horizontal pseudo-natural transformation
\begin{tz}
\node[](A) {$\VK{k}\otimes\OM{m}$};
\node[right of=A,rr](B) {$\bA$};
\draw[->,bend left] (A.north east) to node(a)[above,la] {$F$} (B.north west);
\draw[->,bend right] (A.south east) to node(b)[below,la] {$G$} (B.south west);
\cell[la,right]{a}{b}{$\varphi$};
\end{tz}
    such that, the horizontal morphism $\varphi_i\colon Fi\to Gi$ is a horizontal adjoint equivalence in $\bA$, for each object $i\in \VK{k}\otimes\OM{m}$, and the square $\sq{\varphi_u}{\varphi_i}{\varphi_{i'}}{Fu}{Gu}$ is weakly horizontally invertible, for each vertical morphism $u$ in $\VK{k}\otimes\OM{m}$. In what follows, we call such a $\varphi$ a \textbf{horizontal pseudo-natural adjoint equivalence} and we write $\varphi\colon F\xRightarrow{\simeq} G$.
    \item A $2$-simplex is the data of three horizontal pseudo-natural adjoint equivalences $\varphi\colon F\xRightarrow{\simeq} G$, $\psi\colon G\xRightarrow{\simeq} H$, and $\theta\colon F\xRightarrow{\simeq} H$ together with an invertible modification $\mu$ as follows. 
\begin{center}
\begin{tikzpicture}[node distance=1.7cm]
\node[](1) {$F$};
\node[above right of=1](2) {$G$};
\node[below right of=2](3) {$H$};
\draw[n] (1) to node[left,la,yshift=.3cm]{$\varphi$} (2);
\draw[n] (2) to node[right,la,yshift=.3cm]{$\psi$}(3);
\draw[n] (1) to node(b)[below,la]{$\theta$}(3);
\coordinate(a) at ($(1)!0.5!(3)$);
\cell[la,right,xshift=2pt][t][0.37]{a}{2}{$\cong$};
\node[la] at ($(2)!0.5!(b)-(.3cm,.1cm)$) {$\mu$};
\end{tikzpicture}
\end{center}
\end{enumerate}
\end{descr}

We first compute the space $(\ND\bA)_{0,0}$, which is given by the \emph{space of objects}. As expected from the completeness condition being in the horizontal direction, its $0$-simplices are given by the objects, and its $1$-simplices by the horizontal adjoint equivalences. 

\begin{descr}[$m=0$, $k=0$] \label{desc:double00}
We describe the space $(\ND\bA)_{0,0}$. First note that the double category $\VK{0}\otimes\OM{0}=[0]$ is the terminal (double) category.
\begin{enumerate}\addtocounter{enumi}{-1}
    \item A $0$-simplex in $(\ND\bA)_{0,0}$ is a double functor $A\colon [0]\to \bA$, i.e., the data of an object $A\in \bA$.
    \item A $1$-simplex in the space $(\ND\bA)_{0,0}$ is a horizontal pseudo-natural adjoint equivalence $\varphi\colon A \xRightarrow{\simeq} B$, i.e., the data of a horizontal adjoint equivalence $\varphi\colon A\xrightarrow{\simeq} C$ in $\bA$.
    \item A $2$-simplex in $(\ND\bA)_{0,0}$ is an invertible modification $\mu\colon \theta\cong \psi\varphi$ between such horizontal pseudo-natural adjoint equivalences, i.e., the data of a vertically invertible square in $\bA$
\begin{tz}
\node[](A) {$A$};
\node[right of=A,rr](C) {$E$};
\node[below of=A](A') {$A$};
\node[right of=A'](B') {$C$};
\node[right of=B'](C') {$E$};
\node at ($(C'.east)-(0,4pt)$) {.};
\draw[d,pro] (A) to node(a)[]{} (A'); 
\draw[d,pro] (C) to node(b)[]{} (C'); 
\draw[->] (A) to node[below,la] {$\simeq$} node[above,la] {$\theta$} (C);
\draw[->] (A') to node[below,la] {$\varphi$} node[above,la] {$\simeq$} (B');
\draw[->] (B') to node[below,la] {$\psi$} node[above,la] {$\simeq$} (C');

\node[la] at ($(a)!0.45!(b)$) {$\mu$};
\node[la] at ($(a)!0.55!(b)$) {$\vcong$};
\end{tz}
\end{enumerate}
\end{descr}

We now turn our attention to the \emph{space of horizontal morphisms} $(\ND\bA)_{1,0}$. We observe that the squares appearing as $n$-simplices of this space all have trivial vertical boundaries. In particular, this prevents a completeness condition for $(\ND\bA)_{1,-}$ for a general double category. 

\begin{descr}[$m=1$, $k=0$] \label{desc:double10}
We describe the space $(\ND\bA)_{1,0}$. First note that $\VK{0}\otimes \OM{1}=\bbH[1]$ is the free double category on a horizontal morphism. 
\begin{enumerate}\addtocounter{enumi}{-1}
    \item A $0$-simplex in $(\ND\bA)_{1,0}$ is a double functor $f\colon \bbH[1]\to \bA$, i.e., the data of a horizontal morphism $f\colon A\to B$ in $\bA$.
    \item A $1$-simplex in the space $(\ND\bA)_{1,0}$ is a horizontal pseudo-natural adjoint equivalence $\varphi\colon f\xRightarrow{\simeq} g$, i.e., the data of two horizontal adjoint equivalences $\varphi_0\colon A\xrightarrow{\simeq} C$ and $\varphi_1\colon B\xrightarrow{\simeq} D$ together with a vertically invertible square in $\bA$
\begin{tz}
\node[](A) {$A$};
\node[right of=A](B) {$C$};
\node[right of=B](C) {$D$};
\node[below of=A](A') {$A$};
\node[right of=A'](B') {$B$};
\node[right of=B'](C') {$D$};
\node at ($(C'.east)-(0,4pt)$) {.};
\draw[->] (A) to node[above,la] {$\varphi_0$} node[below,la] {$\simeq$} (B);
\draw[->] (B) to node[above,la] {$g$} (C);
\draw[->] (A') to node[below,la] {$f$} (B');
\draw[->] (B') to node[below,la]  {$\varphi_1$} node[above,la]{$\simeq$}  (C');
\draw[d,pro] (A) to node(a)[]{} (A');
\draw[d,pro] (C) to node(b)[]{} (C');

\node[la] at ($(a)!0.45!(b)$) {$\varphi$};
\node[la] at ($(a)!0.55!(b)$) {$\vcong$};
\end{tz}
\item A $2$-simplex in $(\ND\bA)_{1,0}$ is an invertible modification $\mu\colon \theta\cong \psi\varphi$ between such horizontal pseudo-natural adjoint equivalences, i.e., the data of two vertically invertible squares $\mu_0$ and $\mu_1$ in $\bA$ satisfying the following pasting equality. 
\begin{tz}
\node[](1) {$A$};
\node[right of=1,rr](2) {$E$};
\node[right of=2](3) {$F$};
\node[below of=1](4) {$A$};
\node[right of=4](5) {$C$};
\node[right of=5](6) {$E$};
\node[right of=6](7) {$F$};
\draw[->] (1) to node[above,la] {$\theta_0$} node[below,la] {$\simeq$} (2);
\draw[->] (2) to node[above,la] {$h$} (3);
\draw[d,pro] (1) to node(a)[]{} (4);
\draw[d,pro] (2) to node(b)[]{} (6);
\draw[d,pro] (3) to (7);
\draw[->] (4) to node[above,la] {$\simeq$} node[below,la] {$\varphi_0$} (5);
\draw[->] (5) to node[above,la] {$\simeq$} node[below,la] {$\psi_0$} (6);
\draw[->] (6) to node[over,la] {$h$} (7);

\node[la] at ($(a)!0.45!(b)$) {$\mu_0$};
\node[la] at ($(a)!0.55!(b)$) {$\vcong$};
\node[la] at ($(2)!0.5!(7)$) {$e_h$};

\node[below of=4](1) {$A$};
\node[right of=1](2) {$C$};
\node[right of=2](8) {$D$};
\node[right of=8](9) {$F$};
\draw[d,pro] (5) to node(a)[]{} (2);
\draw[d,pro] (7) to node(z)[]{} (9);
\draw[d,pro] (4) to (1);
\draw[->] (1) to node[above,la] {$\simeq$} node[below,la] {$\varphi_0$} (2);
\draw[->] (8) to node[above,la] {$\simeq$} node[below,la] {$\psi_1$} (9);
\draw[->] (2) to node[over,la] {$g$} (8);

\node[la] at ($(a)!0.45!(z)$) {$\psi$};
\node[la] at ($(a)!0.55!(z)$) {$\vcong$};
\node[la] at ($(4)!0.5!(2)$) {$e_{\varphi_0}$};

\node[below of=1](3) {$A$};
\node[right of=3](4) {$B$};
\node[right of=4](5) {$D$};
\node[right of=5](6) {$F$};
\draw[d,pro] (1) to node(a)[]{} (3);
\draw[d,pro] (8) to node(b)[]{} (5);
\draw[d,pro] (9) to (6);
\draw[->] (4) to node[above,la] {$\simeq$} node[below,la] {$\varphi_1$} (5);
\draw[->] (5) to node[above,la] {$\simeq$} node[below,la] {$\psi_1$} (6);
\draw[->] (3) to node[below,la] {$f$} (4);

\node[la] at ($(a)!0.45!(b)$) {$\varphi$};
\node[la] at ($(a)!0.55!(b)$) {$\vcong$};
\node[la] at ($(8)!0.5!(6)$) {$e_{\psi_1}$};

\node[right of=7,xshift=.5cm,yshift=.75cm](1) {$A$};
\node[right of=1,rr](2) {$E$};
\node[right of=2](3) {$F$};
\node[below of=1](4) {$A$};
\node[la] at ($(z)!0.5!(4)$) {$=$};
\node[right of=4](5) {$B$};
\node[right of=5,rr](6) {$E$};
\draw[->] (1) to node[above,la] {$\theta_0$} node[below,la] {$\simeq$} (2);
\draw[->] (2) to node[above,la] {$h$} (3);
\draw[d,pro] (1) to node(a)[]{} (4);
\draw[d,pro] (3) to node(b)[]{} (6);
\draw[->] (5) to node[above,la] {$\theta_1$} node[below,la] {$\simeq$} (6);
\draw[->] (4) to node[over,la] {$f$} (5);

\node[la] at ($(a)!0.45!(b)$) {$\theta$};
\node[la] at ($(a)!0.55!(b)$) {$\vcong$};

\node[below of=4](1) {$A$};
\node[right of=1](2) {$B$};
\node[right of=2](3) {$D$};
\node[right of=3](9) {$F$};
\draw[d,pro] (4) to (1);
\draw[d,pro] (5) to node(a)[]{} (2);
\draw[d,pro] (6) to node(b)[]{} (9);
\draw[->] (2) to node[above,la] {$\simeq$} node[below,la] {$\varphi_1$} (3);
\draw[->] (3) to node[above,la] {$\simeq$} node[below,la] {$\psi_1$} (9);
\draw[->] (1) to node[below,la] {$f$} (2);

\node[la] at ($(a)!0.45!(b)$) {$\mu_1$};
\node[la] at ($(a)!0.55!(b)$) {$\vcong$};
\node[la] at ($(4)!0.5!(2)$) {$e_f$};
\end{tz}
\end{enumerate}
\end{descr}

We now compute the lower simplices of the space $(\ND\bA)_{0,1}$ -- the \emph{space of vertical morphisms}. As expected from the horizontal completeness condition, its $0$-simplices are given by the vertical morphisms, and its $1$-simplices by the weakly horizontally invertible squares. 

\begin{descr}[$m=0$, $k=1$] \label{desc:double01}
We describe the space $(\ND\bA)_{0,1}$. First note that $\VK{1}\otimes \OM{0}=\bbV[1]$ is the free double category on a vertical morphism. 
\begin{enumerate}\addtocounter{enumi}{-1}
    \item A $0$-simplex in $(\ND\bA)_{0,1}$ is a double functor $u\colon \bbV[1]\to \bA$, i.e., the data of a vertical morphism $u\colon A\arrowdot A'$ in $\bA$.
    \item A $1$-simplex in the space $(\ND\bA)_{0,1}$ is a horizontal pseudo-natural adjoint equivalence $\varphi\colon u\xRightarrow{\simeq} w$, i.e., the data of two horizontal adjoint equivalences $\varphi\colon A\xrightarrow{\simeq} C$ and $\varphi'\colon A'\xrightarrow{\simeq} C'$ together with a weakly horizontally invertible square in $\bA$
\begin{tz}
\node[](B') {$A'$};
\node[above of=B'](B) {$A$};
\node[right of=B](A) {$C$};
\node[below of=A](A') {$C'$};
\node at ($(A'.east)-(0,4pt)$) {.};
\draw[->] (B) to node[above,la] {$\varphi$} node[below,la]{$\simeq$} (A);
\draw[->,pro] (B) to node(a)[left,la] {$u$} (B');
\draw[->] (B') to node[below,la] {$\varphi'$} node[above,la]{$\simeq$} (A');
\draw[->,pro] (A) to node(b)[right,la] {$w$} (A');

\node[la] at ($(a)!0.4!(b)$) {$\widetilde{\varphi}$};
\node[la] at ($(a)!0.6!(b)$) {$\simeq$};
\end{tz}
    \item A $2$-simplex in $(\ND\bA)_{0,1}$ is an invertible modification $\mu\colon \theta\cong \psi\varphi$ between such horizontal pseudo-natural adjoint equivalences, i.e., the data of two vertically invertible squares $\mu$ and $\mu'$ in $\bA$ satisfying the following pasting equality.
\begin{tz}
\node[](1) {$A$};
\node[right of=1,rr](2) {$E$};
\node[below of=1](4) {$A$};
\node[right of=4](5) {$C$};
\node[right of=5](6) {$E$};
\draw[->] (1) to node[above,la] {$\theta$} node[below,la] {$\simeq$} (2);
\draw[d,pro] (1) to node(a)[]{} (4);
\draw[d,pro] (2) to node(b)[]{} (6);
\draw[->] (4) to node[above,la] {$\varphi$}node[below,la]{$\simeq$}   (5);
\draw[->] (5) to node[above,la] {$\psi$} node[below,la]{$\simeq$}  (6);

\node[la] at ($(a)!0.45!(b)$) {$\mu$};
\node[la] at ($(a)!0.55!(b)$) {$\vcong$};

\node[below of=4](7) {$A'$};
\node[right of=7](8) {$C'$};
\node[right of=8](9) {$E'$};
\draw[->,pro] (4) to node(a)[left,la] {$u$} (7);
\draw[->,pro] (5) to node(b)[left,la] {$w$} (8);
\draw[->,pro] (6) to node(c)[right,la] {$y$} (9);
\draw[->] (7) to node[above,la] {$\simeq$} node[below,la] {$\varphi'$} (8);
\draw[->] (8) to node[above,la] {$\simeq$} node[below,la] {$\psi'$} (9);

\node[la] at ($(a)!0.5!(b)$) {$\widetilde{\varphi}$};
\node[la] at ($(a)!0.7!(b)$) {$\simeq$};
\node[la] at ($(b)!0.4!(c)$) {$\widetilde{\psi}$};
\node[la] at ($(b)!0.6!(c)$) {$\simeq$};

\node[right of=2,xshift=.5cm](1) {$A$};
\node[la] at ($(9)!0.5!(1)$) {$=$};
\node[right of=1,rr](2) {$E$};
\node[below of=1](3) {$A'$};
\node[right of=3,rr](4) {$E'$};
\draw[->] (1) to node[above,la] {$\theta$} node[below,la]{$\simeq$} (2);
\draw[->,pro] (1) to node(a)[left,la] {$u$} (3);
\draw[->] (3) to node[below,la] {$\simeq$} node[above,la]{$\theta'$} (4);
\draw[->,pro] (2) to node(b)[right,la] {$y$} (4);

\node[la] at ($(a)!0.45!(b)$) {$\widetilde{\theta}$};
\node[la] at ($(a)!0.55!(b)$) {$\simeq$};

\node[below of=3](7) {$A'$};
\node[right of=7](8) {$C'$};
\node[right of=8](9) {$E'$};
\draw[->] (7) to node[above,la] {$\simeq$} node[below,la] {$\varphi'$} (8);
\draw[->] (8) to node[above,la] {$\simeq$} node[below,la] {$\psi'$} (9);
\draw[d,pro] (3) to node(a)[]{} (7);
\draw[d,pro] (4) to node(b)[]{} (9);

\node[la] at ($(a)!0.45!(b)$) {$\mu'$};
\node[la] at ($(a)!0.55!(b)$) {$\vcong$};
\end{tz}
\end{enumerate}
\end{descr}

Finally, we consider the \emph{space of squares} $(\ND\bA)_{1,1}$.

\begin{descr}[$m=1$, $k=1$] \label{desc:double11}
We describe the space $(\ND\bA)_{1,1}$. First note that $\VK{1}\otimes \OM{1}=\bbV[1]\times \bbH[1]$ is the free double category on a square.
\begin{enumerate}\addtocounter{enumi}{-1}
    \item A $0$-simplex in $(\ND\bA)_{1,1}$ is a double functor $\alpha\colon \bbV[1]\times \bbH[1]\to \bA$, i.e., the data of a square $\alpha$ in $\bA$
\begin{tz}
\node[](B') {$A'$};
\node[above of=B'](B) {$A$};
\node[right of=B](A) {$B$};
\node[below of=A](A') {$B'$};
\node at ($(A'.east)-(0,4pt)$) {.};
\draw[->] (B) to node[above,la] {$f$} (A);
\draw[->,pro] (B) to node[left,la] {$u$} (B');
\draw[->] (B') to node[below,la] {$f'$} (A');
\draw[->,pro] (A) to node[right,la] {$v$} (A');

\node[la] at ($(B)!0.5!(A')$) {$\alpha$};
\end{tz}
    \item A $1$-simplex in the space $(\ND\bA)_{1,1}$ is a horizontal pseudo-natural adjoint equivalence $\varphi\colon \alpha\xRightarrow{\simeq} \beta$, i.e., the data of four horizontal adjoint equivalences $\varphi_0$, $\varphi_1$, $\varphi'_0$, and $\varphi'_1$, two vertically invertible squares $\varphi$ and $\varphi'$, and two weakly horizontally invertible squares $\widetilde{\varphi_0}$ and $\widetilde{\varphi_1}$ in $\bA$ fitting in the following pasting equality. 
\begin{tz}
\node[](1) {$A$};
\node[right of=1](2) {$C$};
\node[right of=2](3) {$D$};
\node[below of=1](4) {$A$};
\node[right of=4](5) {$B$};
\node[right of=5](6) {$D$};
\draw[->] (1) to node[above,la] {$\varphi_0$} node[below,la] {$\simeq$} (2);
\draw[->] (2) to node[above,la] {$g$} (3);
\draw[d,pro] (1) to node(a)[]{} (4);
\draw[d,pro] (3) to node(b)[]{} (6);
\draw[->] (4) to node[over,la] {$f$}  (5);
\draw[->] (5) to node[above,la] {$\varphi_1$} node[below,la]{$\simeq$}  (6);

\node[la] at ($(a)!0.45!(b)$) {$\varphi$};
\node[la] at ($(a)!0.55!(b)$) {$\vcong$};

\node[below of=4](7) {$A'$};
\node[right of=7](8) {$B'$};
\node[right of=8](9) {$D'$};
\draw[->,pro] (4) to node[left,la] {$u$} (7);
\draw[->,pro] (5) to node(b)[left,la] {$v$} (8);
\draw[->,pro] (6) to node(c)[right,la] {$x$} (9);
\draw[->] (7) to node[below,la] {$f'$} (8);
\draw[->] (8) to node[above,la] {$\simeq$} node[below,la] {$\varphi'_1$} (9);

\node[la] at ($(4)!0.5!(8)$) {$\alpha$};
\node[la] at ($(b)!0.4!(c)$) {$\widetilde{\varphi_1}$};
\node[la] at ($(b)!0.6!(c)$) {$\simeq$};

\node[right of=3,xshift=.5cm](1) {$A$};
\node[la] at ($(9)!0.5!(1)$) {$=$};
\node[right of=1](2) {$C$};
\node[right of=2](3') {$D$};
\node[below of=1](3) {$A'$};
\node[right of=3](4) {$C'$};
\node[right of=4](5) {$D'$};
\draw[->] (1) to node[above,la] {$\varphi_0$} node[below,la] {$\simeq$} (2);
\draw[->] (2) to node[above,la] {$g$} (3');
\draw[->,pro] (1) to node(a)[left,la] {$u$} (3);
\draw[->,pro] (2) to node(b)[right,la] {$w$} (4);
\draw[->,pro] (3') to node[right,la] {$x$} (5);
\draw[->] (3) to node[below,la] {$\varphi'_0$} node[above,la]{$\simeq$} (4);
\draw[->] (4) to node[over,la] {$g'$} (5);

\node[la] at ($(a)!0.4!(b)$) {$\widetilde{\varphi_0}$};
\node[la] at ($(a)!0.6!(b)$) {$\simeq$};
\node[la] at ($(2)!0.5!(5)$) {$\beta$};

\node[below of=3](7) {$A'$};
\node[right of=7](8) {$B'$};
\node[right of=8](9) {$D'$};
\draw[->] (7) to node[below,la] {$f'$} (8);
\draw[->] (8) to node[above,la] {$\simeq$} node[below,la] {$\varphi'_1$} (9);
\draw[d,pro] (3) to node(a)[]{} (7);
\draw[d,pro] (5) to node(b)[]{} (9);

\node[la] at ($(a)!0.45!(b)$) {$\varphi'$};
\node[la] at ($(a)!0.55!(b)$) {$\vcong$};
\end{tz}
    \item A $2$-simplex in $(\ND\bA)_{1,1}$ is an invertible modification $\mu\colon \theta\cong \psi\varphi$ between such horizontal pseudo-natural adjoint equivalences, i.e., the data of four vertically invertible squares in $\bA$
\begin{tz}
\node[](A) {$A$};
\node[right of=A,rr](C) {$E$};
\node[below of=A](A') {$A$};
\node[right of=A'](B') {$C$};
\node[right of=B'](C') {$E$};
\draw[d,pro] (A) to node(a)[]{} (A'); 
\draw[d,pro] (C) to node(b)[]{} (C'); 
\draw[->] (A) to node[below,la] {$\simeq$} node[above,la] {$\theta_0$} (C);
\draw[->] (A') to node[below,la] {$\varphi_0$} node[above,la] {$\simeq$} (B');
\draw[->] (B') to node[below,la] {$\psi_0$} node[above,la] {$\simeq$} (C');

\node[la] at ($(a)!0.45!(b)$) {$\mu_0$};
\node[la] at ($(a)!0.55!(b)$) {$\vcong$};

\node[right of =C,xshift=1cm](A) {$B$};
\node[right of=A,rr](C) {$F$};
\node[below of=A](A') {$B$};
\node[right of=A'](B') {$D$};
\node[right of=B'](C') {$F$};
\draw[d,pro] (A) to node(a)[]{} (A'); 
\draw[d,pro] (C) to node(b)[]{} (C'); 
\draw[->] (A) to node[below,la] {$\simeq$} node[above,la] {$\theta_1$} (C);
\draw[->] (A') to node[below,la] {$\varphi_1$} node[above,la] {$\simeq$} (B');
\draw[->] (B') to node[below,la] {$\psi_1$} node[above,la] {$\simeq$} (C');

\node[la] at ($(a)!0.45!(b)$) {$\mu_1$};
\node[la] at ($(a)!0.55!(b)$) {$\vcong$};

\node[below of=A',yshift=.5cm](A) {$B'$};
\node[right of=A,rr](C) {$F'$};
\node[below of=A](A') {$B'$};
\node[right of=A'](B') {$D'$};
\node[right of=B'](C') {$F'$};
\draw[d,pro] (A) to node(a)[]{} (A'); 
\draw[d,pro] (C) to node(b)[]{} (C'); 
\draw[->] (A) to node[below,la] {$\simeq$} node[above,la] {$\theta'_1$} (C);
\draw[->] (A') to node[below,la] {$\varphi'_1$} node[above,la] {$\simeq$} (B');
\draw[->] (B') to node[below,la] {$\psi'_1$} node[above,la] {$\simeq$} (C');

\node[la] at ($(a)!0.45!(b)$) {$\mu'_1$};
\node[la] at ($(a)!0.55!(b)$) {$\vcong$};

\node[left of=A,xshift=-1cm](C) {$E'$};
\node[left of=C,xshift=-1.5cm](A) {$A'$};
\node[below of=A](A') {$A'$};
\node[right of=A'](B') {$C'$};
\node[right of=B'](C') {$E'$};
\draw[d,pro] (A) to node(a)[]{} (A'); 
\draw[d,pro] (C) to node(b)[]{} (C'); 
\draw[->] (A) to node[below,la] {$\simeq$} node[above,la] {$\theta'_0$} (C);
\draw[->] (A') to node[below,la] {$\varphi'_0$} node[above,la] {$\simeq$} (B');
\draw[->] (B') to node[below,la] {$\psi'_0$} node[above,la] {$\simeq$} (C');

\node[la] at ($(a)!0.45!(b)$) {$\mu'_0$};
\node[la] at ($(a)!0.55!(b)$) {$\vcong$};
\end{tz}
for $i=0,1$, such that
\begin{itemize}
    \item $(\mu_0,\mu_1)$ satisfies the pasting equality as in \cref{desc:double10} (2) with respect to $\varphi$, $\psi$, and $\theta$,
    \item $(\mu'_0,\mu'_1)$ satisfies the pasting equality as in \cref{desc:double10} (2) with respect to $\varphi'$, $\psi'$, and $\theta'$,
    \item $(\mu_0,\mu'_0)$ satisfies the pasting equality as in \cref{desc:double01} (2) with respect to $\widetilde{\varphi_0}$, $\widetilde{\psi_0}$, and $\widetilde{\theta_0}$,
    \item $(\mu_1,\mu'_1)$ satisfies the pasting equality as in \cref{desc:double01} (2) with respect to $\widetilde{\varphi_1}$, $\widetilde{\psi_1}$, and $\widetilde{\theta_1}$.
\end{itemize}
\end{enumerate}
\end{descr}

To further get intuition on higher simplex directions, we further describe the $0$- and $1$-simplices of the spaces $(\ND\bA)_{2,0}$ and $(\ND\bA)_{0,2}$. These should be thought of as the \emph{spaces of horizontal composites} and \emph{vertical composites}, respectively.

\begin{descr}[$m=2$, $k=0$]
We describe the space $(\ND\bA)_{2,0}$. First note that $\VK{0}\otimes \OM{2}=\bbH\OM{2}$ is the horizontal double category associated to $\OM{2}$.  
\begin{enumerate}\addtocounter{enumi}{-1}
\item A $0$-simplex in $(\ND\bA)_{2,0}$ is a double functor $\alpha\colon \bbH\OM{2}\to \bA$, i.e., the data of a vertically invertible square $\alpha$ in $\bA$
\begin{tz}
\node[](A) {$A$};
\node[right of=A,rr](C) {$C$};
\node[below of=A](A') {$A$};
\node[right of=A'](B') {$B$};
\node[right of=B'](C') {$C$};

\node at ($(C'.east)-(0,4pt)$) {.};
\draw[d,pro] (A) to node(a)[]{} (A'); 
\draw[d,pro] (C) to node(b)[]{} (C'); 
\draw[->] (A) to node[above,la] {$h$} (C);
\draw[->] (A') to node[below,la] {$f$} (B');
\draw[->] (B') to node[below,la] {$g$} (C');

\node[la] at ($(a)!0.45!(b)$) {$\alpha$};
\node[la] at ($(a)!0.55!(b)$) {$\vcong$};
\end{tz}
\item A $1$-simplex in the space $(\ND\bA)_{2,0}$ is a horizontal pseudo-natural adjoint equivalence $\varphi\colon \alpha\xRightarrow{\simeq} \alpha'$, i.e., the data of three horizontal adjoint equivalences $\varphi_0$, $\varphi_1$, and $\varphi_2$, and three vertically invertible squares $\varphi_f$, $\varphi_g$, and $\varphi_h$ in $\bA$ fitting in the following pasting equality.
\begin{tz}
\node[](1) {$A$};
\node[right of=1](2) {$A'$};
\node[right of=2,rr](3) {$C'$};
\node[below of=1](4) {$A$};
\node[right of=4](5) {$A'$};
\node[right of=5](6) {$B'$};
\node[right of=6](7) {$C'$};
\draw[->] (1) to node[above,la] {$\varphi_0$} node[below,la] {$\simeq$} (2);
\draw[->] (2) to node[above,la] {$h'$} (3);
\draw[d,pro] (1) to (4);
\draw[d,pro] (2) to node(a)[]{} (5);
\draw[d,pro] (3) to node(b)[]{} (7);
\draw[->] (4) to node[above,la] {$\simeq$} node[below,la] {$\varphi_0$} (5);
\draw[->] (5) to node[over,la] {$f'$} (6);
\draw[->] (6) to node[over,la] {$g'$} (7);

\node[la] at ($(a)!0.44!(b)$) {$\alpha'$};
\node[la] at ($(a)!0.56!(b)$) {$\vcong$};
\node[la] at ($(1)!0.5!(5)$) {$e_{\varphi_0}$};

\node[below of=4](1) {$A$};
\node[right of=1](2) {$B$};
\node[right of=2](8) {$B'$};
\node[right of=8](9) {$C'$};
\draw[d,pro] (6) to node(z)[]{} (8);
\draw[d,pro] (7) to node(y)[]{}(9);
\draw[d,pro] (4) to node(a)[]{} (1);
\draw[->] (1) to node[over,la] {$f$} (2);
\draw[->] (8) to node[over,la] {$g'$} (9);
\draw[->] (2) to node[above,la] {$\simeq$} node[below,la] {$\varphi_1$} (8);

\node[la] at ($(a)!0.44!(z)$) {$\varphi_f$};
\node[la] at ($(a)!0.56!(z)$) {$\vcong$};
\node[la] at ($(6)!0.5!(9)$) {$e_{g'}$};

\node[below of=1](3) {$A$};
\node[right of=3](4) {$B$};
\node[right of=4](5) {$C$};
\node[right of=5](6) {$C'$};
\draw[d,pro] (1) to (3);
\draw[d,pro] (2) to node(a)[]{} (4);
\draw[d,pro] (9) to node(b)[]{} (6);
\draw[->] (4) to node[below,la] {$g$} (5);
\draw[->] (5) to node[above,la] {$\simeq$} node[below,la] {$\varphi_2$} (6);
\draw[->] (3) to node[below,la] {$f$} (4);

\node[la] at ($(a)!0.44!(b)$) {$\varphi_g$};
\node[la] at ($(a)!0.56!(b)$) {$\vcong$};
\node[la] at ($(1)!0.5!(4)$) {$e_{f}$};

\node[right of=7,xshift=.5cm,yshift=.75cm](1) {$A$};
\node[right of=1](2) {$A'$};
\node[right of=2,rr](3) {$C'$};
\node[below of=1](4) {$A$};
\node[la] at ($(y)!0.5!(4)$) {$=$};
\node[right of=4,rr](5) {$C$};
\node[right of=5](6) {$C'$};
\draw[->] (1) to node[above,la] {$\varphi_0$} node[below,la] {$\simeq$} (2);
\draw[->] (2) to node[above,la] {$h'$} (3);
\draw[d,pro] (1) to node(a)[]{} (4);
\draw[d,pro] (3) to node(b)[]{} (6);
\draw[->] (5) to node[above,la] {$\varphi_2$} node[below,la] {$\simeq$} (6);
\draw[->] (4) to node[over,la] {$h$} (5);

\node[la] at ($(a)!0.45!(b)$) {$\varphi_h$};
\node[la] at ($(a)!0.55!(b)$) {$\vcong$};

\node[below of=4](1) {$A$};
\node[right of=1](2) {$B$};
\node[right of=2](3) {$C$};
\node[right of=3](9) {$C'$};
\draw[d,pro] (4) to node(a)[]{} (1);
\draw[d,pro] (5) to node(b)[]{} (3);
\draw[d,pro] (6) to (9);
\draw[->] (2) to node[below,la] {$g$} (3);
\draw[->] (3) to node[above,la] {$\simeq$} node[below,la] {$\varphi_2$} (9);
\draw[->] (1) to node[below,la] {$f$} (2);

\node[la] at ($(a)!0.44!(b)$) {$\alpha$};
\node[la] at ($(a)!0.56!(b)$) {$\vcong$};
\node[la] at ($(5)!0.5!(9)$) {$e_{\varphi_2}$};
\end{tz}
\end{enumerate}
\end{descr}

\begin{descr}[$m=0$, $k=2$] 
We describe the space $(\ND\bA)_{0,2}$. First note that $\VK{2}\otimes \OM{2}=\VK{2}$ is the vertical double category associated to $\OM{2}$. 
\begin{enumerate}\addtocounter{enumi}{-1}
\item A $0$-simplex in $(\ND\bA)_{0,2}$ is a double functor $\alpha\colon \VK{2}\to \bA$, i.e., the data of a horizontally invertible square $\alpha$ in $\bA$
\begin{tz}
    \node[](1) {$A$}; 
\node[right of=1](2) {$A$}; 
\node[below of=2](3) {$A'$}; 
\node[below of=3](4) {$A''$}; 
\node[left of=4](6) {$A''$}; 

\node at ($(4.east)-(0,4pt)$) {.};
\draw[pro,->] (1) to node(b)[left,la] {$u''$} (6);
\draw[pro,->] (2) to node[right,la] {$u$} (3);
\draw[pro,->] (3) to node[right,la] {$u'$} (4); 
\draw[d] (1) to (2); 
\draw[d] (6) to (4); 
\node[la] at ($(b)!0.45!(3)$) {$\alpha$};
\node[la] at ($(b)!0.65!(3)$) {$\cong$};
\end{tz}
\item A $1$-simplex in the space $(\ND\bA)_{0,2}$ is a horizontal pseudo-natural adjoint equivalence $\varphi\colon \alpha\xRightarrow{\simeq} \beta$, i.e., the data of three horizontal adjoint equivalences $\varphi$, $\varphi'$, and $\varphi''$, and three weakly horizontally invertible squares $\widetilde{\varphi}$, $\widetilde{\varphi}'$, and $\widetilde{\varphi}''$ fitting in the following pasting equality.
\begin{tz}
    \node[](1) {$A$}; 
\node[right of=1](2) {$A$}; 
\node[below of=2](3) {$A'$}; 
\node[below of=3](4) {$A''$}; 
\node[left of=4](6) {$A''$}; 
\draw[pro,->] (1) to node(b)[left,la] {$u''$} (6);
\draw[pro,->] (2) to node(a)[left,la] {$u$} (3);
\draw[pro,->] (3) to node(a')[left,la] {$u'$} (4); 
\draw[d] (1) to (2); 
\draw[d] (6) to (4); 
\node[la] at ($(b)!0.45!(3)$) {$\alpha$};
\node[la] at ($(b)!0.65!(3)$) {$\cong$};

\node[right of=2](2') {$C$}; 
\node[below of=2'](3') {$C'$}; 
\node[below of=3'](4') {$C''$}; 
\draw[->](2) to node[above,la]{$\varphi$} node[below,la]{$\simeq$} (2');
\draw[->](3) to node[below,la]{$\varphi'$} node[above,la]{$\simeq$} (3');
\draw[->](4) to node[below,la]{$\varphi''$} node[above,la]{$\simeq$} (4');
\draw[pro,->] (2') to node(b)[right,la] {$v$} (3');
\draw[pro,->] (3') to node(b')[right,la] {$v'$} (4');
\node[la] at ($(a)!0.4!(b)$) {$\widetilde{\varphi}$};
\node[la] at ($(a)!0.6!(b)$) {$\simeq$};
\node[la] at ($(a')!0.4!(b')$) {$\widetilde{\varphi}'$};
\node[la] at ($(a')!0.6!(b')$) {$\simeq$};

\node[right of=2',xshift=.5cm](1) {$A$}; 
\node[la] at ($(2')!0.5!(1)$) {$=$};
\node[right of=1](2) {$C$};
\node[below of=2,yshift=-1.5cm](4) {$C''$}; 
\node[left of=4](6) {$A''$}; 
\draw[pro,->] (1) to node(b)[left,la] {$u''$} (6);
\draw[pro,->] (2) to node(a)[right,la] {$v''$} (4); 
\draw[->](1) to node[above,la]{$\varphi$} node[below,la]{$\simeq$} (2);
\draw[->](6) to node[below,la]{$\varphi''$} node[above,la]{$\simeq$} (4);
\node[la] at ($(b)!0.4!(a)$) {$\widetilde{\varphi}''$};
\node[la] at ($(b)!0.6!(a)$) {$\simeq$};

\node[right of=2](2') {$C$}; 
\node[below of=2'](3') {$C'$}; 
\node[below of=3'](4') {$C''$};
\draw[pro,->] (2') to node[right,la] {$v$} (3');
\draw[pro,->] (3') to node[right,la] {$v'$} (4');
\draw[d] (2) to (2'); 
\draw[d] (4) to (4'); 
\node[la] at ($(a)!0.3!(3')$) {$\beta$};
\node[la] at ($(a)!0.6!(3')$) {$\cong$};
\end{tz}
\end{enumerate}
\end{descr}

\subsection{Nerve of a \texorpdfstring{$2$}{2}-category} \label{subsec:descr-NHsim}

By computing the nerve of a $2$-category, we expect to see the \emph{space of objects} at $(m,k)=(0,0)$, the \emph{space of morphisms} at $(m,k)=(1,0)$, and the \emph{space of $2$-morphisms} at $(m,k)=(1,1)$, while the space at $(m,k)=(0,1)$ should be weakly equivalent to the space of objects, since the first column of $2$-fold complete Segal space is essentially constant.

Let $\cA$ be a $2$-category. Recall that its nerve is given by the nerve of its associated double category $\bbHsim\cA$. We therefore translate \cref{desc:double00,desc:double10,desc:double01,desc:double11} to this setting. In particular, we first obtain the \emph{space of objects} $(\ND\bbHsim \cA)_{0,0}$, whose $0$-simplices are the objects, and whose $1$-simplices are the adjoint equivalences of $\cA$, as expected by the completeness condition. 
\begin{descr}[$m=0$, $k=0$] \label{desc:2cat00}
We describe the space $(\ND\bbHsim\cA)_{0,0}$. 
\begin{enumerate}\addtocounter{enumi}{-1}
    \item A $0$-simplex in $(\ND\bbHsim\cA)_{0,0}$ is the data of an object $A\in \cA$. 
    \item A $1$-simplex in $(\ND\bbHsim\cA)_{0,0}$ is the data of an adjoint equivalence $A\xrightarrow{\simeq} C$ in $\cA$. 
    \item A $2$-simplex in $(\ND\bbHsim\cA)_{0,0}$ is the data of an invertible $2$-morphism as in the following diagram. 
\begin{tz}
\node[](1) {$A$};
\node[below right of=1,xshift=.5cm](2) {$C$};
\node[above right of=2,xshift=.5cm](3) {$E$};
\draw[->] (1) to node[below,la,pos=0.4]{$\simeq$} (2);
\draw[->] (2) to node[below,la,pos=0.6]{$\simeq$} (3);
\draw[->] (1) to node[above,la]{$\simeq$} (3);
\coordinate(a) at ($(1)!0.5!(3)$);
\cell[la,right][n][0.37]{a}{2}{$\cong$};
\end{tz}
\end{enumerate}
\end{descr}

As for the \emph{space of morphisms} $(\ND\bbHsim\cA)_{1,0}$, we can see that the completeness condition is now satisfied for $(\ND\bbHsim\cA)_{1,-}$, since vertical morphisms are now adjoint equivalences in~$\cA$ and they therefore also appear in the horizontal direction. 

\begin{descr}[$m=1$, $k=0$] \label{desc:2cat10}
We describe the space $(\ND\bbHsim\cA)_{1,0}$. 
\begin{enumerate}\addtocounter{enumi}{-1}
    \item A $0$-simplex in $(\ND\bbHsim\cA)_{1,0}$ is the data of a morphism $f\colon A\to B$ in $\cA$. 
    \item A $1$-simplex in $(\ND\bbHsim\cA)_{1,0}$ is the data of two adjoint equivalences and an invertible $2$-morphism in $\cA$ as in the following diagram. 
\begin{tz}
    \node[](A) {$A$};
    \node[right of=A](B) {$C$};
    \node[below of=A](A') {$B$};
    \node[right of=A'](B') {$D$};
    \draw[->] (A) to node[above,la]{$\simeq$}  (B);
    \draw[->] (A') to node[below,la] {$\simeq$} (B');
    \draw[->] (A) to node[left,la]{$f$} (A');
    \draw[->] (B) to node[right,la]{$g$} (B');
    
    \cell[la,above,xshift=-.2cm]{B}{A'}{$\cong$};
\end{tz}
    \item A $2$-simplex in $(\ND\bbHsim\cA)_{1,0}$ is the data of two invertible $2$-morphisms filling the triangles of the following pasting equality. 
\begin{tz}
\node[](1) {$A$};
\node[below right of=1,xshift=.5cm](2) {$C$};
\node[above right of=2,xshift=.5cm](3) {$E$};
\draw[->] (1) to node[below,la,pos=0.4]{$\simeq$} (2);
\draw[->] (2) to node[below,la,pos=0.6]{$\simeq$} (3);
\draw[->] (1) to node[above,la]{$\simeq$} (3);
\coordinate(a) at ($(1)!0.5!(3)$);
\cell[la,right][n][0.37]{a}{2}{$\cong$};

\node[below of=1](4) {$B$};
\node[below of=2](5) {$D$};
\node[below of=3](6) {$F$};
\draw[->] (1) to node[left,la]{$f$} (4);
\draw[->] (2) to node[over,la]{$g$} (5);
\draw[->] (3) to node[right,la]{$h$} (6);
\draw[->] (4) to node[below,la,pos=0.4]{$\simeq$} (5);
\draw[->] (5) to node[below,la,pos=0.6]{$\simeq$} (6);

\cell[la,below,xshift=4pt]{2}{4}{$\cong$};
\cell[la,right,xshift=1pt]{3}{5}{$\cong$};

\node[right of=3,xshift=.5cm](1) {$A$};
\node[la] at ($(6)!0.5!(1)$) {$=$};

\node[below of=1](4) {$B$};
\node[below right of=4,xshift=.5cm](5) {$D$};
\node[above right of=5,xshift=.5cm](6) {$F$};
\node[above of=6](3) {$E$};
\draw[->] (1) to node[above,la]{$\simeq$} (3);
\draw[->] (4) to node[above,la]{$\simeq$} (6);
\coordinate(a) at ($(4)!0.5!(6)$);
\draw[->] (1) to node[left,la]{$f$} (4);
\draw[->] (3) to node[right,la]{$h$} (6);
\draw[->] (4) to node[below,la,pos=0.4]{$\simeq$} (5);
\draw[->] (5) to node[below,la,pos=0.6]{$\simeq$} (6);
\coordinate(b) at ($(3)-(.5cm,0)$);
\coordinate(c) at ($(4)+(.5cm,0)$);

\cell[la,right][n][0.37]{a}{5}{$\cong$};
\cell[la,above,xshift=-.2cm]{b}{c}{$\cong$};
\end{tz}
\end{enumerate}
\end{descr}

The space $(\ND\bbHsim\cA)_{0,1}$ is actually given by the \emph{space of adjoint equivalences}. Since the ``free-living adjoint equivalence'' is biequivalent to the point, this space can be interpreted as ``homotopically the same'' as the space of objects.

\begin{descr}[$m=0$, $k=1$] \label{desc:2cat01}
We describe the space $(\ND\bbHsim\cA)_{0,1}$. 
\begin{enumerate}\addtocounter{enumi}{-1}
    \item A $0$-simplex in $(\ND\bbHsim\cA)_{0,1}$ is the data of an adjoint equivalence $u\colon A\xrightarrow{\simeq} A'$ in $\cA$. 
    \item A $1$-simplex in $(\ND\bbHsim\cA)_{0,1}$ is the data of an invertible $2$-morphism as in the following diagram, by \cref{lem:weakinvinHsim}.
\begin{tz}
    \node[](A) {$A$};
    \node[right of=A](B) {$C$};
    \node[below of=A](A') {$A'$};
    \node[right of=A'](B') {$C'$};
    \draw[->] (A) to node[above,la]{$\simeq$}  (B);
    \draw[->] (A') to node[below,la] {$\simeq$} (B');
    \draw[->] (A) to node[left,la]{$u$} node[right,la]{$\lsimeq$} (A');
    \draw[->] (B) to node[right,la]{$w$} node[left,la]{$\rsimeq$} (B');
    
    \cell[la,above,xshift=-.2cm]{B}{A'}{$\cong$};
\end{tz}
    \item A $2$-simplex in $(\ND\bbHsim\cA)_{0,1}$ is the data of two invertible $2$-morphisms filling the triangles of the following pasting equality. 
\begin{tz}
\node[](1) {$A$};
\node[below right of=1,xshift=.5cm](2) {$C$};
\node[above right of=2,xshift=.5cm](3) {$E$};
\draw[->] (1) to node[below,la,pos=0.4]{$\simeq$} (2);
\draw[->] (2) to node[below,la,pos=0.6]{$\simeq$} (3);
\draw[->] (1) to node[above,la]{$\simeq$} (3);
\coordinate(a) at ($(1)!0.5!(3)$);
\cell[la,right][n][0.37]{a}{2}{$\cong$};

\node[below of=1](4) {$A'$};
\node[below of=2](5) {$C'$};
\node[below of=3](6) {$E'$};
\draw[->] (1) to node[left,la]{$u$} node[right,la]{$\lsimeq$} (4);
\draw[->] (2) to node[left,la]{$w$} node[right,la]{$\lsimeq$} (5);
\draw[->] (3) to node[right,la]{$y$} node[left,la]{$\rsimeq$} (6);
\draw[->] (4) to node[below,la,pos=0.4]{$\simeq$} (5);
\draw[->] (5) to node[below,la,pos=0.6]{$\simeq$} (6);

\cell[la,below,xshift=4pt]{2}{4}{$\cong$};
\cell[la,right,xshift=1pt]{3}{5}{$\cong$};

\node[right of=3,xshift=.5cm](1) {$A$};
\node[la] at ($(6)!0.5!(1)$) {$=$};

\node[below of=1](4) {$A'$};
\node[below right of=4,xshift=.5cm](5) {$C'$};
\node[above right of=5,xshift=.5cm](6) {$E'$};
\node[above of=6](3) {$E$};
\draw[->] (1) to node[above,la]{$\simeq$} (3);
\draw[->] (4) to node[above,la]{$\simeq$} (6);
\coordinate(a) at ($(4)!0.5!(6)$);
\draw[->] (1) to node[left,la]{$u$} node[right,la]{$\lsimeq$} (4);
\draw[->] (3) to node[right,la]{$y$} node[left,la]{$\rsimeq$} (6);
\draw[->] (4) to node[below,la,pos=0.4]{$\simeq$} (5);
\draw[->] (5) to node[below,la,pos=0.6]{$\simeq$} (6);
\coordinate(b) at ($(3)-(.5cm,0)$);
\coordinate(c) at ($(4)+(.5cm,0)$);

\cell[la,right][n][0.37]{a}{5}{$\cong$};
\cell[la,above,xshift=-.2cm]{b}{c}{$\cong$};
\end{tz}
\end{enumerate}
\end{descr}

Finally, we compute the \emph{space of $2$-morphisms} $(\ND\bbHsim\cA)_{1,1}$. Although its $0$-simplices are not precisely the $2$-morphisms of $\cA$, homotopically they give the right notion as the vertical morphisms $u$ and $v$ in the square below are adjoint equivalences. 

\begin{descr}
We describe the space $(\ND\bbHsim\cA)_{1,1}$. 
\begin{enumerate}\addtocounter{enumi}{-1}
    \item A $0$-simplex in $(\ND\bbHsim\cA)_{1,1}$ is the data of a $2$-morphism in $\cA$ as in the following diagram.
\begin{tz}
    \node[](A) {$A$};
    \node[right of=A](B) {$B$};
    \node[below of=A](A') {$A'$};
    \node[right of=A'](B') {$B'$};
    \draw[->] (A) to node[above,la]{$f$}  (B);
    \draw[->] (A') to node[below,la] {$f'$} (B');
    \draw[->] (A) to node[left,la]{$u$} node[right,la]{$\lsimeq$} (A');
    \draw[->] (B) to node[right,la]{$v$} node[left,la]{$\rsimeq$} (B');
    
    \cell[la,above,xshift=-.2cm]{B}{A'}{$\alpha$};
\end{tz}
    \item A $1$-simplex in $(\ND\bbHsim\cA)_{1,1}$ is the data of four adjoint equivalences and four invertible $2$-morphisms in $\cA$ as in the following diagram.
\begin{tz}
\node[](1) {$A$};
\node[below right of=1,xshift=.5cm](2) {$B$};
\node[above right of=1, xshift=.5cm](1') {$C$};
\node[above right of=2,xshift=.5cm](3) {$D$};
\draw[->] (1) to node[over,la]{$f$} (2);
\draw[->] (2) to node[above,la,pos=0.4]{$\simeq$} (3);
\draw[->] (1) to node[above,la,pos=0.4]{$\simeq$} (1');
\draw[->] (1') to node[above,la,pos=0.6]{$g$}(3);
\cell[la,right]{1'}{2}{$\cong$};

\node[below of=1](4) {$A'$};
\node[below of=2](5) {$B'$};
\node[below of=3](6) {$D'$};
\draw[->] (1) to node[left,la]{$u$} node[right,la]{$\lsimeq$} (4);
\draw[->] (2) to node[left,la]{$v$} node[right,la]{$\lsimeq$} (5);
\draw[->] (3) to node[right,la]{$x$} node[left,la]{$\rsimeq$} (6);
\draw[->] (4) to node[below,la,pos=0.4]{$f'$} (5);
\draw[->] (5) to node[below,la,pos=0.6]{$\simeq$} (6);

\cell[la,below,xshift=4pt]{2}{4}{$\alpha$};
\cell[la,right,xshift=1pt]{3}{5}{$\cong$};

\node[right of=3,xshift=.5cm](1) {$A$};
\node[above right of=1, xshift=.5cm](1') {$C$};
\node[la] at ($(6)!0.5!(1)$) {$=$};

\node[below of=1](4) {$A'$};
\node[above right of=4,xshift=.5cm](5) {$C'$};
\node[below right of=4,xshift=.5cm](4') {$B'$};
\node[below right of=5,xshift=.5cm](6) {$D'$};
\node[above of=6](3) {$D$};
\draw[->] (1) to node[above,la,pos=0.4]{$\simeq$} (1');
\draw[->] (1') to node[above,la,pos=0.6]{$g$}(3);
\coordinate(a) at ($(4)!0.5!(6)$);
\draw[->] (1) to node[left,la]{$u$} node[right,la]{$\lsimeq$} (4);
\draw[->] (1') to node[left,la]{$w$} node[right,la]{$\lsimeq$} (5);
\draw[->] (3) to node[right,la]{$x$} node[left,la]{$\rsimeq$} (6);
\draw[->] (4) to node[below,la,pos=0.6]{$\simeq$} (5);
\draw[->] (5) to node[over,la]{$g'$} (6);
\draw[->] (4) to node[below,la,pos=0.4]{$f'$} (4');
\draw[->] (4') to node[below,la,pos=0.6]{$\simeq$} (6);

\cell[la,right]{5}{4'}{$\cong$};

\cell[la,below,xshift=4pt]{3}{5}{$\beta$};
\cell[la,right,xshift=1pt]{1'}{4}{$\cong$};

\end{tz}
\item A $2$-simplex in $(\ND\bbHsim\cA)_{1,1}$ is the data of four invertible $2$-morphisms filling triangles satisfying relations as described in \cref{desc:2cat10} (2) and \cref{desc:2cat01}~(2).
\end{enumerate}
\end{descr}

\subsection{Nerve of a horizontal double category} \label{subsec:descr-NH}

Finally, we compute the nerve of a horizontal double category $\bbH\cA$ in lower dimensions, where $\cA$ is a $2$-category, in order to compare it with the nerve $\ND\bbHsim\cA$ described above. Since $\bbH\cA$ and $\bbHsim\cA$ have the same underlying horizontal $2$-category, namely $\cA$ itself, then the spaces $(\ND\bbH\cA)_{0,0}$ and $(\ND\bbH\cA)_{1,0}$ are equal to the spaces $(\ND\bbHsim\cA)_{0,0}$ and $(\ND\bbHsim\cA)_{1,0}$ and they can therefore be described as in \cref{desc:2cat00,desc:2cat10}, respectively. In particular, they are the desired \emph{space of objects} and \emph{space of morphisms}. 

We now turn our attention to the space $(\ND\bbH\cA)_{0,1}$. Unlike $(\ND\bbHsim\cA)_{0,1}$, this space has as $0$-simplices the objects of $\cA$. This prohibits a completeness condition in the vertical direction since equalities are not homotopically good enough. 

\begin{descr}[$m=0$, $k=1$] \label{desc:hor2cat01}
We describe the space $(\ND\bbH\cA)_{0,1}$. 
\begin{enumerate}\addtocounter{enumi}{-1}
    \item A $0$-simplex in $(\ND\bbH\cA)_{0,1}$ is the data of an object $A\in \cA$. 
    \item A $1$-simplex in $(\ND\bbH\cA)_{0,1}$ is the data of an invertible $2$-morphism as in the following diagram, by \cref{lem:whiiffvi}.
\begin{tz}
\node[](A) {$A$};
\node[right of=A,xshift=.5cm](B) {$C$};
\draw[->,bend left] (A.north east) to node(a)[above,la] {$\simeq$} (B.north west);
\draw[->,bend right] (A.south east) to node(b)[below,la] {$\simeq$} (B.south west);
\cell[la,right]{a}{b}{$\cong$};
\end{tz}
    \item A $2$-simplex in $(\ND\bbH\cA)_{0,1}$ is the data of two invertible $2$-morphisms filling the triangles of the following pasting equality. 
\begin{tz}
\node[](1) {$A$};
\node[below right of=1,xshift=.5cm](2) {$C$};
\node[above right of=2,xshift=.5cm](3) {$E$};
\draw[->] (1) to node(a'')[over,la]{$\simeq$} (2);
\draw[->] (2) to node(b)[over,la]{$\simeq$} (3);
\draw[->] (1) to node[above,la]{$\simeq$} (3);
\coordinate(a) at ($(1)!0.5!(3)$);
\cell[la,right][n][0.37]{a}{2}{$\cong$};

\draw[->,bend right=60] (1) to node(a')[below,la,pos=0.4]{$\simeq$} (2);
\draw[->,bend right=60] (2) to node(b')[below,la,pos=0.6]{$\simeq$} (3);

\cell[la,right,yshift=-4pt][n][0.44][0.18cm]{a''}{a'}{$\cong$};
\cell[la,above,xshift=8pt,yshift=-3pt][n][0.44][0.18cm]{b}{b'}{$\cong$};

\node[right of=3,xshift=.5cm](4) {$A$};
\node[la] at ($(3)!0.5!(4)-(0,.5cm)$) {$=$};
\node[below right of=4,xshift=.5cm](5) {$C$};
\node[above right of=5,xshift=.5cm](6) {$E$};
\draw[->,bend left=40] (4) to node(b)[above,la]{$\simeq$} (6);
\draw[->] (4) to node(b')[over,la]{$\simeq$} (6);
\draw[->] (4) to node[below,la,pos=0.4]{$\simeq$} (5);
\draw[->] (5) to node[below,la,pos=0.6]{$\simeq$} (6);

\cell[la,right][n][0.4]{b'}{5}{$\cong$};
\cell[la,right][n][0.6][0.18cm]{b}{b'}{$\cong$};
\end{tz}
\end{enumerate}
\end{descr}

Finally, we compute the \emph{space of $2$-morphisms} $(\ND\bbH\cA)_{1,1}$, which appears to have precisely the $2$-morphisms of $\cA$ as $0$-simplices. However, as explained above, this description is not homotopically well-behaved, since we would also need to consider adjoint equivalences in the vertical direction. 

\begin{descr}
We describe the space $(\ND\bbH\cA)_{1,1}$. 
\begin{enumerate}\addtocounter{enumi}{-1}
    \item A $0$-simplex in $(\ND\bbH\cA)_{1,1}$ is the data of a $2$-morphism in $\cA$
\begin{tz}
\node[](A) {$A$};
\node[right of=A,xshift=.5cm](B) {$B$};
\node at ($(B.east)-(0,4pt)$) {.};
\draw[->,bend left] (A.north east) to node(a)[above,la] {$f$} (B.north west);
\draw[->,bend right] (A.south east) to node(b)[below,la] {$f'$} (B.south west);
\cell[la,right]{a}{b}{$\alpha$};
\end{tz}
    \item A $1$-simplex in $(\ND\bbH\cA)_{1,1}$ is the data of four adjoint equivalences and four invertible $2$-morphisms in $\cA$ as in the following diagram.
\begin{tz}
\node[](1) {$A$};
\node[below right of=1,xshift=.5cm,yshift=.3cm](2) {$B$};
\node[above right of=1, xshift=.5cm,yshift=-.3cm](1') {$C$};
\node[right of=1,rr](3) {$D$};
\draw[->] (1) to node[above,la,pos=0.4]{$\simeq$} (1');
\draw[->] (1') to node[above,la,pos=0.6]{$g$}(3);
\cell[la,right]{1'}{2}{$\cong$};

\draw[->] (1) to node(a'')[over,la]{$f$} (2);
\draw[->] (2) to node(b)[over,la]{$\simeq$} (3);

\draw[->,bend right=70] (1) to node(a')[below,la,pos=0.42]{$f'$} (2);
\draw[->,bend right=70] (2) to node(b')[below,la,pos=0.62]{$\simeq$} (3);

\cell[la,right,yshift=-4pt][n][0.44][0.18cm]{a''}{a'}{$\alpha$};
\cell[la,above,xshift=8pt,yshift=-3pt][n][0.44][0.18cm]{b}{b'}{$\cong$};

\node[right of=3,xshift=.5cm](1) {$A$};
\node[la] at ($(3)!0.5!(1)$) {$=$};
\node[below right of=1,xshift=.5cm,,yshift=.3cm](2) {$B$};
\node[above right of=1, xshift=.5cm,,yshift=-.3cm](1') {$C$};
\node[right of=1,rr](3) {$D$};
\draw[->] (1) to node(a'')[over,la]{$\simeq$} (1');
\draw[->] (1') to node(b)[over,la]{$g'$}(3);
\cell[la,right]{1'}{2}{$\cong$};

\draw[->] (1) to node[below,la,pos=0.4]{$f'$} (2);
\draw[->] (2) to node[below,la,pos=0.6]{$\simeq$} (3);

\draw[->,bend left=70] (1) to node(a')[above,la,pos=0.4]{$\simeq$} (1');
\draw[->,bend left=70] (1') to node(b')[above,la,pos=0.6]{$g$} (3);

\cell[la,above,xshift=8pt,yshift=-3pt][n][0.55][0.18cm]{a'}{a''}{$\cong$};
\cell[la,right,yshift=-4pt][n][0.55][0.18cm]{b'}{b}{$\beta$};
\end{tz}
\item A $2$-simplex in $(\ND\bbH\cA)_{1,1}$ is the data of four invertible $2$-morphisms filling triangles satisfying relations as described in \cref{desc:2cat10} (2) and \cref{desc:hor2cat01} (2).
\end{enumerate}
\end{descr}

\bibliographystyle{plain}
\bibliography{Reference}

\begin{thebibliography}{10}

\bibitem{Bar}
Clark Barwick.
\newblock {\em $(\infty, n)$-{C}at as a closed model category}.
\newblock ProQuest LLC, Ann Arbor, MI, 2005.
\newblock Thesis (Ph.D.)--University of Pennsylvania.

\bibitem{BergnerRezk}
Julia~E. Bergner and Charles Rezk.
\newblock Reedy categories and the {$\varTheta$}-construction.
\newblock {\em Math. Z.}, 274(1-2):499--514, 2013.

\bibitem{BR2}
Julia~E. Bergner and Charles Rezk.
\newblock Comparison of models for {$(\infty, n)$}-categories, {II}.
\newblock {\em J. Topol.}, 13(4):1554--1581, 2020.

\bibitem{Bohm}
Gabriella B\"ohm.
\newblock The {G}ray monoidal product of double categories.
\newblock {\em Appl. Categ. Structures}, 2019.

\bibitem{CampMos}
Alexander Campbell.
\newblock The folk model structure for double categories.
\newblock Seminar talk,
  \url{http://web.science.mq.edu.au/groups/coact/seminar/cgi-bin/abstract.cgi?talkid=1616}.

\bibitem{Camp}
Alexander Campbell.
\newblock A homotopy coherent cellular nerve for bicategories.
\newblock {\em Adv. Math.}, 368:107138, 2020.

\bibitem{Cisinski}
Denis-Charles Cisinski.
\newblock {\em Higher categories and homotopical algebra}, volume 180 of {\em
  Cambridge Studies in Advanced Mathematics}.
\newblock Cambridge University Press, Cambridge, 2019.

\bibitem{clingmanMoser}
tslil clingman and Lyne Moser.
\newblock 2-limits and 2-terminal objects are too different.
\newblock {\em Appl. Categ. Structures}, 30(6):1283--1304, 2022.

\bibitem{cM2}
tslil clingman and Lyne Moser.
\newblock Bi-initial objects and bi-representations are not so different.
\newblock {\em Cah. Topol. G\'{e}om. Diff\'{e}r. Cat\'{e}g.}, 63(3):259--330,
  2022.

\bibitem{FPP}
Thomas~M. Fiore, Simona Paoli, and Dorette Pronk.
\newblock Model structures on the category of small double categories.
\newblock {\em Algebr. Geom. Topol.}, 8(4):1855--1959, 2008.

\bibitem{GHL}
Andrea Gagna, Yonatan Harpaz, and Edoardo Lanari.
\newblock On the equivalence of all models for {$(\infty,2)$}-categories.
\newblock {\em J. Lond. Math. Soc. (2)}, 106(3):1920--1982, 2022.

\bibitem{Grandis}
Marco Grandis.
\newblock {\em Higher dimensional categories}.
\newblock World Scientific Publishing Co. Pte. Ltd., Hackensack, NJ, 2020.
\newblock From double to multiple categories.

\bibitem{GraPar19}
Marco Grandis and Robert Par\'{e}.
\newblock Persistent double limits.
\newblock {\em Cah. Topol. G\'{e}om. Diff\'{e}r. Cat\'{e}g.}, 60(3):255--297,
  2019.

\bibitem{GrandisPare}
Marco Grandis and Robert Paré.
\newblock Limits in double categories.
\newblock {\em Cahiers Topologie G\'{e}om. Diff\'{e}rentielle Cat\'{e}g.},
  40(3):162--220, 1999.

\bibitem{Gray}
John~W. Gray.
\newblock {\em Formal category theory: adjointness for {$2$}-categories}.
\newblock Lecture Notes in Mathematics, Vol. 391. Springer-Verlag, Berlin-New
  York, 1974.

\bibitem{HaugsengThesis}
Rune Haugseng.
\newblock {\em Weakly {E}nriched {H}igher {C}ategories}.
\newblock ProQuest LLC, Ann Arbor, MI, 2013.
\newblock Thesis (Ph.D.)--Massachusetts Institute of Technology.

\bibitem{Hirschhorn}
Philip~S. Hirschhorn.
\newblock {\em Model categories and their localizations}, volume~99 of {\em
  Mathematical Surveys and Monographs}.
\newblock American Mathematical Society, Providence, RI, 2003.

\bibitem{Hovey}
Mark Hovey.
\newblock {\em Model categories}, volume~63 of {\em Mathematical Surveys and
  Monographs}.
\newblock American Mathematical Society, Providence, RI, 1999.

\bibitem{JohYau}
Niles Johnson and Donald Yau.
\newblock {\em 2-dimensional categories}.
\newblock Oxford University Press, Oxford, 2021.

\bibitem{Lack2Cat}
Stephen Lack.
\newblock A {Q}uillen model structure for 2-categories.
\newblock {\em $K$-Theory}, 26(2):171--205, 2002.

\bibitem{LackBicat}
Stephen Lack.
\newblock A {Q}uillen model structure for bicategories.
\newblock {\em $K$-Theory}, 33(3):185--197, 2004.

\bibitem{MOR}
Lyne Moser, Viktoriya Ozornova, and Martina Rovelli.
\newblock Model independence of $(\infty,2)$-categorical nerves.
\newblock \href{https://arxiv.org/abs/2206.00660}{arXiv:2206.00660}, 2022.

\bibitem{MRR3}
Lyne Moser, Nima Rasekh, and Martina Rovelli.
\newblock $(\infty,n)$-limits {I}: Definition and first consistency results.
\newblock \href{https://arxiv.org/abs/2312.11101}{arXiv:2312.11101}, 2023.

\bibitem{MSVfirst}
Lyne Moser, Maru Sarazola, and Paula Verdugo.
\newblock A 2{C}at-inspired model structure for double categories.
\newblock {\em Cah. Topol. G\'{e}om. Diff\'{e}r. Cat\'{e}g.}, 63(2):184--236,
  2022.

\bibitem{MSVsecond}
Lyne Moser, Maru Sarazola, and Paula Verdugo.
\newblock A model structure for weakly horizontally invariant double
  categories.
\newblock {\em Algebr. Geom. Topol.}, 23(4):1725--1786, 2023.

\bibitem{OR2}
Viktoriya Ozornova and Martina Rovelli.
\newblock Nerves of 2-categories and 2-categorification of
  {$(\infty,2)$}-categories.
\newblock {\em Adv. Math.}, 391:Paper No. 107948, 39, 2021.

\bibitem{Qui}
Daniel~G. Quillen.
\newblock {\em Homotopical algebra}.
\newblock Lecture Notes in Mathematics, No. 43. Springer-Verlag, Berlin-New
  York, 1967.

\bibitem{Rezk}
Charles Rezk.
\newblock A model for the homotopy theory of homotopy theory.
\newblock {\em Trans. Amer. Math. Soc.}, 353(3):973--1007, 2001.

\bibitem{RiehlVerity}
Emily Riehl and Dominic Verity.
\newblock {\em Elements of {$\infty$}-category theory}, volume 194 of {\em
  Cambridge Studies in Advanced Mathematics}.
\newblock Cambridge University Press, Cambridge, 2022.

\bibitem{Street}
Ross Street.
\newblock The algebra of oriented simplexes.
\newblock {\em J. Pure Appl. Algebra}, 49(3):283--335, 1987.

\end{thebibliography}
\end{document}